\documentclass{article}
\usepackage{amsthm, amscd, amsmath, amsfonts, mathrsfs, bm}
\newtheorem{theo}{Theorem}[section]
\newtheorem{proposition}[theo]{Proposition}
\newtheorem{lemma}[theo]{Lemma}
\newtheorem{corollary}[theo]{Corollary}
\newtheorem{rem}[theo]{Remark}
\newtheorem{definition}[theo]{Definition}

\input xy
\xyoption{all}

\begin{document}

\title{On a Noncommutative Deformation of Holomorphic Line Bundles on Complex Tori and the SYZ Transform} 

\author{Kazushi Kobayashi\footnote{Department of Mathematics, Faculty of Education, University of Teacher Education Fukuoka, 1-1 Akamabunkyo-machi, Munakata, Fukuoka, 811-4192, Japan. E-mail: kobayashi-k@fukuoka-edu.ac.jp. 2020 Mathematics Subject Classification: 14F08, 14J33, 58B34 (primary), 53D37, 53C08 (secondary). Keywords: torus, homological mirror symmetry, SYZ transform, noncommutative geometry, gerbe.}}

\date{}

\maketitle

\begin{abstract}
By regarding a given $n$-dimensional complex torus $X^n$ as the trivial torus fibration $X^n \to \mathbb{R}^n/\mathbb{Z}^n$, we can obtain a mirror dual complexified symplectic torus $\check{X}^n$ based on the SYZ construction. In the middle 2000s, as a part of the study on noncommutative deformations of $X^n$, Kajiura examined the noncommutative complex torus $X_{\theta}^n$ obtained via the (real) nonformal deformation quantization of $X^n \to \mathbb{R}^n/\mathbb{Z}^n$ by a Poisson bivector $\theta$ defined along the fibers. In particular, he constructed the noncommutative deformations $L_{\theta} \to X_{\theta}^n$ of holomorphic line bundles on $X^n$ and a curved dg-category consisting of them. On the other hand, associated to this noncommutative deformation, we can construct a non-trivial deformation of the trivial holomorphic line bundle on $X^n$ by twisting it with a suitable isomorphism. In this paper, from this point of view, we extend the construction of $L_{\theta}$ to the more general setting. Moreover, we also consider objects defined on a mirror partner of $X_{\theta}^n$ which are mirror dual to such extended noncommutative objects.
\end{abstract}

\tableofcontents

\section{Introduction}
\subsection{Background}
\subsubsection{The homological mirror symmetry and the SYZ construction}
Let $X^n$ be an $n$-dimensional complex torus and $\check{X}^n$ a mirror partner of $X^n$. The homological mirror symmetry \cite{Kon} predicts that there exists an equivalence between the bounded derived category $D^b(Coh(X^n))$ of coherent sheaves over $X^n$ and the enhanced triangulated category $Tr(Fuk(\check{X}^n))$ of the Fukaya category $Fuk(\check{X}^n)$ over $\check{X}^n$ \cite{Fukaya category} in the sense of the Bondal-Kapranov-Kontsevich construction \cite{bondal, Kon}:
\begin{equation}
D^b(Coh(X^n))\cong Tr(Fuk(\check{X}^n)). \label{hms}
\end{equation}
So far, various studies on the homological mirror symmetry for tori have been conducted (\cite{dg, Fuk, kap1, kap2, abouzaid}, etc.) starting with the case of elliptic curves \cite{elliptic}. On the other hand, considering a noncommutative analogue of the statement (\ref{hms}) is also an interesting problem as discussed in \cite{hms-nc-two-tori, kim, kajiura} for instance. Our main object of interest in this paper is the noncommutative deformation $X_{\theta}^n$ of $X^n$ associated to the (real) nonformal deformation quantization of $X^n$ by a Poisson bivector $\Pi_{\theta}$ which corresponds to the noncommutative deformation of type $\theta_3$ in \cite{nc}. Of course, $\theta$ in this context is a noncommutative parameter. The purpose of this paper is to discuss a noncommutative analogue of the statement (\ref{hms}) between $X_{\theta}^n$ and its mirror partner $\check{X}_{\theta}^n$ on the object level.

It is valid to consider the SYZ construction \cite{SYZ} when we discuss the homological mirror symmetry. Below, we explain why this approach is valid briefly. 

The SYZ construction conjectured by Strominger-Yau-Zaslow states that each mirror pair $(M, \check{M})$ of Calabi-Yau manifolds is realized as the special Lagrangian torus fibrations $\pi : M \to B$, $\check{\pi} : \check{M} \to B$ on the same base space $B$. Moreover, in this construction, the fibers $\pi^{-1}(b)$ and $\check{\pi}^{-1}(b)$ on each point $b\in B$ are related by the T-duality, i.e., the torus $\pi^{-1}(b)$ is the dual of the torus $\check{\pi}^{-1}(b)$. For example, we can regard the mirror pair $(X^n, \check{X}^n)$ as the trivial torus fibrations $X^n \to \mathbb{R}^n/\mathbb{Z}^n$, $\check{X}^n \to \mathbb{R}^n/\mathbb{Z}^n$ in this context. Here, let us recall that the Fourier-Mukai transform proposed in \cite{fm} gives an equivalence 
\begin{equation}
D^b(Coh(X^n))\cong D^b(Coh(\hat{X}^n)), \label{mukai}
\end{equation}
where $\hat{X}^n$ is the dual of a given abelian variety $X^n$. Similarly, the T-duality causes the so-called SYZ transform which is an analogue of the Fourier-Mukai transform along the fibers of $\pi : M \to B$, $\check{\pi} : \check{M} \to B$ (see \cite{leung, A-P}). By the SYZ transform, it becomes easier to describe the mirror of each object of the Fukaya category explicitly, and vice versa.

On the other hand, it is expected that $D^b(Coh(X^n))$ is generated by a dg-category $DG_{X^n}$ consisting of holomorphic line bundles on $X^n$:
\begin{equation*}
Tr(DG_{X^n})\cong D^b(Coh(X^n)).
\end{equation*}
At least, it is known that it split generates $D^b(Coh(X^n))$ when $X^n$ is an abelian variety (see \cite{orlov, abouzaid}).

Summarizing the above discussions, as an example of an idea in order to prove the homological mirror symmetry (\ref{hms}) for $(X^n, \check{X}^n)$ based on the SYZ construction, we obtain the following (see \cite{Fuk, kajiura} for instance). Firstly, seek mirrors of holomorphic line bundles in $DG_{X^n}$ by considering the SYZ transform, etc. Secondly, focusing on $DG_{X^n}$ and the full subcategory $Fuk_{\rm sub}(\check{X}^n)$ of $Fuk(\check{X}^n)$ consisting of objects obtained in the first step, verify that $DG_{X^n}$ and $Fuk_{\rm sub}(\check{X}^n)$ actually generate $D^b(Coh(X^n))$ and $Tr(Fuk(\check{X}^n))$, respectively:
\begin{align}
&Tr(DG_{X^n})\cong D^b(Coh(X^n)), \label{hms_1} \\
&Tr(Fuk_{\rm sub}(\check{X}^n))\cong Tr(Fuk(\check{X}^n)). \label{hms_2}
\end{align}
In this context, we should take these generators $DG_{X^n}$ and $Fuk_{\rm sub}(\check{X}^n)$ as small as possible. Finally, prove that there exists an equivalence
\begin{equation}
DG_{X^n}\cong Fuk_{\rm sub}(\check{X}^n) \label{hms_3}
\end{equation}
as $A_{\infty}$-categories. In particular, the homological mirror symmetry (\ref{hms}) for $(X^n, \check{X}^n)$ follows from the statements (\ref{hms_1}), (\ref{hms_2}), (\ref{hms_3}) automatically.

\subsubsection{A problem on noncommutative deformations of $X^n$}
In this paper, we also discuss the following problem on noncommutative deformations of $X^n$. In general, associated to the noncommutative deformation of $X^n$ by $\Pi_{\theta}$ in the above sense, we can construct a non-trivial deformation of the trivial holomorphic line bundle $\mathcal{O}_{X^n} \to X^n$ by twisting it with a suitable isomorphism. Actually, let us assume that the trivial transition function $1$ of $\mathcal{O}_{X^n}$ changes to $\varphi(x_1,\ldots,x_n,y_1,\ldots,y_n)$ by an isomorphism, where $\varphi$ is not a constant function, and $(x_1,\ldots,x_n,y_1,\ldots,y_n)\in \mathbb{R}^{2n}$ denotes the (real) local coordinate system of $X^n$. Then, although $1$ is not affected by the Moyal star product associated to $\Pi_{\theta}$, $\varphi(x_1,\ldots,x_n,y_1,\ldots,y_n)$ is affected by the Moyal star product non-trivially. This fact causes ambiguity of the noncommutative deformation of objects of $DG_{X^n}$ explained below. Let $E$ be an object of $DG_{X^n}$, i.e., $E \to X^n$ is a holomorphic line bundle. We denote the noncommutative deformation of $E$ associated to the deformation from $X^n$ to $X_{\theta}^n$ by $E_{\theta}$ (we use this notation including $\theta$ for other holomorphic line bundles on $X^n$). Then, for a certain holomorphic line bundle $L \to X^n$ which is isomorphic to $\mathcal{O}_{X^n}$, even though there always exists an isomorphism $E\cong E\otimes L$, $E_{\theta}$ and $(E\otimes L)_{\theta}$ are not isomorphic to each other. 

\subsection{Motivation}
One of classical examples of deformations of $X^n$ is a deformation of the complex structure of $X^n$. As other more complicated examples, we can also consider noncommutative and gerby deformations of $X^n$. Since we focus on the noncommutative deformation of $X^n$ associated to a (real) nonformal deformation quantization of $X^n$, let us consider a Poisson bivector which is expressed locally as
\begin{equation}
\sum_{i, j=1}^n \left( \frac{1}{2}(\theta_1)_{ij}\frac{\partial}{\partial x_i}\wedge \frac{\partial}{\partial x_j}+(\theta_2)_{ij}\frac{\partial}{\partial x_i}\wedge \frac{\partial}{\partial y_j}+\frac{1}{2}(\theta_3)_{ij}\frac{\partial}{\partial y_i}\wedge \frac{\partial}{\partial y_j} \right). \label{P_bivector_intro}
\end{equation}  
Here, $\theta_1=((\theta_1)_{ij})$, $\theta_3=((\theta_3)_{ij})$ are real alternating matrices of order $n$, $\theta_2=((\theta_2)_{ij})$ is a real matrix of order $n$, and $(x, y)\in \mathbb{R}^{2n}$ denotes the coordinate system defined locally, where
\begin{equation*}
x:=\left( \begin{array}{ccc} x_1 \\ \vdots \\ x_n \end{array} \right), \ \ \ y:=\left( \begin{array}{ccc} y_1 \\ \vdots \\ y_n \end{array} \right)\in \mathbb{R}^n.
\end{equation*}
Moreover, in fact, note that the Poisson bivector (\ref{P_bivector_intro}) is defined globally. For simplicity, for each $i=1$, $2$, $3$, we call the noncommutative deformation of $X^n$ associated to the (real) nonformal deformation quantization of $X^n$ by the Poisson bivector (\ref{P_bivector_intro}) in the case that the parameters except for $\theta_i$ are trivial the noncommutative deformation of type $\theta_i$. 

We now explain some prior works. Historically, first, the noncommutative deformation of type $\theta_2$ is studied in the case of elliptic curves in \cite{p-s-nc}. Note that the noncommutative deformations of type $\theta_1$ and type $\theta_3$ are trivial on elliptic curves. Afterwards, in \cite{hms-nc-two-tori}, the homological mirror symmetry for noncommutative two-tori (the noncommutative deformation of type $\theta_2$) is discussed. As a generalization of such works to the higher dimensional cases, in \cite{nc}, for each $i=1$, $2$, $3$, the deformation of $DG_{X^n}$ associated to the noncommutative deformation of type $\theta_i$ is constructed as a curved dg-category (see also \cite{star}). The (non-trivial) curvature in this context essentially comes from the fact that the holomorphicity of line bundles in $DG_{X^n}$ does not necessarily preserved under those noncommutative deformations of $X^n$. In particular, including the higher dimensional cases, the description which is mirror dual to the noncommutative deformation of type $\theta_2$ is studied in \cite{kajiura}\footnote{Precisely speaking, in \cite{kajiura}, based on the SYZ construction, general mirror pairs of Calabi-Yau manifolds with no singular fibers are treated.}. On the other hand, at the moment, the descriptions which are mirror dual to the noncommutative deformations of type $\theta_1$ and type $\theta_3$ are not investigated.

Concerning these prior works, in this paper, we focus on the noncommutative deformation of type $\theta_3$, i.e., consider the Poisson bivector
\begin{equation*}
\Pi_{\theta}:=\frac{1}{2}\sum_{i, j=1}^n \theta_{ij}\frac{\partial}{\partial y_i}\wedge \frac{\partial}{\partial y_j} 
\end{equation*}
by using a fixed real alternating matrix $\theta=(\theta_{ij})$ of order $n$ (we will comment on the noncommutative deformation of type $\theta_1$ a little bit at the last of subsection 1.3). We further denote the pair consisting of $X^n$ and the Poisson bivector $\Pi_{\theta}$ by $X_{\theta}^n$. Of corse, this is the precise definition of $X_{\theta}^n$ which is mentioned in subsection 1.1.

Here, let us explicitly explain the main purposes of this paper once again. Firstly, by focusing on the problem associated to ambiguity explained in subsection 1.1, to extend the noncommutative deformation of holomorphic line bundles on $X^n$ which is constructed in \cite{nc} to the more general setting. Secondly, to specify the moduli spaces of such extended noncommutative objects and their mirrors (Theorem \ref{main_theorem_1}, Theorem \ref{main_theorem_2}). In particular, we verify that the moduli space in the symplectic geometry side is naturally identified with the moduli space in the complex geometry side in Theorem \ref{main_theorem_3}. We can also regard this result as a generalization of the SYZ transform between $Fuk_{\rm sub}(\check{X}^n)$ consisting of affine Lagrangian submanifolds in $\check{X}^n$ with unitary local systems and $DG_{X^n}$ to the setting which includes the noncommutative paramater $\theta$. 
 
\subsection{Other related works}
Although our main object of interest in this paper is a noncommutative deformation of $X^n$, in general, noncommutative deformations of $X^n$ are closely related to gerby deformations of $X^n$. Frankly speaking, it is expected that such correspondences are described as the Fourier-Mukai transforms. In particular, this is also a generalization of the equivalence (\ref{mukai}) to the setting which includes a deformation parameter. For example, in \cite{nc-fm}, by focusing on the noncommutative deformations associated to the formal deformation quantizations, the Fourier-Mukai partners of noncommutative complex tori are investigated. On the other hand, in light of \cite{nc-fm}, the Fourier-Mukai partners of noncommutative complex tori for nonformal parameters are discussed in \cite{okuda} (see also \cite{block-1, block-2}). 

Here, in order to comment on the relation between this paper and our previous paper \cite{b-field}, etc., let us consider a B-field which is expressed locally as
\begin{equation}
\sum_{i, j=1}^n \left( \frac{1}{2} (\tau_1)_{ij} dx_i \wedge dx_j +(\tau_2)_{ij} dx_i \wedge dy_j +\frac{1}{2} (\tau_3)_{ij} dy_i \wedge dy_j \right), \label{b_field_intro} 
\end{equation}
where $\tau_1=((\tau_1)_{ij})$, $\tau_3=((\tau_3)_{ij})$ are real alternating matrices of order $n$, and $\tau_2=((\tau_2)_{ij})$ is a real matrix of order $n$. Note that the B-field (\ref{b_field_intro}) is actually defined globally. Similarly as in the case of the Poisson bivector $\Pi_{\theta}$, for each $i=1$, $2$, $3$, we call the gerby deformation of $X^n$ associated to the B-field (\ref{b_field_intro}) in the case that the parameters except for $\tau_i$ are trivial the gerby deformation of type $\tau_i$. For example, the gerby deformations treated in our prior papers \cite{b-field} and \cite{gerby} are the gerby deformations of type $\tau_1$ and type $\tau_2$, respectively. In particular, we can expect that the Fourier-Mukai partner of the noncommutative complex torus $X_{\theta}^n$ (the noncommutative deformation of type $\theta_3$) is the gerby deformed complex torus treated in our previous paper \cite{b-field} (the gerby deformation of type $\tau_1$). 

However, during this study period, similar errors on the construction of such gerby deformations (the holomorphicity of deformed objects) are found in both the articles \cite{b-field, gerby} even though they are already published. In this paper, we also explain how to improve those errors (see section 6).

Also, the descriptions which are mirror dual to the gerby deformations of type $\tau_1$ and type $\tau_2$ are already studied in \cite{b-field} and \cite{gerby}, respectively. On the other hand, although it is not difficult to construct the deformation of holomorphic line bundles on $X^n$ associated to the gerby deformation of type $\tau_3$, anyway, at the moment, the counterpart in the symplectic geometry side is not discussed. Hence, including the noncommutative deformations of type $\theta_1$, type $\theta_2$, and type $\theta_3$, the cases that the understanding for the counterparts in the symplectic geometry side is insufficient are the noncommutative deformation of type $\theta_1$ and the gerby deformation of type $\tau_3$. In fact, in order to understand them, we need to consider the deformation of the Poisson structure associated to the symplectic structure of $\check{X}^n$. In general, this problem is closely related to a deformation quantization of $\check{X}^n$, i.e., a certain kind of a deformation quantization of objects of $Fuk(\check{X}^n)$. This will be discussed in \cite{beta} separately.

\subsection{Main results and the plan of this paper}
As an object of $Fuk(\check{X}^n)$, let us take a pair $(L,\mathcal{L})$ of an affine Lagrangian submanifold $L$ in $\check{X}^n$ and a unitary local system $\mathcal{L}$ along it. By the SYZ transform, we can obtain a holomorphic line bundle $E(L,\mathcal{L}) \to X^n$, and these $E(L,\mathcal{L})$ forms a dg-category $DG_{X^n}$ which is mentioned in subsection 1.1. In this paper, focusing on the problem caused by ambiguity explained in subsection 1.1, we first construct the noncommutative deformation $E(L,\mathcal{L})_{\theta}$ of each $E(L,\mathcal{L})$ associated to the deformation from $X^n$ to $X_{\theta}^n$. Then, although we can expect that the moduli space $\mathcal{M}_{\theta}$ of $E(L,\mathcal{L})_{\theta}$ is homeomorphic to a real torus, in this paper, we specify such a real torus explicitly. This result is given in Theorem \ref{main_theorem_1}. 

On the other hand, we need to be careful when we consider the object which is mirror dual to each $E(L,\mathcal{L})_{\theta}$. Now, under the deformation from $\check{X}^n$ to the mirror partner $\check{X}_{\theta}^n$ of $X_{\theta}^n$, the symplectic form is preserved. In other words, it is natural to consider that each $L$ is preserved under this deformation. By contrast, the B-field is twisted by a B-field $B_{\theta}$ depending on $\theta$, and unfortunately, there exists a case such that a given $L$ does not satisfy the requirement written in \cite[Definition 1.1]{Fuk} (\cite[Definition 1.1]{Fuk} gives the definition of objects of the Fukaya categories):
\begin{equation*}
[B_{\theta}]\not \in H^2(L, \mathbb{Z}).
\end{equation*} 
This implies that we can not discuss the deformation of objects $(L,\mathcal{L})$ within the realm of the usual Fukaya categories. In order to overcome this problem, in this paper, we employ the following idea which is proposed in our previous paper \cite{b-field}: instead of ``usual'' line bundles (unitary local systems), employ a ``twisted'' line bundle on each $L$ associated to the flat gerbe whose 1-connection is determined by $B_{\theta}|_L$. Actually, based on this idea, we can construct the deformation $(L,\mathcal{L}_{\theta})$ of each $(L,\mathcal{L})$ associated to the deformation from $\check{X}^n$ to $\check{X}_{\theta}^n$. Moreover, similarly as in the complex geometry side, we specify the moduli space $\mathcal{M}_{\theta}^{\vee}$ of $(L,\mathcal{L}_{\theta})$. This result is given in Theorem \ref{main_theorem_2}. Finally, in Theorem \ref{main_theorem_3}, by comparing $\mathcal{M}_{\theta}^{\vee}$ with $\mathcal{M}_{\theta}$, we give a generalization of the SYZ transform on $(X^n,\check{X}^n)$ to the setting which includes the noncommutative parameter $\theta$.  

This paper is organized as follows. In section 2, we review the basics of the homological mirror symmetry for tori. Main discussions in the complex geometry side are provided in section 3. In subsection 3.1, we recall the noncommutative deformation of holomorphic line bundles $E(L,\mathcal{L})$ which is constructed in \cite{nc} associated to the deformation from $X^n$ to $X_{\theta}^n$. In subsection 3.2, we intensively discuss the problem caused by ambiguity explained in subsection 1.1. In particular, concerning this problem, we extend the noncommutative deformation of holomorphic line bundles $E(L,\mathcal{L})$ which is explained in subsection 3.1 to the more general setting. In subsection 3.3, we specify the moduli space of those extended noncommutative objects $E(L,\mathcal{L})_{\theta}$, and this result is given in Theorem \ref{main_theorem_1}. Main discussions in the symplectic geometry side are provided in section 4. In subsection 4.1, we construct the deformation $(L,\mathcal{L}_{\theta})$ of each $(L,\mathcal{L})$ based on the idea which is proposed in our previous paper \cite{b-field}. These $(L,\mathcal{L}_{\theta})$ are also the mirrors of noncommutative objects $E(L,\mathcal{L})_{\theta}$. In subsection 4.2, we specify the moduli space of the deformed objects $(L,\mathcal{L}_{\theta})$, and this result is given in Theorem \ref{main_theorem_2}. Moreover, in Theorem \ref{main_theorem_3}, we give a generalization of the SYZ transform on $(X^n,\check{X}^n)$ to the setting which includes the noncommutative parameter $\theta$ as a corollary of Theorem \ref{main_theorem_2} and Theorem \ref{main_theorem_1}. In section 5, we comment on the relation between this paper and our previous paper \cite{b-field} from the viewpoint of the Fourier-Mukai duality and generalized complex geometry. In section 6, we explain how to improve the errors included in our prior works \cite{b-field, gerby}. In particular, subsection 6,1 and subsection 6.2 are devoted to revise the errors in \cite{b-field} and \cite{gerby}, respectively.

\subsection{Notations}
We use the following notations throughout this paper. 

Let $M$ be a smooth manifold and $E$ a vector bundle on $M$. Then $TM$ and $T^*M$ denote the tangent bundle on $M$ and the cotangent bundle on $M$, respectively. Also, we denote the space of global sections of $E$ by $\Gamma(M, E)$ or $\Gamma(E)$ for short. However, exceptionally, there are cases that we also denote the space of sections of $E$ defined locally on an open set in $M$ by $\Gamma(E)$. 

On the other hand, $M(n;\mathbb{R})$, $\mathrm{Sym}(n;\mathbb{R})$, and $\mathrm{Alt}(n;\mathbb{R})$ denote the set of real matrices of order $n$, the set of real symmetric matrices of order $n$, and the set of real alternating matrices of order $n$, respectively (we use the notations $M(n;\mathbb{Z})$, etc. in this sense). Moreover, for each $i=1, \ldots, n$, we denote the elementary column vector of order $n$ whose $i$-th component is $1$ by $\bm{e}_i$. Similarly as in the cases of $\bm{e}_1, \ldots, \bm{e}_n$, unless otherwise noted, when we use some kind of vectors, we treat them as column vectors.

\section{Preliminaries}
In this section, we review the basics of the homological mirror symmetry for tori. In particular, we use the framework of generalized complex geometry \cite{hitchin, Gual} auxiliary when we define a mirror partner of a given complex torus.

\subsection{Complex tori and their mirror partners}
The purpose of this subsection is to explain how to define a mirror partner of a given complex torus via the framework of generalized complex geometry.

Let $T$ be a complex matrix of order $n\in \mathbb{N}$ such that $\mathrm{Im}T$ is positive definite and $\mathrm{det}T\not=0$\footnote{Although we do not need to assume the condition $\mathrm{det}T\not=0$ when we define an $n$-dimensional complex torus $\mathbb{C}^n/(\mathbb{Z}^n\oplus T\mathbb{Z}^n)$, in our setting described below, its mirror partner does not exist if $\mathrm{det}T=0$. However, at least, two ways to avoid this problem are actually proposed. For details, see \cite{kazushi}.}. We denote the $n$-dimensional complex torus $\mathbb{C}^n/(\mathbb{Z}^n\oplus T\mathbb{Z}^n)$ by $X^n$:
\begin{equation*}
X^n:=\mathbb{C}^n/(\mathbb{Z}^n\oplus T\mathbb{Z}^n),
\end{equation*}
and in fact, this $X^n$ can be regarded as an $n$-dimensional complex manifold as follows. We fix an $\epsilon>0$ small enough and let
\begin{align*}
O_{m_1\cdots m_n}^{l_1\cdots l_n}:=\biggl\{ \ [x+Ty]\in X^n \ \Bigl. \Bigr| \ &\frac{l_j-1}{3}-\epsilon <x_j <\frac{l_j}{3}+\epsilon, \\
&\frac{m_k-1}{3}-\epsilon <y_k <\frac{m_k}{3}+\epsilon, \ j,k=1,\ldots, n \biggr\}
\end{align*}
be subsets in $X^n$, where $l_j$, $m_k=1$, $2$, $3$, and
\begin{equation*}
x:=\left( \begin{array}{ccc} x_1 \\ \vdots \\ x_n \end{array} \right), \ \ \ y:=\left( \begin{array}{ccc} y_1 \\ \vdots \\ y_n \end{array} \right)\in \mathbb{R}^n.
\end{equation*}
Also, in order to simplify the notations, for each open set $O_{m_1\cdots m_n}^{l_1\cdots l_n}$ in $X^n$, we put
\begin{equation*}
l:=(l_1\cdots l_n), \ \ \ m:=(m_1\cdots m_n),
\end{equation*}
and denote $O_{m_1\cdots m_n}^{l_1\cdots l_n}$ by $O_m^l$:
\begin{equation*}
O_m^l:=O_{m_1\cdots m_n}^{l_1\cdots l_n}.
\end{equation*}
The family $\{ O_m^l \}_{(l;m)\in I}$ of these open sets $O_m^l$ gives an open covering of $X^n$:
\begin{equation*}
X^n=\bigcup_{(l;m)\in I}O_m^l,
\end{equation*}
where
\begin{equation*}
I:=\Bigl\{ (l;m)=(l_1\cdots l_n ; m_1\cdots m_n) \ | \ l_1,\cdots, l_n, m_1,\cdots, m_n=1,2,3 \Bigr\}.
\end{equation*}
Then we locally define the complex coordinate system for each $O_m^l$ by $z:=x+Ty\in \mathbb{C}^n$, where
\begin{equation*}
z:=\left( \begin{array}{ccc} z_1 \\ \vdots \\ z_n \end{array} \right).
\end{equation*}
In particular, the non-trivial coordinate transforms are locally given by the translations $z\mapsto z+\bm{e}_1, \ldots, z\mapsto z+\bm{e}_n$, $z\mapsto z+T\bm{e}_1, \ldots, z\mapsto z+T\bm{e}_n$, and it is clear that they are holomorphic. Moreover, we can regard $X^n$ as a $2n$-dimensional real torus $\mathbb{R}^{2n}/\mathbb{Z}^{2n}\approx (\mathbb{R}^n/\mathbb{Z}^n)\times (\mathbb{R}^n/\mathbb{Z}^n)$ by forgetting its complex structure. In this context, since each point $[z=x+Ty]$ in an open set $O_m^l$ can be identified with $([x], [y])$, we can interpret $(x, y)\in \mathbb{R}^{2n}$ as the coordinate system for each $O_m^l$.

Now, we would like to define a mirror partner of $X^n$. However, in general, it is difficult to find a mirror partner of a given Calabi-Yau manifold, and even more so when we need to consider deformations (noncommutative deformations, gerby deformations, etc.) of it. Concerning this fact, in this paper, we employ the framework of generalized complex geometry in order to find mirror partners of $X^n$ and a noncommutative deformation of it. More precisely, according to \cite{Kaj}, we consider the mirror transforms of the generalized complex structure induced from the complex structure of $X^n$ and a $\beta$-field transform of it by regarding $X^n$ and its noncommutative deformation as generalized complex manifolds, and determine the complexified symplectic tori which are mirror dual to them.

We consider a linear map
\begin{equation*}
\mathcal{I}_T : \Gamma(TX^n\oplus T^*X^n) \to \Gamma(TX^n\oplus T^*X^n)
\end{equation*}
which is expressed locally as 
\begin{align}
&\mathcal{I}_T \left( \frac{\partial}{\partial x}^t, \frac{\partial}{\partial y}^t, dx^t, dy^t \right) \notag \\
&=\left( \frac{\partial}{\partial x}^t, \frac{\partial}{\partial y}^t, dx^t, dy^t \right) \notag \\
&\hspace{3.5mm} \left( \begin{array}{cccc} -XY^{-1} & -Y-XY^{-1}X & O & O \\ Y^{-1} & Y^{-1}X & O & O \\ O & O & (Y^{-1})^tX^t & -(Y^{-1})^t \\ O & O & Y^t+X^t(Y^{-1})^tX^t & -X^t(Y^{-1})^t \end{array} \right), \label{g_c_c}
\end{align}
where
\begin{gather*}
X:=\mathrm{Re}T, \ \ \ Y:=\mathrm{Im}T, \\ 
\frac{\partial}{\partial x}:=\left( \begin{array}{ccc} \frac{\partial}{\partial x_1} \\ \vdots \\ \frac{\partial}{\partial x_n} \end{array} \right), \ \ \ \frac{\partial}{\partial y}:=\left( \begin{array}{ccc} \frac{\partial}{\partial y_1} \\ \vdots \\ \frac{\partial}{\partial y_n} \end{array} \right), \ \ \ 
dx:=\left( \begin{array}{ccc} dx_1 \\ \vdots \\ dx_n \end{array} \right), \ \ \ dy:=\left( \begin{array}{ccc} dy_1 \\ \vdots \\ dy_n \end{array} \right),
\end{gather*}
and $M^t$ denotes the transpose of a given matrix $M$. We see that the linear map $\mathcal{I}_T$ is a generalized complex structure induced from the complex structure of $X^n$.

Let us define a mirror partner $\check{X}^n$ of $X^n$. In general, for a $2n$-dimensional real torus $T^{2n}$ equipped with a generalized complex structure $\mathcal{I}$ over $T^{2n}$, a generalized complex structure $\check{\mathcal{I}}$ over $T^{2n}$ which is mirror dual to $\mathcal{I}$ is defined by
\begin{equation}
\check{\mathcal{I}}=M_{(n)} \hspace{0.5mm} \mathcal{I} \hspace{0.5mm} M_{(n)}, \label{mirror}
\end{equation}
where
\begin{equation*}
M_{(n)}:=\left( \begin{array}{cccc} I_n & O & O & O \\ O & O & O & I_n \\ O & O & I_n & O \\ O & I_n & O & O \end{array} \right), 
\end{equation*}
$I_n$ is the identity matrix of order $n$, and note that $M_{(n)}^{-1}=M_{(n)}$ (see \cite{part1, part2, Kaj, kazushi})\footnote{Precisely speaking, objects which are treated in \cite{Kaj} are $2n$-dimensional flat generalized K$\ddot{\mathrm{a}}$hler tori, and see \cite{Kaj} for details of the definition of this notion. We can regard the tori which are treated in this paper as examples of such objects.}. We will explain the meaning of this definition at the last of this subsection. According to the mirror transform (\ref{mirror}), we can calculate the generalized complex structure $\check{\mathcal{I}}_T$ which is mirror dual to $\mathcal{I}_T$ locally as 
\begin{align}
&\check{\mathcal{I}}_T \left( \frac{\partial}{\partial \check{x}}^t, \frac{\partial}{\partial \check{y}}^t, d\check{x}^t, d\check{y}^t \right) \notag \\
&=\left( \frac{\partial}{\partial \check{x}}^t, \frac{\partial}{\partial \check{y}}^t, d\check{x}^t, d\check{y}^t \right) \notag \\
&\hspace{3.5mm} \left( \begin{array}{cccc} -XY^{-1} & O & O & -Y-XY^{-1}X \\ O & -X^t(Y^{-1})^t & Y^t+X^t(Y^{-1})^tX^t & O \\ O & -(Y^{-1})^t & (Y^{-1})^tX^t & O \\ Y^{-1} & O & O & Y^{-1}X \end{array} \right). \label{g_c_s}
\end{align}
Here,
\begin{equation*}
\frac{\partial}{\partial \check{x}}:=\left( \begin{array}{ccc} \frac{\partial}{\partial x^1} \\ \vdots \\ \frac{\partial}{\partial x^n} \end{array} \right), \ \ \ \frac{\partial}{\partial \check{y}}:=\left( \begin{array}{ccc} \frac{\partial}{\partial y^1} \\ \vdots \\ \frac{\partial}{\partial y^n} \end{array} \right), \ \ \ d\check{x}:=\left( \begin{array}{ccc} dx^1 \\ \vdots \\ dx^n \end{array} \right), \ \ \ d\check{y}:=\left( \begin{array}{ccc} dy^1 \\ \vdots \\ dy^n \end{array} \right)
\end{equation*}
are associated to
\begin{equation*}
\check{x}:=\left( \begin{array}{ccc} x^1 \\ \vdots \\ x^n \end{array} \right), \ \ \ \check{y}:=\left( \begin{array}{ccc} y^1 \\ \vdots \\ y^n \end{array} \right)\in \mathbb{R}^n 
\end{equation*}
which are the local coordinates of a $2n$-dimensional real torus $\mathbb{R}^{2n}/\mathbb{Z}^{2n}\approx (\mathbb{R}^n/\mathbb{Z}^n)\times (\mathbb{R}^n/\mathbb{Z}^n)$ in the following sense. Similarly as in the complex geometry side, we fix an $\epsilon>0$ small enough and take an open covering $\{\check{O}_m^l \}_{(l;m)\in I}$ of $\mathbb{R}^{2n}/\mathbb{Z}^{2n}$ which is obtained by replacing $(x,y)$ with $(\check{x}, \check{y})$, i.e.,
\begin{align*}
\check{O}_m^l:=\biggl\{ \ \left( \begin{array}{cc} \lbrack \check{x} \rbrack \\ \lbrack \check{y} \rbrack \end{array} \right)\in (\mathbb{R}^n/\mathbb{Z}^n)\times (\mathbb{R}^n/\mathbb{Z}^n) \ \Bigl. \Bigr| \ &\frac{l_j-1}{3}-\epsilon <x^j <\frac{l_j}{3}+\epsilon, \\
&\frac{m_k-1}{3}-\epsilon <y^k <\frac{m_k}{3}+\epsilon, \ j,k=1,\ldots, n \biggr\}.
\end{align*}
The family $\{ \check{O}_m^l \}_{(l;m)\in I}$ give an open covering of $\mathbb{R}^{2n}/\mathbb{Z}^{2n}$:
\begin{equation*}
\mathbb{R}^{2n}/\mathbb{Z}^{2n}=\bigcup_{(l;m)\in I} \check{O}_m^l.
\end{equation*}
Then we locally define the coordinate system for each $\check{O}_m^l$ by $(\check{x}, \check{y})\in \mathbb{R}^{2n}$, and the non-trivial coordinate transforms are locally given by the translations by the elementary column vectors of order $2n$. On the other hand, the representation matrix in (\ref{g_c_s}) can be rewritten to
\begin{align}
&\left( \begin{array}{cccc} I_n & O & O & O \\ O & I_n & O & O \\ O & -\mathrm{Re}(-(T^{-1})^t) & I_n & O \\ (\mathrm{Re}(-(T^{-1})^t))^t & O & O & I_n \end{array} \right) \notag \\
&\left( \begin{array}{cccc} O & O & O & -((\mathrm{Im}(-(T^{-1})^t))^{-1})^t \\ O & O & (\mathrm{Im}(-(T^{-1})^t))^{-1} & O \\ O & -\mathrm{Im}(-(T^{-1})^t) & O & O \\ (\mathrm{Im}(-(T^{-1})^t))^t & O & O & O \end{array} \right) \notag \\
&\left( \begin{array}{cccc} I_n & O & O & O \\ O & I_n & O & O \\ O & \mathrm{Re}(-(T^{-1})^t) & I_n & O \\ -(\mathrm{Re}(-(T^{-1})^t))^t & O & O & I_n \end{array} \right), \label{g_c_s_m}
\end{align}
where the assumption $\mathrm{det}T\not=0$ implies $\mathrm{det}(Y+XY^{-1}X)\not=0$, and
\begin{align*}
&\mathrm{Re}(-(T^{-1})^t)=-((Y+XY^{-1}X)^{-1})^tX^t(Y^{-1})^t, \\ 
&\mathrm{Im}(-(T^{-1})^t)=((Y+XY^{-1}X)^{-1})^t.
\end{align*}
Thus, when we define a mirror partner of $X^n$ according to the mirror transform (\ref{mirror}), it is given by the $2n$-dimensional real torus $\mathbb{R}^{2n}/\mathbb{Z}^{2n}$ equipped with the following complexified symplectic structure. Let us define two matrices $\underline{\omega}^{\vee}$ and $\underline{B}^{\vee}$ by
\begin{equation*}
\underline{\omega}^{\vee}=\mathrm{Im}(-(T^{-1})^t)
\end{equation*}
and
\begin{equation*}
\underline{B}^{\vee}=\mathrm{Re}(-(T^{-1})^t),
\end{equation*}
respectively. Since the representation matrix in (\ref{g_c_s}) turns out to be the matrix (\ref{g_c_s_m}), we see that the generalized complex structure (\ref{g_c_s}) is induced from the symplectic form
\begin{equation*}
\omega^{\vee}:=d\check{x}^t \underline{\omega}^{\vee} d\check{y}
\end{equation*} 
and the B-field
\begin{equation*}
B^{\vee}:=d\check{x}^t \underline{B}^{\vee} d\check{y}.
\end{equation*}
By using these $\omega^{\vee}$ and $B^{\vee}$, we define a complexified symplectic form $\tilde{\omega}^{\vee}$ by
\begin{equation*}
\tilde{\omega}^{\vee}=\mathbf{i}\omega^{\vee}+B^{\vee}=d\check{x}^t (-(T^{-1})^t) d\check{y},
\end{equation*} 
where $\mathbf{i}:=\sqrt{-1}$. Hereafter, we denote this complexified symplectic torus by
\begin{equation*}
\check{X}^n:=\Bigl( \mathbb{R}^{2n}/\mathbb{Z}^{2n}, \ \tilde{\omega}^{\vee}=d\check{x}^t (-(T^{-1})^t) d\check{y} \Bigr).
\end{equation*}

We now explain the meaning of the mirror transform (\ref{mirror}). We can regard $\check{X}^n$ as the trivial torus fibration $\check{X}^n \to \mathbb{R}^n/\mathbb{Z}^n$ ; $([\check{x}], [\check{y}]) \mapsto [\check{x}]$, and similarly, $X^n$ can also be regarded as the trivial torus fibration $X^n\approx (\mathbb{R}^n/\mathbb{Z}^n)\times (\mathbb{R}^n/\mathbb{Z}^n) \to \mathbb{R}^n/\mathbb{Z}^n$ ; $([x], [y]) \mapsto [x]$ via the identification $X^n\approx \mathbb{R}^{2n}/\mathbb{Z}^{2n}\approx (\mathbb{R}^n/\mathbb{Z}^n)\times (\mathbb{R}^n/\mathbb{Z}^n)$. In particular, $y$ and $\check{y}$ are the local coordinates of the fibers of these trivial torus fibrations. Although we use the notations $x$ and $\check{x}$ in this paper, from the viewpoint of the SYZ construction \cite{SYZ}, we may regard the trivial torus fibrations $X^n \to \mathbb{R}^n/\mathbb{Z}^n$ and $\check{X}^n \to \mathbb{R}^n/\mathbb{Z}^n$ as the trivial torus fibrations on the same base space $\mathbb{R}^n/\mathbb{Z}^n$, namely, we may identify $x$ with $\check{x}$. The mirror transform (\ref{mirror}) describes the T-duality between the fibers of these trivial torus fibrations.

\subsection{Holomorphic line bundles on $X^n$}
The purpose of this subsection is to explain a class of holomorphic line bundles which are expected to generate the bounded derived category of coherent sheaves over $X^n$.

We define a class of holomorphic line bundles $E_A \to X^n$ with integrable connections $\nabla_{(A,p,q)}$ as follows. In fact, we first construct these objects as smooth complex line bundles on $X^n$ with connections, and discuss the holomorphicity of them later. Let us take $A\in M(n;\mathbb{Z})$,
\begin{equation*}
p=\left( \begin{array}{ccc} p_1 \\ \vdots \\ p_n \end{array} \right), \ \ \ q=\left( \begin{array}{ccc} q_1 \\ \vdots \\ q_n \end{array} \right)\in \mathbb{R}^n
\end{equation*}
arbitrary, and fix them. We define a map $j_A : (\mathbb{Z}^n\oplus T\mathbb{Z}^n)\times \mathbb{C}^n \to \mathbb{C}^{\times}$ (it is enough to define such a map for $\bm{e}_1, \ldots, \bm{e}_n$, $T\bm{e}_1, \ldots, T\bm{e}_n$ and elements in $\mathbb{C}^n$ since $\mathbb{Z}^n\oplus T\mathbb{Z}^n$ is generated by $\bm{e}_1, \ldots, \bm{e}_n$, $T\bm{e}_1, \ldots, T\bm{e}_n$) by
\begin{align*}
&j_A(\bm{e}_i, z)=\mathrm{exp}\left( 2\pi \mathbf{i} \bm{e}_i^t A^t \left( T-\bar{T} \right)^{-1} (z-\bar{z}) \right)=\mathrm{exp}\left( 2\pi \mathbf{i} \bm{e}_i^t A^t y \right), \\
&j_A(T\bm{e}_i, z)=1,
\end{align*} 
where $i=1, \ldots, n$, $\mathbb{C}^{\times}:=\mathbb{C}\backslash \{ 0 \}$, and note that $z=x+Ty$. For any $k$, $l=1, \ldots, n$, it is easy to check that the map $j_A$ satisfies the relations
\begin{align}
&j_A(\bm{e}_l, z+\bm{e}_k) j_A(\bm{e}_k, z)=j_A(\bm{e}_k, z+\bm{e}_l) j_A(\bm{e}_l, z), \label{cocycle_1} \\
&j_A(T\bm{e}_l, z+\bm{e}_k) j_A(\bm{e}_k, z)=j_A(\bm{e}_k, z+T\bm{e}_l) j_A(T\bm{e}_l, z), \label{cocycle_2} \\
&j_A(T\bm{e}_l, z+T\bm{e}_k) j_A(T\bm{e}_k, z)=j_A(T\bm{e}_k, z+T\bm{e}_l) j_A(T\bm{e}_l, z), \label{cocycle_3}
\end{align}
namely, the following are well-defined:
\begin{align*}
&j_A(\bm{e}_k+\bm{e}_l, z):=j_A(\bm{e}_l, z+\bm{e}_k) j_A(\bm{e}_k, z), \\
&j_A(\bm{e}_k+T\bm{e}_l, z):=j_A(T\bm{e}_l, z+\bm{e}_k) j_A(\bm{e}_k, z), \\
&j_A(T\bm{e}_k+T\bm{e}_l, z):=j_A(T\bm{e}_l, z+T\bm{e}_k) j_A(T\bm{e}_k, z).
\end{align*}
In general, the map $j_A$ is called the factor of automorphy for the smooth complex line bundle $E_A$ that we should treat, and the smooth complex functions $j_A(\bm{e}_1, z), \ldots, j_A(\bm{e}_n, z)$, $j_A(T\bm{e}_1, z), \ldots, j_A(T\bm{e}_n, z)$ are essentially the transition functions of it. In particular, the relations (\ref{cocycle_1}), (\ref{cocycle_2}), (\ref{cocycle_3}) correspond to the cocycle condition. By using this $j_A$, let us define an action of $\mathbb{Z}^n\oplus T\mathbb{Z}^n$ on $\mathbb{C}^n\times \mathbb{C}$ by
\begin{align*}
&\bigl( z, t \bigr)\in \mathbb{C}^n\times \mathbb{C} \ \longmapsto \ \bigl( z+\bm{e}_i, j_A(\bm{e}_i, z) t \bigr)\in \mathbb{C}^n\times \mathbb{C}, \\
&\bigl( z, t \bigr)\in \mathbb{C}^n\times \mathbb{C} \ \longmapsto \ \bigl( z+T\bm{e}_i, j_A(T\bm{e}_i, z) t \bigr)\in \mathbb{C}^n\times \mathbb{C}
\end{align*}
for each $i=1, \ldots, n$. Then we can obtain the smooth complex line bundle $E_A \to X^n$ as the quotient of $\mathbb{C}^n\times \mathbb{C}$ by the action of $\mathbb{Z}^n\oplus T\mathbb{Z}^n$:
\begin{equation*}
\mathrm{Tot}(E_A):=\big( \mathbb{C}^n\times \mathbb{C} \big)/\bigl( \mathbb{Z}^n\oplus T\mathbb{Z}^n \bigr),
\end{equation*}
where $\mathrm{Tot}(\mathcal{E})$ denotes the total space of a vector bundle $\mathcal{E}$ on a smooth manifold. Hereafter, for each $i=1, \ldots, n$, via the identification $\mathbb{C}^n\cong \mathbb{R}^{2n}$ ; $z=x+Ty \mapsto (x,y)$, we denote $j_A(\bm{e}_i, z)$ and $j_A(T\bm{e}_i, z)$ by $j_A(\bm{e}_i, x, y)$ and $j_A(T\bm{e}_i, x, y)$, respectively, because using the real coordinate system $(x,y)$ is more convenient than using the complex coordinate system $z=x+Ty$ from the viewpoint of SYZ picture:
\begin{equation*}
j_A(\bm{e}_i, x, y):=j_A(\bm{e}_i, z), \ \ \ j_A(T\bm{e}_i, x, y):=j_A(T\bm{e}_i, z).
\end{equation*}
Moreover, by using $p$, $q\in \mathbb{R}^n$, we locally define a differential operator $\nabla_{(A,p,q)}$ on $\Gamma(E_A)$ by
\begin{equation}
\nabla_{(A,p,q)}=d+\omega_{(A,p,q)}, \ \ \ \omega_{(A,p,q)}=-2\pi \mathbf{i} \Bigl( \Bigl( x^t A^t+p^t \Bigr)+q^t T \Bigr)dy, \label{conn}
\end{equation}
where $d$ denotes the exterior derivative. We can easily check that this $\nabla_{(A,p,q)}$ is compatible with the transition functions of $E_A$, i.e., the factor of automorphy $j_A$ of $E_A$, so indeed, the differential operator $\nabla_{(A,p,q)}$ defines a connection on $E_A$. We denote the smooth complex line bundle $E_A \to X^n$ with the connection $\nabla_{(A,p,q)}$ by $E_{(A,p,q)}^{\nabla}$, i.e., $E_{(A,p,q)}^{\nabla}:=(E_A, \nabla_{(A,p,q)})$. 

In order to investigate the holomorphicity of each $E_{(A,p,q)}^{\nabla}$, we recall the definition of the integrability of connections following \cite{steven}. Let $X$ be a compact K$\ddot{\mathrm{a}}$hler manifold, and we consider a smooth complex vector bundle $E$ with a connection $D$ on $X$. In this context, $D^{(0,1)}$ denotes the (0,1)-part of $D$. 
\begin{definition}[{\cite[Definition 1.4.1]{steven}}] 
A connection $D$ is called integrable if $(D^{(0,1)})^2=0$.
\end{definition}
Therefore, $(E,D)$ can be considered as a holomorphic vector bundle on $X$ if $D$ is integrable.

Now, we give the following proposition without its proof since a similar statement is proved in \cite[Proposition 3.1]{bijection} for instance.
\begin{proposition} \label{hol}
For a given triple $(A,p,q)\in M(n;\mathbb{Z})\times \mathbb{R}^n\times \mathbb{R}^n$, the connection $\nabla_{(A,p,q)}$ is an integrable connection on $E_A \to X^n$ if and only if 
\begin{equation*}
AT\in \mathrm{Sym}(n;\mathbb{C}).
\end{equation*}
\end{proposition}
Under the assumption $AT\in \mathrm{Sym}(n;\mathbb{C})$, we see that holomorphic line bundles $E_{(A,p,q)}^{\nabla}$ forms a dg-category $DG_{X^n}$ as a corollary of \cite[Theorem 5.7]{b-field} (see also \cite{kajiura}). In general, for any $A_{\infty}$-category $\mathscr{A}$, we can construct the enhanced triangulated category $Tr(\mathscr{A})$ by using the Bondal-Kapranov-Kontsevich construction \cite{bondal, Kon}. We expect that the dg-category $DG_{X^n}$ generates the bounded derived category $D^b(Coh(X^n))$ of coherent sheaves over $X^n$ in the sense of the Bondal-Kapranov-Kontsevich construction:
\begin{equation*}
Tr(DG_{X^n})\cong D^b(Coh(X^n)).
\end{equation*}
At least, it is known that it split generates $D^b(Coh(X^n))$ when $X^n$ is an abelian variety (see \cite{orlov, abouzaid}).

Here, based on subsection 3.2 in \cite{star}, we give the local expression of a global holomorphic section of $E_{(A,p,q)}^{\nabla}$ in the case that $T=\mathbf{i}\cdot I_n$ and $A$ is positive definite (it is clear that these $A$ and $T$ satisfies the condition $AT\in \mathrm{Sym}(n;\mathbb{C})$). Note that $\# (\mathbb{Z}^n/A\mathbb{Z}^n)=\mathrm{det}A\in \mathbb{N}$. In this case, a global holomorphic section $\xi_{(A,p,q)}$ of $E_{(A,p,q)}^{\nabla}$ is expressed locally as 
\begin{align*}
&\xi_{(A,p,q)}(x,y):=\sum_{m\in \mathbb{Z}^n/A\mathbb{Z}^n} C_m \sum_{l\in \mathbb{Z}^n} \xi_{(A,p,q,l,m)}(x) \ \mathrm{exp}\left( 2\pi \mathbf{i} (-Al+m)^t y \right), \\
&\xi_{(A,p,q,l,m)}(x):=\mathrm{exp}\left( -\pi (x^t-(-Al+m-p)^t(A^{-1})^t)A(x-A^{-1}(-Al+m-p))-2\pi \mathbf{i} (x+l)^t q \right),
\end{align*}
where $C_m\in \mathbb{C}$ for each $m\in \mathbb{Z}^n/A\mathbb{Z}^n$. In other words, now, the space $\Gamma(X^n, E_{(A,p,q)}^{\nabla})$ is a $(\mathrm{det}A)$-dimensional $\mathbb{C}$-vector space, and a basis of this $\mathbb{C}$-vector space $\Gamma(X^n, E_{(A,p,q)}^{\nabla})$ is given by
\begin{equation*}
\left\{ \sum_{l\in \mathbb{Z}^n} \xi_{(A,p,q,l,m)}(x) \ \mathrm{exp}\left( 2\pi \mathbf{i} (-Al+m)^t y \right) \ \Bigl. \Bigr| \ m\in \mathbb{Z}^n/A\mathbb{Z}^n \right\}.
\end{equation*}
Actually, for each $m\in \mathbb{Z}^n/A\mathbb{Z}^n$, we can verify that the locally defined smooth function $\sum_{l\in \mathbb{Z}^n} \xi_{(A,p,q,l,m)}(x) \ \mathrm{exp}\left( 2\pi \mathbf{i} (-Al+m)^t y \right)$ satisfies the relations
\begin{align*}
&j_A(\bm{e}_i, x, y) \left( \sum_{l\in \mathbb{Z}^n} \xi_{(A,p,q,l,m)}(x) \ \mathrm{exp}\left( 2\pi \mathbf{i} (-Al+m)^t y \right) \right) \\
&=\sum_{l\in \mathbb{Z}^n} \xi_{(A,p,q,l,m)}(x+\gamma_i) \ \mathrm{exp}\left( 2\pi \mathbf{i} (-Al+m)^t y \right), \\
&j_A(T\bm{e}_i, x, y) \left( \sum_{l\in \mathbb{Z}^n} \xi_{(A,p,q,l,m)}(x) \ \mathrm{exp}\left( 2\pi \mathbf{i} (-Al+m)^t y \right) \right) \\
&=\sum_{l\in \mathbb{Z}^n} \xi_{(A,p,q,l,m)}(x) \ \mathrm{exp}\left( 2\pi \mathbf{i} (-Al+m)^t (y+\bm{e}_i) \right), \\
&\left( \bar{\partial}+\omega_{(A,p,q)}^{(0,1)} \right) \left( \sum_{l\in \mathbb{Z}^n} \xi_{(A,p,q,l,m)}(x) \ \mathrm{exp}\left( 2\pi \mathbf{i} (-Al+m)^t y \right) \right)=0,
\end{align*}
where $\bar{\partial}$, $\omega_{(A,p,q)}^{(0,1)}$ denote the (0,1)-part of $d$, the (0,1)-part of $\omega_{(A,p,q)}$, respectively, and $i=1, \ldots, n$. Moreover, $\xi_{(A,p,q,l,m)}$ belongs to the Schwarz space $\mathcal{S}(\mathbb{R}^n)$ over $\mathbb{R}^n$ for each pair $(l,m)\in \mathbb{Z}^n \times (\mathbb{Z}^n/A\mathbb{Z}^n)$ since each $\xi_{(A,p,q,l,m)}$ is a Gaussian.

\subsection{Affine Lagrangian submanifolds in $\check{X}^n$ with unitary local systems}
The purpose of this subsection is to give the objects defined on $\check{X}^n$ which are mirror dual to holomorphic line bundles $E_{(A,p,q)}^{\nabla}$. 

Before starting main discussions, we recall the definition of objects of the Fukaya category over a given complexified symplectic manifold (see \cite[Definition 1.1]{Fuk} for example). Let $M_{\omega, B}:=(M, \mathbf{i}\omega+B)$ be a complexified symplectic manifold, where $M$ is a smooth even-dimensional real manifold, $\omega$ is a symplectic form on $M$, and $B$ is a B-field on $M$. For this $M_{\omega, B}$, an object of the Fukaya category over $M_{\omega, B}$ consists of a Lagrangian submanifold $L$ in $M_{\omega, B}$, i.e., a submanifold $L$ in $M_{\omega, B}$ which satisfies
\begin{equation*}
\mathrm{dim}_{\mathbb{R}}L=\frac{1}{2}\mathrm{dim}_{\mathbb{R}}M, \ \ \ \Bigl. \omega \Bigr|_L=0,
\end{equation*}
and a smooth complex line bundle $\mathcal{L} \to L$ with a unitary connection $\nabla_{\mathcal{L}}$ such that 
\begin{equation}
\Omega_{\nabla_{\mathcal{L}}}=\Bigl. 2\pi \mathbf{i}B \Bigr|_L, \label{curv_b_field}
\end{equation}
where $\Omega_{\nabla_{\mathcal{L}}}$ denotes the curvature form of $\nabla_{\mathcal{L}}$.

\begin{rem} \label{fuk_ob_tori}
As remarked in \textup{\cite[Remark 1.19]{Fuk}}, actually, it is enough to consider the conditions
\begin{align}
&\Omega_{\nabla_{\mathcal{L}}}=0, \label{curv_b_field_1} \\
&\Bigl. B \Bigr|_L=0 \label{curv_b_field_2}
\end{align}
instead of the requirement \textup{(\ref{curv_b_field})} on $\check{X}^n$. Concerning this fact, below, we use the conditions \textup{(\ref{curv_b_field_1})}, \textup{(\ref{curv_b_field_2})} when we consider objects of the Fukaya category over $\check{X}^n$.
\end{rem}

Let us take $A\in M(n;\mathbb{Z})$, 
\begin{equation*}
p=\left( \begin{array}{ccc} p_1 \\ \vdots \\ p_n \end{array} \right), \ \ \ q=\left( \begin{array}{ccc} q_1 \\ \vdots \\ q_n \end{array} \right)\in \mathbb{R}^n
\end{equation*}
arbitrary, and fix them. We consider an $n$-dimensional affine submanifold
\begin{equation*}
\tilde{L}_{(A,p)}:=\left\{ \left( \begin{array}{cc} \check{x} \\ A\check{x}+p \end{array} \right)\in \mathbb{R}^n\times \mathbb{R}^n\approx \mathbb{R}^{2n} \right\}
\end{equation*}
in the covering space $\mathbb{R}^{2n}$ of $\check{X}^n$ (we may regard $(\check{x}, \check{y})$ as the coordinates of the covering space $\mathbb{R}^{2n}$ of $\check{X}^n$), and define an $n$-dimensional affine submanifold $L_{(A,p)}$ in $\check{X}^n$ by
\begin{equation*}
L_{(A,p)}=\pi \left( \tilde{L}_{(A,p)} \right),
\end{equation*}
where $\pi : \mathbb{R}^{2n} \to \check{X}^n$ is the covering map. We further take the trivial complex line bundle $\mathcal{O}_{(A,p)} \to L_{(A,p)}$ with a flat connection
\begin{equation*}
\nabla_{(A,p,q)}^{\vee}:=d+\omega_{(A,p,q)}^{\vee}, \ \ \ \omega_{(A,p,q)}^{\vee}:=2\pi \mathbf{i} q^t d\check{x}
\end{equation*}
according to Remark \ref{fuk_ob_tori}, and set $\mathcal{O}_{(A,p,q)}^{\nabla}:=(\mathcal{O}_{(A,p)}, \nabla_{(A,p,q)}^{\vee})$. Let us denote the pair of the $n$-dimensional affine submanifold $L_{(A,p)}$ in $\check{X}^n$ and the trivial complex line bundle with the flat connection $\mathcal{O}_{(A,p,q)}^{\nabla}$ by $\mathcal{L}_{(A,p,q)}^{\nabla}$, i.e., $\mathcal{L}_{(A,p,q)}^{\nabla}:=(L_{(A,p)}, \mathcal{O}_{(A,p,q)}^{\nabla})$. Then the following proposition holds. Here, we omit its proof since a similar statement is proved in subsection 4.1 in \cite{bijection} for instance.
\begin{proposition} \label{fuk_ob} 
For a given triple $(A,p,q)\in M(n;\mathbb{Z})\times \mathbb{R}^n\times \mathbb{R}^n$, the pair $\mathcal{L}_{(A,p,q)}^{\nabla}$ gives an object of the Fukaya category over $\check{X}^n$ if and only if
\begin{equation*}
AT\in \mathrm{Sym}(n;\mathbb{C}).
\end{equation*}
\end{proposition}
We denote the full subcategory of the Fukaya category over $\check{X}^n$ consisting of objects $\mathcal{L}_{(A,p,q)}^{\nabla}$ satisfying the condition $AT\in \mathrm{Sym}(n;\mathbb{C})$ by $Fuk_{\rm sub}(\check{X}^n)$.

By comparing Proposition \ref{hol} with Proposition \ref{fuk_ob}, we immediately obtain the following proposition which indicates that the object 
\begin{equation*}
E_{(A,p,q)}^{\nabla}
\end{equation*}
of $DG_{X^n}$ and the object 
\begin{equation*}
\mathcal{L}_{(A,p,q)}^{\nabla}
\end{equation*}
of $Fuk_{\rm sub}(\check{X}^n)$ are mirror dual to each other for a given triple $(A,p,q)\in M(n;\mathbb{Z})\times \mathbb{R}^n\times \mathbb{R}^n$.
\begin{proposition} \label{ob_mirror}
For a given triple $(A,p,q)\in M(n;\mathbb{Z})\times \mathbb{R}^n\times \mathbb{R}^n$, the connection $\nabla_{(A,p,q)}$ is an integrable connection on $E_A \to X^n$ if and only if the pair $\mathcal{L}_{(A,p,q)}^{\nabla}$ gives an object of the Fukaya category over $\check{X}^n$.
\end{proposition}
Below, we assume that
\begin{equation*}
AT\in \mathrm{Sym}(n;\mathbb{C}).
\end{equation*}

We give two additional explanations for the correspondence between $E_{(A,p,q)}^{\nabla}$ and $\mathcal{L}_{(A,p,q)}^{\nabla}$. 

As the first point, we comment on the SYZ transform. We can regard the complexified symplectic torus $\check{X}^n$ as the trivial torus fibration $\check{X}^n \to \mathbb{R}^n/\mathbb{Z}^n$ ; $([\check{x}], [\check{y}]) \mapsto [\check{x}]$ as explained in subsection 2.1. Then we can regard each affine Lagrangian submanifold $L_{(A,p)}$ in $\check{X}^n$ as the graph of the section $\mathbb{R}^n/\mathbb{Z}^n \to \check{X}^n$ ; $[\check{x}] \mapsto ([\check{x}], [A\check{x}+p])$ of $\check{X}^n \to \mathbb{R}^n/\mathbb{Z}^n$. In particular, the correspondence 
\begin{equation*}
\mathcal{L}_{(A,p,q)}^{\nabla} \ \longmapsto \ E_{(A,p,q)}^{\nabla}
\end{equation*}
is sometimes called the SYZ transform (see \cite{leung, A-P}). 

As the second point, we explain an interpretation for the global holomorphic section $\xi_{(A,p,q)}$ which is given in subsection 2.2 from the viewpoint of the homological mirror symmetry, so let us assume that $T=\mathbf{i}\cdot I_n$ and $A$ is positive definite. Here, for simplicity, we put $q:=0\in \mathbb{R}^n$. We can regard the set $\pi^{-1}(L_{(A,p)})=(\pi^{-1}\circ \pi)(\tilde{L}_{(A,p)})$ as 
\begin{equation*}
\pi^{-1} \left( L_{(A,p)} \right)=\left\{ \left( \begin{array}{cc} \check{x} \\ A\check{x}+p+Al-m \end{array} \right)\in \mathbb{R}^n\times \mathbb{R}^n\approx \mathbb{R}^{2n} \right\},
\end{equation*}
where $l\in \mathbb{Z}^n$ and $m\in \mathbb{Z}^n/A\mathbb{Z}^n$ correspond to the ones appeared in the local expression of $\xi_{(A,p,0)}$. On the other hand, in the complex geometry side, the space $\Gamma(X^n, E_{(A,p,0)}^{\nabla})$ can be identified with the space of holomorphic maps from the trivial holomorphic line bundle $E_{(O,0,0)}^{\nabla} \to X^n$ to $E_{(A,p,0)}^{\nabla}$, and $E_{(O,0,0)}^{\nabla}$ corresponds to the affine Lagrangian submanifold
\begin{equation*}
L_{(O,0)}=\pi \left(\tilde{L}_{(O,0)} \right)=\pi \left( \left\{ \left( \begin{array}{cc} \check{x} \\ 0 \end{array} \right)\in \mathbb{R}^n\times \mathbb{R}^n\approx \mathbb{R}^{2n} \right\} \right)
\end{equation*} 
in $\check{X}^n$. Under this situation, frankly speaking, the axes $x=A^{-1}(m-p)$ ($l=0\in \mathbb{Z}^n$, $m\in \mathbb{Z}^n/A\mathbb{Z}^n$) of the Gaussians $\xi_{(A,p,0,0,m)}(x)$ correspond to the intersection points of $L_{(A,p)}$ and $L_{(O,0)}$ in $\check{X}^n$:
\begin{align*}
&L_{(A,p)}\cap L_{(O,0)}=\left\{ \left( \begin{array}{cc} \lbrack A^{-1}(m-p) \rbrack \\ \lbrack 0 \rbrack \end{array} \right)\in (\mathbb{R}^n/\mathbb{Z}^n)\times (\mathbb{R}^n/\mathbb{Z}^n)\approx \mathbb{R}^{2n}/\mathbb{Z}^{2n}\approx \check{X}^n \ \Bigl. \Bigr| \ m\in \mathbb{Z}^n/A\mathbb{Z}^n \right\}, \\
&\# \left( L_{(A,p)}\cap L_{(O,0)} \right)=\# \left( \mathbb{Z}^n/A\mathbb{Z}^n \right)=\mathrm{det}A=\mathrm{dim}_{\mathbb{C}}\Gamma\left( X^n, E_{(A,p,0)}^{\nabla} \right). 
\end{align*} 

\section{The deformation of holomorphic line bundles associated to a noncommutative deformation of $X^n$}
Before explaining the purpose of this section, we mention the Poisson bivector on $X^n$ briefly. Let us take matrices $\theta_1=((\theta_1)_{ij})$, $\theta_3=((\theta_3)_{ij})\in \mathrm{Alt}(n;\mathbb{R})$, $\theta_2=((\theta_2)_{ij})\in M(n;\mathbb{R})$. By using them, we consider a Poisson bivector which is expressed locally as
\begin{equation}
\sum_{i, j=1}^n \left( \frac{1}{2}(\theta_1)_{ij}\frac{\partial}{\partial x_i}\wedge \frac{\partial}{\partial x_j}+(\theta_2)_{ij}\frac{\partial}{\partial x_i}\wedge \frac{\partial}{\partial y_j}+\frac{1}{2}(\theta_3)_{ij}\frac{\partial}{\partial y_i}\wedge \frac{\partial}{\partial y_j} \right). \label{P_bivector}
\end{equation}  
In particular, the matrices $\theta_1$, $\theta_2$, $\theta_3$ are constant matrices and the non-trivial coordinate transforms are the translations by the elementary column vectors. Therefore, in fact, the Poisson bivector (\ref{P_bivector}) is defined globally on $\mathbb{R}^{2n}/\mathbb{Z}^{2n}\approx (\mathbb{R}^n/\mathbb{Z}^n)\times (\mathbb{R}^n/\mathbb{Z}^n)\approx X^n$. 

Let us take an arbitrary matrix $\theta=(\theta_{ij})\in \mathrm{Alt}(n;\mathbb{R})$, and fix it. Below, we focus on the noncommutative deformation of $X^n$ associated to the nonformal deformation quantization of $X^n$ by the Poisson bivector (\ref{P_bivector}) in the case that $\theta_3=\theta$ and $\theta_1=\theta_2=O$:
\begin{equation}
\Pi_{\theta}:=\frac{1}{2}\sum_{i, j=1}^n \theta_{ij}\frac{\partial}{\partial y_i}\wedge \frac{\partial}{\partial y_j} \label{theta}
\end{equation} 
which is treated in \cite{nc} (this deformation is called Type $\theta_3$ in \cite{nc}). 

In general, associated to the deformation of $X^n$ in the above sense, we can construct a non-trivial deformation of the trivial holomorphic line bundle $E_{(O,0,0)}^{\nabla} \to X^n$ by twisting it with a suitable isomorphism. In this section, from this point of view, we extend the noncommutative deformation of holomorphic line bundles $E_{(A,p,q)}^{\nabla}$ which is constructed in \cite{nc} to the more general setting. In particular, we specify the moduli space of such extended noncommutative objects in Theorem \ref{main_theorem_1}.

\subsection{A noncommutative deformation of $E_{(A,p,q)}^{\nabla}$}
Let us denote the pair $(X^n, \Pi_{\theta})$ by $X_{\theta}^n$: $X_{\theta}^n:=(X^n, \Pi_{\theta})$, namely, $X_{\theta}^n$ is a Poisson manifold. This subsection is essentially devoted to explain the noncommutative deformation of holomorphic line bundles $E_{(A,p,q)}^{\nabla}$ which is constructed in \cite{nc} associated to the deformation from $X^n$ to $X_{\theta}^n$. However, in this paper, we do not assume that $X^n$ is the standard complex torus although the standard complex torus $\mathbb{C}^n/(\mathbb{Z}^n\oplus \mathbf{i}\mathbb{Z}^n)$ is only treated in \cite{nc}\footnote{Precisely speaking, our discussions are mainly based on section 4 in \cite{nc}, and there, the standard complex structure is only treated.}.

We first recall the formal deformation quantization of $X_{\theta}^n$ and nonformal deformation quantizations of $X_{\theta}^n$ briefly. We consider the following formal power series ring such that its formal parameter is $\hbar$:
\begin{equation*}
C^{\infty}(X^n)[[\hbar]]:=\left\{ \sum_{i=0}^{\infty}f_i \hbar^i \ \Bigl. \Bigr| \ f_0, \ f_1, \ \cdots \in C^{\infty}(X^n) \right\}.
\end{equation*}
Then, the formal deformation quantization of the Poisson manifold $X_{\theta}^n$ is a pair (a noncommutative algebra) $(C^{\infty}(X^n)[[\hbar]], \star)$, where $\star : C^{\infty}(X^n)[[\hbar]]\times C^{\infty}(X^n)[[\hbar]] \to C^{\infty}(X^n)[[\hbar]]$ is an $\hbar$-bilinear associative product satisfying the following: when we denote 
\begin{equation*}
f\star g=\sum_{i=0}^{\infty}C_i(f,g)\hbar^i
\end{equation*}
for $f$, $g\in C^{\infty}(X^n)$, the relations
\begin{equation*}
C_0(f,g)=fg, \ \ \ C_1(f,g)=\frac{\hbar}{2\pi}\{ f, g \}
\end{equation*}
hold. Here, $\{ \cdot, \cdot \} : C^{\infty}(X^n)\times C^{\infty}(X^n) \to C^{\infty}(X^n)$ is the Poisson bracket which is defined by the Poisson bivector $\Pi_{\theta}$, and this product $\star$ is called the star product. Furthermore, now, since $\theta$ is a constant matrix, we can express the star product $\star$ as the Moyal star product explicitly:
\begin{equation}
f\star g=f\mathrm{exp}\left( \frac{\hbar}{2\pi} \frac{1}{2}\sum_{i,j=1}^n \theta_{ij} \overleftarrow{\frac{\partial}{\partial y_i}}\overrightarrow{\frac{\partial}{\partial y_j}} \right)g, \label{formal_moyal}
\end{equation}
where $f$, $g\in C^{\infty}(X^n)$. Note that the Moyal star product (\ref{formal_moyal}) is defined globally since the Poisson bivector $\Pi_{\theta}$ associated to the constant matrix $\theta$ is defined globally. On the other hand, the models which are obtained by substituting $\mathbf{i}$ or $-\mathbf{i}$ to $\hbar$ in the Moyal star product (\ref{formal_moyal}) correspond to typical examples of nonformal deformation quantizations of $X_{\theta}^n$. In this paper, we focus on the case $\hbar=-\mathbf{i}$ according to \cite{nc}:
\begin{equation}
f\star g=f\mathrm{exp}\left( -\frac{\mathbf{i}}{2\pi} \frac{1}{2}\sum_{i,j=1}^n \theta_{ij} \overleftarrow{\frac{\partial}{\partial y_i}}\overrightarrow{\frac{\partial}{\partial y_j}} \right)g, \label{nonformal_moyal}
\end{equation}
where $f$, $g\in C^{\infty}(X^n)$. This nonformal deformation quantization of $X_{\theta}^n$ can be regarded as a noncommutative torus such as in \cite{con-rie}. In this sense, hereafter, we interpret the Poisson manifold $X_{\theta}^n$ equipped with the noncommutative product (\ref{nonformal_moyal}) as the noncommutative torus.

In general, it is valid to use the following formula (this is sometimes called ``Bopp shifts'') when we compute the Moyal star product. We assume that $f$, $g\in C^{\infty}(X^n)$ are expressed as follows (the Taylor expansions), where $\epsilon$, $\delta \in \mathbb{R}^n$ are suitable constant vectors:
\begin{equation*}
f(y+\epsilon)=\sum_{k=0}^{\infty} \frac{1}{k!} \left( \frac{\partial}{\partial y}^t \epsilon \right)^k f(y), \ \ \ g(y+\delta)=\sum_{l=0}^{\infty} \frac{1}{l!} \left( \frac{\partial}{\partial y}^t \delta \right)^l g(y).
\end{equation*}
Then we have
\begin{equation}
f(y)\star g(y)=f \left( y-\frac{\mathbf{i}}{4\pi}\theta \overrightarrow{\frac{\partial}{\partial y}} \right) g(y)=f(y) g \left( y+\frac{\mathbf{i}}{4\pi}\theta \overleftarrow{\frac{\partial}{\partial y}} \right) \label{bopp} 
\end{equation}
since we can regard $f(y)\star g(y)$ as
\begin{equation*}
f(y)\mathrm{exp} \left( \overleftarrow{\frac{\partial}{\partial y}^t} \left( -\frac{\mathbf{i}}{4\pi} \theta \overrightarrow{\frac{\partial}{\partial y}} \right) \right) g(y)=f(y)\mathrm{exp} \left( \overrightarrow{\frac{\partial}{\partial y}^t} \left( \frac{\mathbf{i}}{4\pi} \theta \overleftarrow{\frac{\partial}{\partial y}} \right) \right) g(y)
\end{equation*}
and $f(y+\epsilon)$, $g(y+\delta)$ can be interpreted as
\begin{equation*}
f(y+\epsilon)=\mathrm{exp} \left( \frac{\partial}{\partial y}^t \epsilon \right) f(y), \ \ \ g(y+\delta)=\mathrm{exp} \left( \frac{\partial}{\partial y}^t \delta \right) g(y),
\end{equation*}
respectively.

According to \cite{nc}, we consider the noncommutative deformation of holomorphic line bundles $E_{(A,p,q)}^{\nabla}$ associated to the deformation from $X^n$ to $X_{\theta}^n$. In general, the space $\Gamma(X^n, E_{(A,p,q)}^{\nabla})$ can be regarded as a module over $C^{\infty}(X^n)$. Therefore, it is natural to consider $\Gamma(X^n, E_{(A,p,q)}^{\nabla})$ is also quantized associated to the nonformal deformation quantization of $X_{\theta}^n$ which is discribed by the Moyal star product (\ref{nonformal_moyal}). Let us consider the following deformation of the factor of automorphy $j_A$ (see section 4 in \cite{nc}), where $i=1, \ldots, n$ (each $j_A(T\bm{e}_i, x, y)$ is preserved under this noncommutative deformation):
\begin{align*}
&j_A(\bm{e}_i, x, y) \\
&\mapsto j_{\theta,A}(\bm{e}_i, x, y):=\mathrm{exp}\left( -\pi \mathbf{i} \bm{e}_i^t A^t \theta A x \right)j_A(\bm{e}_i, x, y)=\mathrm{exp}\left( -\pi \mathbf{i} \bm{e}_i^t A^t \theta A x+2\pi \mathbf{i}\bm{e}_i^t A^t y \right), \\
&j_A(T\bm{e}_i, x, y) \\
&\mapsto j_{\theta,A}(T\bm{e}_i, x, y):=j_A(T\bm{e}_i, x, y)=1.
\end{align*}  
In fact, for each $i=1, \ldots, n$, we see
\begin{equation*}
j_{\theta,A}(\bm{e}_i, x, y)\star \mathrm{exp}\left( \pi \mathbf{i}\bm{e}_i^t A^t \theta A x \right)j_A(\bm{e}_i, x, y)^{-1}=1,
\end{equation*}
namely, the inverse of each $j_{\theta,A}(\bm{e}_i, x, y)$ with respect to the Moyal star product (\ref{nonformal_moyal}) is $\mathrm{exp}( \pi \mathbf{i}\bm{e}_i^t A^t \theta A x )j_A(\bm{e}_i, x, y)^{-1}$: 
\begin{equation*}
j_{\theta,A}(\bm{e}_i, x, y)^{-1}=\mathrm{exp}\left( \pi \mathbf{i}\bm{e}_i^t A^t \theta A x \right)j_A(\bm{e}_i, x, y)^{-1}.
\end{equation*}
Furthermore, we have
\begin{equation}
j_{\theta,A}(\bm{e}_i, x+\bm{e}_j, y)\star j_{\theta,A}(\bm{e}_j, x, y)=j_{\theta,A}(\bm{e}_j, x+\bm{e}_i, y)\star j_{\theta,A}(\bm{e}_i, x, y), \label{cocycle_theta_0}
\end{equation}
where $i$, $j=1, \ldots, n$, and this result implies that the cocycle condition is satisfied. Hence the map $j_{\theta,A}$ defines a smooth complex line bundle on $X_{\theta}^n$, so let us denote it by $E_{\theta,A}\to X_{\theta}^n$ or $E_{\theta,A}$ for short. On the other hand, as the deformation of the connection 1-form $\omega_{(A,p,q)}$, we should consider the 1-form $\omega_{\theta,(A,p,q)}$ which is locally defined by
\begin{align*}
\omega_{\theta,(A,p,q)}&=\omega_{(A,p,q)}+\pi \mathbf{i} x^t A^t \theta A dx \\
&=-2\pi \mathbf{i} \Bigl( \Bigl( x^t A^t+p^t \Bigr)+q^t T \Bigr)dy+\pi \mathbf{i} x^t A^t \theta A dx.
\end{align*}  
Actually, this 1-form $\omega_{\theta,(A,p,q)}$ satisfies the relations
\begin{align*}
&\omega_{\theta,(A,p,q)}(x+\bm{e}_i, y) \\
&=j_{\theta,A}(\bm{e}_i, x, y) \overset{\star}{\wedge} \omega_{\theta,(A,p,q)}(x, y) \overset{\star}{\wedge} j_{\theta,A}(\bm{e}_i, x, y)^{-1}+j_{\theta,A}(\bm{e}_i, x, y) \overset{\star}{\wedge} d\hspace{0.3mm} j_{\theta,A}(\bm{e}_i, x, y)^{-1}, \\
&\omega_{\theta,(A,p,q)}(x, y+\bm{e}_i) \\
&=j_{\theta,A}(T\bm{e}_i, x, y) \overset{\star}{\wedge} \omega_{\theta,(A,p,q)}(x, y) \overset{\star}{\wedge} j_{\theta,A}(T\bm{e}_i, x, y)^{-1}+j_{\theta,A}(T\bm{e}_i, x, y) \overset{\star}{\wedge} d\hspace{0.3mm} j_{\theta,A}(T\bm{e}_i, x, y)^{-1}
\end{align*}
for each $i=1, \ldots, n$. Here, we use the notation $\omega_{\theta,(A,p,q)}(x, y)$ instead of $\omega_{\theta,(A,p,q)}$ in order to specify the coordinate system $(x, y)\in \mathbb{R}^{2n}$, and $\overset{\star}{\wedge}$ is the wedge product with the Moyal star product (\ref{nonformal_moyal}). Thus, the differential operator $\nabla_{\theta,(A,p,q)}$ on $\Gamma(E_{\theta,A})$ which is defined by
\begin{equation*}
\nabla_{\theta,(A,p,q)}=d+\omega_{\theta,(A,p,q)}
\end{equation*} 
gives a connection on $E_{\theta,A}$. Let us denote a smooth complex line bundle $E_{\theta,A}\to X_{\theta}^n$ with a connection $\nabla_{\theta,(A,p,q)}$ by $E_{\theta,(A,p,q)}^{\nabla}$, i.e., $E_{\theta,(A,p,q)}^{\nabla}:=(E_{\theta,A}, \nabla_{\theta,(A,p,q)})$.

We consider the holomorphicity of $E_{\theta,(A,p,q)}^{\nabla}$. The curvature form $\Omega_{\theta,(A,p,q)}$ of the connection $\nabla_{\theta,(A,p,q)}$ is expressed locally as
\begin{align*}
\Omega_{\theta,(A,p,q)}&=d\omega_{\theta,(A,p,q)}+\omega_{\theta,(A,p,q)} \overset{\star}{\wedge} \omega_{\theta,(A,p,q)} \\
&=-2\pi \mathbf{i}dx^t A^t dy+\pi \mathbf{i} dx^t A^t \theta A dx.
\end{align*}
Now, we give the following proposition (cf. \cite[Proposition 4.1]{nc}).
\begin{proposition}
The \textup{(}0,2\textup{)}-part $\Omega_{\theta,(A,p,q)}^{(0,2)}$ vanishes if and only if $A^t \theta A=O$. 
\end{proposition}
\begin{proof}
The (0,2)-part $\Omega_{\theta,(A,p,q)}^{(0,2)}$ turns out to be
\begin{equation*}
\Omega_{\theta,(A,p,q)}^{(0,2)}=\pi \mathbf{i}d\bar{z}^t ((T-\bar{T})^{-1})^t T^t A^t \theta A T (T-\bar{T})^{-1} d\bar{z}
\end{equation*}
since we assume that $AT\in \mathrm{Sym}(n;\mathbb{C})$. Note that $\theta\in \mathrm{Alt}(n;\mathbb{R})$ and $\mathrm{det}T(T-\bar{T})^{-1}\not=0$. Hence, $((T-\bar{T})^{-1})^t T^t A^t \theta A T (T-\bar{T})^{-1}\in \mathrm{Sym}(n;\mathbb{C})$ if and only if $A^t \theta A=O$. This completes the proof.
\end{proof}
As an example, we focus on $E_{\theta,(A,p,q)}^{\nabla}$ with $\mathrm{det}A\not=0$. By the assumption $\mathrm{det}A\not=0$, it is clear that the condition $A^t \theta A=O$ is equivalent to the condition $\theta=O$. In other words, if $\theta$ is non-trivial, then the holomorpicity of $E_{(A,p,q)}^{\nabla}$ is not preserved under the noncommutative deformation associated to $\theta$. Concerning such cases, in \cite{nc}, as a deformation of $DG_{X^n}$ associated to $\theta$ (in the case $T=\mathbf{i}\cdot I_n$, i.e., $X^n=\mathbb{C}^n/(\mathbb{Z}^n\oplus \mathbf{i}\mathbb{Z}^n)$), it is proposed that considering a curved dg-category consisting of $E_{\theta,(A,p,q)}^{\nabla}$ rather than a usual dg-category.

Let us discuss the deformation of a global holomorphic section $\xi_{(A,p,q)}$ of $E_{(A,p,q)}^{\nabla}$ associated to the deformation from $X^n$ to $X_{\theta}^n$ in the case that $T=\mathbf{i}\cdot I_n$ and $A$ is positive definite (see also subsection 3.2 in \cite{star}). In conclusion, $\xi_{(A,p,q)}$ is deformed to the following, where $C_m\in \mathbb{C}$ for each $m\in \mathbb{Z}^n/A\mathbb{Z}^n$:
\begin{align*}
&\xi_{\theta,(A,p,q)}(x,y):=\sum_{m\in \mathbb{Z}^n/A\mathbb{Z}^n} C_m \sum_{l\in \mathbb{Z}^n} \xi_{\theta,(A,p,q,l,m)}(x) \ \mathrm{exp}\left( 2\pi \mathbf{i} (-Al+m)^t y \right), \\
&\xi_{\theta,(A,p,q,l,m)}(x):=\mathrm{exp}\left( -\pi \mathbf{i} (-Al+m)^t \theta Ax \right)\xi_{(A,p,q,l,m)}(x).
\end{align*}
We can verify that this $\xi_{\theta,(A,p,q)}$ satisfies the relations
\begin{align*}
&j_{\theta,A}(\bm{e}_i, x, y)\star \xi_{\theta,(A,p,q)}(x, y)=\xi_{\theta,(A,p,q)}(x+\bm{e}_i, y), \\
&j_{\theta,A}(T\bm{e}_i, x, y)\star \xi_{\theta,(A,p,q)}(x, y)=\xi_{\theta,(A,p,q)}(x, y+\bm{e}_i), 
\end{align*}
where $i=1, \ldots, n$, and note that $\xi_{\theta,(A,p,q,l,m)}\in \mathcal{S}(\mathbb{R}^n)$ for each pair $(l,m)\in \mathbb{Z}^n\times (\mathbb{Z}^n/A\mathbb{Z}^n)$. Moreover, it is clear that $\xi_{\theta,(A,p,q)}$ coincides with $\xi_{(A,p,q)}$ in the case $\theta=O$: $\xi_{\theta=O,(A,p,q)}=\xi_{(A,p,q)}$. Thus, indeed, $\xi_{\theta,(A,p,q)}$ is a global section of $E_{\theta,(A,p,q)}^{\nabla}$.

Below, as a preparation of arguments in subsection 3.2, let us twist $E_{(A,p,q)}^{\nabla}$ by using an isomorphism. Let us take an arbitrary $\mathcal{A}\in \mathrm{Sym}(n;\mathbb{R})$, and fix it. We consider the following factor of automorphy $j_{(A;\mathcal{A})} : (\mathbb{Z}^n\oplus T\mathbb{Z}^n)\times \mathbb{C}^n \cong (\mathbb{Z}^n\oplus T\mathbb{Z}^n)\times \mathbb{R}^{2n} \to \mathbb{C}^{\times}$, where $i=1, \ldots, n$:
\begin{align*}
&j_{(A;\mathcal{A})}(\bm{e}_i, x, y):=j_A(\bm{e}_i, x, y), \\
&j_{(A;\mathcal{A})}(T\bm{e}_i, x, y):=\mathrm{exp}\left( \frac{i}{2}\bm{e}_i^t \mathcal{A} \bm{e}_i+\mathbf{i}\bm{e}_i^t \mathcal{A}y \right) j_A(T\bm{e}_i, x, y).
\end{align*} 
By using this $j_{(A;\mathcal{A})}$, let us define an action of $\mathbb{Z}^n\oplus T\mathbb{Z}^n$ on $\mathbb{C}^n\times \mathbb{C}$ by
\begin{align*}
&\bigl( x+Ty, t \bigr)\in \mathbb{C}^n\times \mathbb{C} \ \longmapsto \ \bigl( (x+Ty)+\bm{e}_i, j_{(A;\mathcal{A})}(\bm{e}_i, x, y) t \bigr)\in \mathbb{C}^n\times \mathbb{C}, \\
&\bigl( x+Ty, t \bigr)\in \mathbb{C}^n\times \mathbb{C} \ \longmapsto \ \bigl( (x+Ty)+T\bm{e}_i, j_{(A;\mathcal{A})}(T\bm{e}_i, x, y) t \bigr)\in \mathbb{C}^n\times \mathbb{C}
\end{align*}
for each $i=1, \ldots, n$. Then we can obtain the smooth complex line bundle $E_{(A;\mathcal{A})} \to X^n$ as the quotient of $\mathbb{C}^n\times \mathbb{C}$ by the action of $\mathbb{Z}^n\oplus T\mathbb{Z}^n$:
\begin{equation*}
\mathrm{Tot}(E_{(A;\mathcal{A})}):=(\mathbb{C}^n\times \mathbb{C})/(\mathbb{Z}^n\oplus T\mathbb{Z}^n).
\end{equation*} 
Furthermore, the following locally defined differential operator $\nabla_{(A,p,q;\mathcal{A})}$ on $\Gamma(E_{(A;\mathcal{A})})$ is compatible with the above action of $\mathbb{Z}^n\oplus T\mathbb{Z}^n$, namely, it defines a connection on $E_{(A;\mathcal{A})}$:
\begin{equation*}
\nabla_{(A,p,q;\mathcal{A})}=d+\omega_{(A,p,q;\mathcal{A})}, \ \ \ \omega_{(A,p,q;\mathcal{A})}=\omega_{(A,p,q)}-\mathbf{i}y^t \mathcal{A}dy.
\end{equation*}
We can easily check that this $\nabla_{(A,p,q;\mathcal{A})}$ is an integrable connection on $E_{(A;\mathcal{A})}$ since $AT\in \mathrm{Sym}(n;\mathbb{C})$, $\mathcal{A}\in \mathrm{Sym}(n;\mathbb{R})$. Let us denote a holomorphic line bundle $E_{(A;\mathcal{A})} \to X^n$ with an integrable connection $\nabla_{(A,p,q;\mathcal{A})}$ by $E_{(A,p,q;\mathcal{A})}^{\nabla}$, i.e., $E_{(A,p,q;\mathcal{A})}^{\nabla}:=(E_{(A;\mathcal{A})}, \nabla_{(A,p,q;\mathcal{A})})$. 

In fact, there exists an isomorphism $\varphi_{\mathcal{A}} : E_{(O,0,0)}^{\nabla} \stackrel{\sim}{\to} E_{(O,0,0;\mathcal{A})}^{\nabla}$ which is expressed locally as
\begin{equation}
\varphi_{\mathcal{A}}(y)=\mathrm{exp}\left( \frac{\mathbf{i}}{2}y^t \mathcal{A} y \right), \label{trivialization}
\end{equation}
so we have
\begin{equation*}
E_{(A,p,q;\mathcal{A})}^{\nabla}\cong E_{(A,p,q)}^{\nabla}\otimes E_{(O,0,0;\mathcal{A})}^{\nabla}\cong E_{(A,p,q)}^{\nabla}\otimes E_{(O,0,0)}^{\nabla}\cong E_{(A,p,q)}^{\nabla}.
\end{equation*}
In this argument, note that $E_{(O,0,0)}^{\nabla}$ is the trivial holomorphic line bundle on $X^n$.

\subsection{Ambiguity of a noncommutative deformation of $E_{(A,p,q)}^{\nabla}$}
The purpose of this subsection is to study ambiguity of the noncommutative deformation of $E_{(A,p,q)}^{\nabla}$ arising from the isomorphism (\ref{trivialization}). As verified at the last of subsection 3.1, we have $E_{(A,p,q)}^{\nabla}\cong E_{(A,p,q;\mathcal{A})}^{\nabla}$ for a given $\mathcal{A}\in \mathrm{Sym}(n;\mathbb{R})$, and similarly as in the case of $E_{\theta,(A,p,q)}^{\nabla}$, we can construct the deformation $E_{\theta,(A,p,q;\mathcal{A})}^{\nabla}$ of $E_{(A,p,q;\mathcal{A})}^{\nabla}$ associated to the deformation from $X^n$ to $X_{\theta}^n$. On the other hand, in general, even though there always exists an isomorphism $E_{(A,p,q)}^{\nabla}\cong E_{(A,p,q;\mathcal{A})}^{\nabla}$, $E_{\theta,(A,p,q)}^{\nabla}$ and $E_{\theta,(A,p,q;\mathcal{A})}^{\nabla}$ are not necessarily isomorphic to each other. In this subsection, we intensively discuss this ambiguity.

Below, we assume that
\begin{equation*}
\mathrm{det}\left( I_n+\frac{1}{2\pi}\theta \mathcal{A} \right)\not=0
\end{equation*}
in order to discuss the problem explained in the above\footnote{For example, if $\theta\in M(n;\mathbb{R})\backslash M(n;\mathbb{Q})$ and $\mathcal{A}$ is described as $\mathcal{A}=2\pi M$ by using $M\in \mathrm{Sym}(n;\mathbb{Z})$, the assumption $\mathrm{det}\left( I_n+\frac{1}{2\pi}\theta \mathcal{A} \right)\not=0$ is satisfied automatically.}, and put
\begin{equation*}
\mathcal{A}^{\theta}:=\mathcal{A} \left( I_n+\frac{1}{2\pi}\theta \mathcal{A} \right)^{-1}=\left( I_n+\frac{1}{2\pi}\mathcal{A} \theta \right)^{-1} \mathcal{A}
\end{equation*}
for simplicity\footnote{Note that $\mathrm{det}\left( I_n+\frac{1}{2\pi}\mathcal{A} \theta \right)=\mathrm{det}\left( I_n+\frac{1}{2\pi}\theta \mathcal{A} \right)\not=0$.}. In this subsection, we first consider the deformation $j_{\theta,(A;\mathcal{A})}$ of each $j_{(A;\mathcal{A})}$ as a generalization of $j_{\theta,A}$ to the setting which includes the non-trivial parameter $\mathcal{A}\in \mathrm{Sym}(n;\mathbb{R})$. At least, under the assumption $\mathrm{det}\left( I_n+\frac{1}{2\pi}\theta \mathcal{A} \right)\not=0$, the deformation of each $\nabla_{(A,p,q;\mathcal{A})}$ which is compatible with $j_{\theta,(A;\mathcal{A})}$ is uniquely determined except for the flat connection part (see also Lemma \ref{connection_theta}). It seems that the situation of other cases is more complicated, so in this paper, we do not (cannot) treat such cases.

We give the strict definition of $E_{\theta,(A,p,q;\mathcal{A})}^{\nabla} \to X_{\theta}^n$ which is the quantization of $E_{(A,p,q;\mathcal{A})}^{\nabla} \to X^n$ associated to the deformation from $X^n$ to $X_{\theta}^n$. Let us define a map $j_{\theta,(A;\mathcal{A})} : (\mathbb{Z}^n\oplus T\mathbb{Z}^n)\times \mathbb{C}^n\cong (\mathbb{Z}^n\oplus T\mathbb{Z}^n)\times \mathbb{R}^{2n} \to \mathbb{C}^{\times}$ by
\begin{align*}
&j_{\theta,(A;\mathcal{A})}(\bm{e}_i, x, y)=\mathrm{exp}\left( -\pi \mathbf{i}\bm{e}_i^t A^t \theta Ax+2\pi \mathbf{i}\bm{e}_i^t A^t y \right), \\
&j_{\theta,(A;\mathcal{A})}(T\bm{e}_i, x, y)=\mathrm{exp}\left( \frac{\mathbf{i}}{2}\bm{e}_i^t \mathcal{A} \bm{e}_i-\mathbf{i}\bm{e}_i^t \mathcal{A}^{\theta} \theta Ax+\mathbf{i}\bm{e}_i^t \mathcal{A}^{\theta} y \right),  
\end{align*}
where $i=1, \ldots, n$. Clearly, $j_{\theta,(A;\mathcal{A})}$ coincides with $j_{(A;\mathcal{A})}$ in the case $\theta=O$, i.e., $j_{\theta=O,(A;\mathcal{A})}=j_{(A;\mathcal{A})}$, and $j_{\theta,(A;\mathcal{A})}$ coincides with $j_{\theta,A}$ in the case $\mathcal{A}=O$, i.e., $j_{\theta,(A;\mathcal{A}=O)}=j_{\theta,A}$. By direct calculations, we see
\begin{align*}
&j_{\theta,(A;\mathcal{A})}(\bm{e}_i, x, y)\star \mathrm{exp}\left( \pi \mathbf{i}\bm{e}_i^t A^t \theta Ax-2\pi \mathbf{i}\bm{e}_i^t A^t y \right)=1, \\
&j_{\theta,(A;\mathcal{A})}(T\bm{e}_i, x, y)\star \mathrm{exp}\left( -\frac{\mathbf{i}}{2}\bm{e}_i^t \mathcal{A} \bm{e}_i+\mathbf{i}\bm{e}_i^t \mathcal{A}^{\theta} \theta Ax-\mathbf{i}\bm{e}_i^t \mathcal{A}^{\theta} y \right)=1
\end{align*}
for each $i=1, \ldots, n$. This implies that the inverse of each $j_{\theta,(A;\mathcal{A})}(\bm{e}_i, x, y)$ and the inverse of each $j_{\theta,(A;\mathcal{A})}(T\bm{e}_i, x, y)$ with respect to the Moyal star product (\ref{nonformal_moyal}) are
\begin{equation*}
\mathrm{exp}\left( \pi \mathbf{i}\bm{e}_i^t A^t \theta Ax-2\pi \mathbf{i}\bm{e}_i^t A^t y \right)
\end{equation*}
and
\begin{equation*}
\mathrm{exp}\left( -\frac{\mathbf{i}}{2}\bm{e}_i^t \mathcal{A} \bm{e}_i+\mathbf{i}\bm{e}_i^t \mathcal{A}^{\theta} \theta Ax-\mathbf{i}\bm{e}_i^t \mathcal{A}^{\theta} y \right),
\end{equation*}
respectively:
\begin{align*}
&j_{\theta,(A;\mathcal{A})}(\bm{e}_i, x, y)^{-1}=\mathrm{exp}\left( \pi \mathbf{i}\bm{e}_i^t A^t \theta Ax-2\pi \mathbf{i}\bm{e}_i^t A^t y \right), \\
&j_{\theta,(A;\mathcal{A})}(T\bm{e}_i, x, y)^{-1}=\mathrm{exp}\left( -\frac{\mathbf{i}}{2}\bm{e}_i^t \mathcal{A} \bm{e}_i+\mathbf{i}\bm{e}_i^t \mathcal{A}^{\theta} \theta Ax-\mathbf{i}\bm{e}_i^t \mathcal{A}^{\theta} y \right).
\end{align*}
Now, we give the following lemma.
\begin{lemma} \label{cocycle_theta}
We have
\begin{align}
&j_{\theta,(A;\mathcal{A})}(\bm{e}_i, x+\bm{e}_j, y)\star j_{\theta,(A;\mathcal{A})}(\bm{e}_j, x, y)=j_{\theta,(A;\mathcal{A})}(\bm{e}_j, x+\bm{e}_i, y)\star j_{\theta,(A;\mathcal{A})}(\bm{e}_i, x, y), \label{cocycle_theta_1} \\
&j_{\theta,(A;\mathcal{A})}(\bm{e}_i, x, y+\bm{e}_j)\star j_{\theta,(A;\mathcal{A})}(T\bm{e}_j, x, y)=j_{\theta,(A;\mathcal{A})}(T\bm{e}_j, x+\bm{e}_i, y)\star j_{\theta,(A;\mathcal{A})}(\bm{e}_i, x, y), \label{cocycle_theta_2}
\end{align}
where $i$, $j=1, \ldots, n$. Moreover, the equality
\begin{equation}
j_{\theta,(A;\mathcal{A})}(T\bm{e}_i, x, y+\bm{e}_j)\star j_{\theta,(A;\mathcal{A})}(T\bm{e}_j, x, y)=j_{\theta,(A;\mathcal{A})}(T\bm{e}_j, x, y+\bm{e}_i)\star j_{\theta,(A;\mathcal{A})}(T\bm{e}_i, x, y) \label{cocycle_theta_3}
\end{equation}
holds for any $i$, $j=1, \ldots, n$ if and only if
\begin{equation*}
\left( \frac{1}{2\pi} \right)^2 \mathcal{A}^{\theta} \theta \left( \mathcal{A}^{\theta} \right)^t \in M(n;\mathbb{Z}).
\end{equation*}
\end{lemma}
\begin{proof}
We take $i$, $j=1, \ldots, n$ arbitrary, and fix them. 

The equality (\ref{cocycle_theta_1}) is the equality (\ref{cocycle_theta_0}) itself since 
\begin{equation*}
j_{\theta,(A;\mathcal{A})}(\bm{e}_1, x, y)=j_{\theta,A}(\bm{e}_1, x, y), \ \ldots, \ j_{\theta,(A;\mathcal{A})}(\bm{e}_n, x, y)=j_{\theta,A}(\bm{e}_n, x, y). 
\end{equation*}
We can prove the equality (\ref{cocycle_theta_2}) as follows. In particular, here, we also explain how to use the ``Bopp shifts'' taking its left hand side's computation as an example (we can similarly compute other cases). By the formula (\ref{bopp}), we have
\begin{align}
&j_{\theta,(A;\mathcal{A})}(\bm{e}_i, x, y+\bm{e}_j)\star j_{\theta,(A;\mathcal{A})}(T\bm{e}_j, x, y) \notag \\
&=\mathrm{exp}\Biggl( -\pi \mathbf{i}\bm{e}_i^t A^t \theta Ax+2\pi \mathbf{i}\bm{e}_i^t A^t \left( y-\frac{\mathbf{i}}{4\pi}\theta \overrightarrow{\frac{\partial}{\partial y}} \right) \Biggr) \mathrm{exp}\Biggl( \frac{\mathbf{i}}{2}\bm{e}_j^t \mathcal{A} \bm{e}_j-\mathbf{i}\bm{e}_j^t \mathcal{A}^{\theta} \theta Ax+\mathbf{i}\bm{e}_j^t \mathcal{A}^{\theta} y \Biggr), \label{bopp_example}
\end{align} 
and since 
\begin{equation*}
2\pi \mathbf{i}\bm{e}_i^t A^t \left(-\frac{\mathbf{i}}{4\pi}\theta \overrightarrow{\frac{\partial}{\partial y}} \right)=\overrightarrow{\frac{\partial}{\partial y}^t} \left( -\frac{1}{2}\theta A\bm{e}_i \right),
\end{equation*}
the local expression (\ref{bopp_example}) turns out to be
\begin{align*}
&\mathrm{exp}\Biggl( \left( -\pi \mathbf{i} \bm{e}_i^t A^t \theta Ax+2\pi \mathbf{i} \bm{e}_i^t A^t y \right)+\left( \frac{\mathbf{i}}{2}\bm{e}_j^t \mathcal{A} \bm{e}_j-\mathbf{i}\bm{e}_j^t \mathcal{A}^{\theta} \theta Ax+\mathbf{i}\bm{e}_j^t \mathcal{A}^{\theta} \left( y-\frac{1}{2}\theta A\bm{e}_i \right) \right) \Biggr) \\
&=\mathrm{exp}\Biggl( \frac{\mathbf{i}}{2}\bm{e}_i^t A^t \theta \left( \mathcal{A}^{\theta} \right)^t \bm{e}_j+\left( -\pi \mathbf{i} \bm{e}_i^t A^t \theta Ax+2\pi \mathbf{i} \bm{e}_i^t A^t y \right) \\
&\hspace{3.8mm} +\left( \frac{\mathbf{i}}{2}\bm{e}_j^t \mathcal{A} \bm{e}_j-\mathbf{i}\bm{e}_j^t \mathcal{A}^{\theta} \theta Ax+\mathbf{i}\bm{e}_j^t \mathcal{A}^{\theta} y \right) \Biggr).
\end{align*} 
On the other hand, the right hand side turns out to be
\begin{align*}
&j_{\theta,(A;\mathcal{A})}(T\bm{e}_j, x+\bm{e}_i, y)\star j_{\theta,(A;\mathcal{A})}(\bm{e}_i, x, y) \\
&=\mathrm{exp}\Biggl( \frac{\mathbf{i}}{2}\bm{e}_j^t \mathcal{A}^{\theta} \theta A\bm{e}_i +\left( \frac{\mathbf{i}}{2}\bm{e}_j^t \mathcal{A}\bm{e}_j-\mathbf{i}\bm{e}_j^t \mathcal{A}^{\theta} \theta A\bm{e}_i-\mathbf{i}\bm{e}_j^t \mathcal{A}^{\theta} \theta Ax+\mathbf{i}\bm{e}_j^t \mathcal{A}^{\theta} y \right) \\
&\hspace{3.8mm} +\left( -\pi \mathbf{i} \bm{e}_i^t A^t \theta Ax+2\pi \mathbf{i} \bm{e}_i^t A^t y \right) \Biggr) \\
&=\mathrm{exp}\Biggl( \frac{\mathbf{i}}{2}\bm{e}_i^t A^t \theta \left( \mathcal{A}^{\theta} \right)^t \bm{e}_j+\left( -\pi \mathbf{i} \bm{e}_i^t A^t \theta Ax+2\pi \mathbf{i} \bm{e}_i^t A^t y \right) \\
&\hspace{3.8mm} +\left( \frac{\mathbf{i}}{2}\bm{e}_j^t \mathcal{A} \bm{e}_j-\mathbf{i}\bm{e}_j^t \mathcal{A}^{\theta} \theta Ax+\mathbf{i}\bm{e}_j^t \mathcal{A}^{\theta} y \right) \Biggr).
\end{align*}
Therefore, the left hand side coincides with the right hand side, so we see that the equality (\ref{cocycle_theta_2}) holds.

Let us consider when the equality (\ref{cocycle_theta_3}) holds. By direct calculations, the left hand side can be rewritten to
\begin{align*}
&j_{\theta,(A;\mathcal{A})}(T\bm{e}_i, x, y+\bm{e}_j)\star j_{\theta,(A;\mathcal{A})}(T\bm{e}_j, x, y) \\
&=\mathrm{exp}\Biggl( \Biggl( \frac{\mathbf{i}}{4\pi} \bm{e}_i^t \mathcal{A}^{\theta} \theta \left( \mathcal{A}^{\theta} \right)^t \bm{e}_j+\mathbf{i} \bm{e}_i^t \mathcal{A}^{\theta} \bm{e}_j \Biggr)+\left( \frac{\mathbf{i}}{2} \bm{e}_i^t \mathcal{A} \bm{e}_i -\mathbf{i} \bm{e}_i^t \mathcal{A}^{\theta} \theta Ax+\mathbf{i} \bm{e}_i^t \mathcal{A}^{\theta} y \right) \\
&\hspace{3.8mm} +\left( \frac{\mathbf{i}}{2} \bm{e}_j^t \mathcal{A} \bm{e}_j -\mathbf{i} \bm{e}_j^t \mathcal{A}^{\theta} \theta Ax+\mathbf{i} \bm{e}_j^t \mathcal{A}^{\theta} y \right) \Biggr). 
\end{align*}
On the other hand, the right hand side can be rewritten to
\begin{align*}
&j_{\theta,(A;\mathcal{A})}(T\bm{e}_j, x, y+\bm{e}_i)\star j_{\theta,(A;\mathcal{A})}(T\bm{e}_i, x, y) \\
&=\mathrm{exp}\Biggl( \Biggl( \frac{\mathbf{i}}{4\pi} \bm{e}_j^t \mathcal{A}^{\theta} \theta \left( \mathcal{A}^{\theta} \right)^t \bm{e}_i+\mathbf{i} \bm{e}_j^t \mathcal{A}^{\theta} \bm{e}_i \Biggr)+\left( \frac{\mathbf{i}}{2} \bm{e}_j^t \mathcal{A} \bm{e}_j -\mathbf{i} \bm{e}_j^t \mathcal{A}^{\theta} \theta Ax+\mathbf{i} \bm{e}_j^t \mathcal{A}^{\theta} y \right) \\ 
&\hspace{3.8mm} +\left( \frac{\mathbf{i}}{2} \bm{e}_i^t \mathcal{A} \bm{e}_i -\mathbf{i} \bm{e}_i^t \mathcal{A}^{\theta} \theta Ax+\mathbf{i} \bm{e}_i^t \mathcal{A}^{\theta} y \right) \Biggr) \\
&=\mathrm{exp}\Biggl( \Biggl( -\frac{\mathbf{i}}{4\pi} \bm{e}_i^t \mathcal{A}^{\theta} \theta \left( \mathcal{A}^{\theta} \right)^t \bm{e}_j+\mathbf{i} \bm{e}_i^t \left( \mathcal{A}^{\theta} \right)^t \bm{e}_j \Biggr)+\left( \frac{\mathbf{i}}{2} \bm{e}_i^t \mathcal{A} \bm{e}_i -\mathbf{i} \bm{e}_i^t \mathcal{A}^{\theta} \theta Ax+\mathbf{i} \bm{e}_i^t \mathcal{A}^{\theta} y \right) \\
&\hspace{3.8mm} +\left( \frac{\mathbf{i}}{2} \bm{e}_j^t \mathcal{A} \bm{e}_j -\mathbf{i} \bm{e}_j^t \mathcal{A}^{\theta} \theta Ax+\mathbf{i} \bm{e}_j^t \mathcal{A}^{\theta} y \right) \Biggr).  
\end{align*}
Hence, the equality (\ref{cocycle_theta_3}) holds if and only if there exists an $m_{ij}\in \mathbb{Z}$ such that
\begin{equation*}
\frac{\mathbf{i}}{4\pi} \bm{e}_i^t \mathcal{A}^{\theta} \theta \left( \mathcal{A}^{\theta} \right)^t \bm{e}_j+\mathbf{i} \bm{e}_i^t \mathcal{A}^{\theta} \bm{e}_j =-\frac{\mathbf{i}}{4\pi} \bm{e}_i^t \mathcal{A}^{\theta} \theta \left( \mathcal{A}^{\theta} \right)^t \bm{e}_j+\mathbf{i} \bm{e}_i^t \left( \mathcal{A}^{\theta} \right)^t \bm{e}_j+2\pi \mathbf{i}m_{ij},
\end{equation*}
i.e., 
\begin{equation*}
\left( \frac{1}{2\pi} \right)^2 \bm{e}_i^t \mathcal{A}^{\theta} \theta \left( \mathcal{A}^{\theta} \right)^t \bm{e}_j+\frac{1}{2\pi} \bm{e}_i^t \mathcal{A}^{\theta} \bm{e}_j -\frac{1}{2\pi} \bm{e}_i^t \left( \mathcal{A}^{\theta} \right)^t \bm{e}_j=m_{ij}.
\end{equation*}
This condition is equivalent to 
\begin{equation*}
\left( \frac{1}{2\pi} \right)^2 \mathcal{A}^{\theta} \theta \left( \mathcal{A}^{\theta} \right)^t +\frac{1}{2\pi} \mathcal{A}^{\theta}-\frac{1}{2\pi} \left( \mathcal{A}^{\theta} \right)^t \in M(n;\mathbb{Z})
\end{equation*}
since how to take the pair $(i,j)$ was arbitrary. Here, note that
\begin{align*}
&\left( \frac{1}{2\pi} \right)^2 \mathcal{A}^{\theta} \theta \left( \mathcal{A}^{\theta} \right)^t +\frac{1}{2\pi} \mathcal{A}^{\theta}-\frac{1}{2\pi} \left( \mathcal{A}^{\theta} \right)^t \\
&=\left( \frac{1}{2\pi} \right)^2 \left( I_n+\frac{1}{2\pi}\mathcal{A} \theta \right)^{-1} \mathcal{A} \theta \mathcal{A} \left( \left( I_n+\frac{1}{2\pi}\mathcal{A} \theta \right)^{-1} \right)^t \\
&\hspace{3.8mm} +\frac{1}{2\pi} \left( I_n+\frac{1}{2\pi}\mathcal{A} \theta \right)^{-1}\mathcal{A} \left( I_n+\frac{1}{2\pi}\mathcal{A} \theta \right)^t \left( \left( I_n+\frac{1}{2\pi}\mathcal{A} \theta \right)^{-1} \right)^t \\
&\hspace{3.8mm} -\frac{1}{2\pi} \left( I_n+\frac{1}{2\pi}\mathcal{A} \theta \right)^{-1} \left( I_n+\frac{1}{2\pi}\mathcal{A} \theta \right) \mathcal{A} \left( \left( I_n+\frac{1}{2\pi}\mathcal{A} \theta \right)^{-1} \right)^t \\
&=\left( \frac{1}{2\pi} \right)^2 \left( I_n+\frac{1}{2\pi}\mathcal{A}\theta \right)^{-1} \left( \mathcal{A} \theta \mathcal{A}+2\pi \mathcal{A} \left( I_n+\frac{1}{2\pi}\mathcal{A}\theta \right)^t-2\pi \left( I_n+\frac{1}{2\pi}\mathcal{A}\theta \right)\mathcal{A} \right) \left( \left( I_n+\frac{1}{2\pi}\mathcal{A} \theta \right)^{-1} \right)^t \\
&=-\left( \frac{1}{2\pi} \right)^2 \mathcal{A}^{\theta} \theta \left( \mathcal{A}^{\theta} \right)^t.
\end{align*}
Thus, in fact, the equality (\ref{cocycle_theta_3}) holds if and only if 
\begin{equation*}
\left( \frac{1}{2\pi} \right)^2 \mathcal{A}^{\theta} \theta \left( \mathcal{A}^{\theta} \right)^t \in M(n;\mathbb{Z}).
\end{equation*}
\end{proof}
\begin{rem}
By the definition of $\mathcal{A}^{\theta}$, it is clear that $\left( \frac{1}{2\pi} \right)^2 \mathcal{A}^{\theta} \theta \left( \mathcal{A}^{\theta} \right)^t=O$ when $\mathcal{A}\theta \mathcal{A}=O$. On the other hand, there also exists an example of a pair $(\theta, \mathcal{A})\in \mathrm{Alt}(n;\mathbb{R})\times \mathrm{Sym}(n;\mathbb{R})$ which satisfies
\begin{equation*}
\mathrm{det}\left( I_n+\frac{1}{2\pi}\theta \mathcal{A} \right)\not=0, \ \ \ \left( \frac{1}{2\pi} \right)^2 \mathcal{A}^{\theta} \theta \left( \mathcal{A}^{\theta} \right)^t \in M(n;\mathbb{Z})\backslash \{ O \}.
\end{equation*}
One of such examples is the following. We consider the case $n=2$, and let us take
\begin{equation*}
\theta:=\frac{1+\sqrt{1+m^2}}{2m}\left( \begin{array}{cc} 0 & 1 \\ -1 & 0 \end{array} \right)\in \mathrm{Alt}(2;\mathbb{R}), \ \ \ \mathcal{A}:=4\pi \left( \begin{array}{cc} 0 & 1 \\ 1 & 0 \end{array} \right)\in \mathrm{Sym}(2;\mathbb{R}),
\end{equation*} 
where $m\in \mathbb{Z}\backslash \{ 0 \}$. Then we have
\begin{align*}
&\mathrm{det}\left( I_2+\frac{1}{2\pi}\theta \mathcal{A} \right)=\frac{-2-2\sqrt{1+m^2}}{m^2}\not=0, \\
&\left( \frac{1}{2\pi} \right)^2 \mathcal{A}^{\theta} \theta \left( \mathcal{A}^{\theta} \right)^t=\left( \begin{array}{cc} 0 & m \\ -m & 0 \end{array} \right) \in M(2;\mathbb{Z})\backslash \{ O \}.
\end{align*}
\end{rem}
Frankly speaking, Lemma \ref{cocycle_theta} states that the map $j_{\theta,(A;\mathcal{A})}$ defines a factor of automorphy if and only if 
\begin{equation*}
\left( \frac{1}{2\pi} \right)^2 \mathcal{A}^{\theta} \theta \left( \mathcal{A}^{\theta} \right)^t \in M(n;\mathbb{Z}).
\end{equation*}
Hereafter, we assume that
\begin{equation}
\left( \frac{1}{2\pi} \right)^2 \mathcal{A}^{\theta} \theta \left( \mathcal{A}^{\theta} \right)^t \in M(n;\mathbb{Z}). \label{cocycle_theta_mat}
\end{equation}
By using $j_{\theta,(A;\mathcal{A})}$, let us define an action of $\mathbb{Z}^n\oplus T\mathbb{Z}^n$ on $\mathbb{C}^n\times \mathbb{C}$ by
\begin{align*}
&\bigl( x+Ty, t \bigr)\in \mathbb{C}^n\times \mathbb{C} \ \longmapsto \ \bigl( (x+Ty)+\bm{e}_i, j_{\theta,(A;\mathcal{A})}(\bm{e}_i, x, y) t \bigr)\in \mathbb{C}^n\times \mathbb{C}, \\
&\bigl( x+Ty, t \bigr)\in \mathbb{C}^n\times \mathbb{C} \ \longmapsto \ \bigl( (x+Ty)+T\bm{e}_i, j_{\theta,(A;\mathcal{A})}(T\bm{e}_i, x, y) t \bigr)\in \mathbb{C}^n\times \mathbb{C}
\end{align*}
for each $i=1, \ldots, n$. Then we can obtain the smooth complex line bundle $E_{\theta,(A;\mathcal{A})} \to X_{\theta}^n$ as the quotient of $\mathbb{C}^n\times \mathbb{C}$ by the action of $\mathbb{Z}^n\oplus T\mathbb{Z}^n$:
\begin{equation*}
\mathrm{Tot}(E_{\theta,(A;\mathcal{A})}):=(\mathbb{C}^n\times \mathbb{C})/(\mathbb{Z}^n\oplus T\mathbb{Z}^n).
\end{equation*} 
We locally define a differential operator $\nabla_{\theta,(A,p,q;\mathcal{A})}$ on $\Gamma(E_{\theta,(A;\mathcal{A})})$ by
\begin{align*}
&\nabla_{\theta,(A,p,q;\mathcal{A})}=d+\omega_{\theta,(A,p,q;\mathcal{A})}, \\
&\omega_{\theta,(A,p,q;\mathcal{A})}=\pi \mathbf{i}x^t A^t \left( I_n+\frac{1}{\pi}\theta \mathcal{A} \right) \theta Adx-2\pi \mathbf{i}x^t A^t \left( I_n+\frac{1}{2\pi}\theta \mathcal{A} \right)dy+\mathbf{i}y^t \mathcal{A} \theta Adx -\mathbf{i}y^t \mathcal{A}dy.
\end{align*}
Similarly as in subsection 3.1, we sometimes use the notation $\omega_{\theta,(A,p,q;\mathcal{A})}(x, y)$ in order to specify the coordinate system $(x, y)\in \mathbb{R}^{2n}$. Then the following lemma holds.
\begin{lemma} \label{connection_theta}
The factor of automorphy $j_{\theta,(A;\mathcal{A})}$ and the 1-form $\omega_{\theta,(A,p,q;\mathcal{A})}$ satisfy the relations
\begin{align}
&\omega_{\theta,(A,p,q;\mathcal{A})}(x+\bm{e}_i, y)=j_{\theta,(A;\mathcal{A})}(\bm{e}_i, x, y) \overset{\star}{\wedge} \omega_{\theta,(A,p,q;\mathcal{A})}(x, y) \overset{\star}{\wedge} j_{\theta,(A;\mathcal{A})}(\bm{e}_i, x, y)^{-1} \notag \\
&\hspace{35.7mm} +j_{\theta,(A;\mathcal{A})}(\bm{e}_i, x, y) \overset{\star}{\wedge} d\hspace{0.3mm} j_{\theta,(A;\mathcal{A})}(\bm{e}_i, x, y)^{-1}, \label{connection_theta_1} \\
&\omega_{\theta,(A,p,q;\mathcal{A})}(x, y+\bm{e}_i)=j_{\theta,(A;\mathcal{A})}(T\bm{e}_i, x, y) \overset{\star}{\wedge} \omega_{\theta,(A,p,q;\mathcal{A})}(x, y) \overset{\star}{\wedge} j_{\theta,(A;\mathcal{A})}(T\bm{e}_i, x, y)^{-1} \notag \\
&\hspace{35.7mm} +j_{\theta,(A;\mathcal{A})}(T\bm{e}_i, x, y) \overset{\star}{\wedge} d\hspace{0.3mm} j_{\theta,(A;\mathcal{A})}(T\bm{e}_i, x, y)^{-1}, \label{connection_theta_2}
\end{align}
where $i=1, \ldots, n$.
\end{lemma}
\begin{proof}
We take an arbitrary $i=1, \ldots, n$, and fix it. Let us verify that the relation (\ref{connection_theta_1}) is satisfied. By direct calculations, we have 
\begin{align*}
&j_{\theta,(A;\mathcal{A})}(\bm{e}_i, x, y) \overset{\star}{\wedge} \omega_{\theta,(A,p,q;\mathcal{A})}(x, y) \overset{\star}{\wedge} j_{\theta,(A;\mathcal{A})}(\bm{e}_i, x, y)^{-1}+j_{\theta,(A;\mathcal{A})}(\bm{e}_i, x, y) \overset{\star}{\wedge} d\hspace{0.3mm} j_{\theta,(A;\mathcal{A})}(\bm{e}_i, x, y)^{-1} \\
&=\omega_{\theta,(A,p,q;\mathcal{A})}(x, y)+\mathbf{i}\bm{e}_i^t A^t \theta \mathcal{A} \theta Adx-\mathbf{i}\bm{e}_i^t A^t \theta \mathcal{A}dy+\pi \mathbf{i}\bm{e}_i^t A^t \theta Adx-2\pi \mathbf{i} \bm{e}_i^t A^t dy \\
&=\omega_{\theta,(A,p,q;\mathcal{A})}(x, y)+\pi \mathbf{i}\bm{e}_i^t A^t \left( I_n+\frac{1}{\pi}\theta \mathcal{A} \right)\theta Adx-2\pi \mathbf{i}\bm{e}_i^t A^t \left( I_n+\frac{1}{2\pi}\theta \mathcal{A} \right)dy \\
&=\omega_{\theta,(A,p,q;\mathcal{A})}(x+\bm{e}_i, y),
\end{align*}
so indeed, the relation (\ref{connection_theta_1}) is satisfied. Similarly, we can calculate as follows:
\begin{align*}
&j_{\theta,(A;\mathcal{A})}(T\bm{e}_i, x, y) \overset{\star}{\wedge} \omega_{\theta,(A,p,q;\mathcal{A})}(x, y) \overset{\star}{\wedge} j_{\theta,(A;\mathcal{A})}(T\bm{e}_i, x, y)^{-1}+j_{\theta,(A;\mathcal{A})}(T\bm{e}_i, x, y) \overset{\star}{\wedge} d\hspace{0.3mm} j_{\theta,(A;\mathcal{A})}(T\bm{e}_i, x, y)^{-1} \\
&=\omega_{\theta,(A,p,q;\mathcal{A})}(x, y)+\frac{\mathbf{i}}{2\pi}\bm{e}_i^t \mathcal{A} \left( I_n+\frac{1}{2\pi}\theta \mathcal{A} \right)^{-1}\theta \mathcal{A} \theta Adx-\frac{\mathbf{i}}{2\pi}\bm{e}_i^t \mathcal{A} \left( I_n+\frac{1}{2\pi}\theta \mathcal{A} \right)^{-1}\theta \mathcal{A}dy \\
&\hspace{3.8mm} +\mathbf{i}\bm{e}_i^t \mathcal{A} \left( I_n+\frac{1}{2\pi}\theta \mathcal{A} \right)^{-1}\theta Adx-\mathbf{i}\bm{e}_i^t \mathcal{A}\left( I_n+\frac{1}{2\pi}\theta \mathcal{A} \right)^{-1}dy \\
&=\omega_{\theta,(A,p,q;\mathcal{A})}(x, y)+\mathbf{i}\bm{e}_i^t \mathcal{A}\left( I_n+\frac{1}{2\pi} \theta \mathcal{A} \right)^{-1} \left( I_n+\frac{1}{2\pi} \theta \mathcal{A} \right) \theta Adx \\
&\hspace{3.8mm} -\mathbf{i}\bm{e}_i^t \mathcal{A}\left( I_n+\frac{1}{2\pi} \theta \mathcal{A} \right)^{-1} \left( I_n+\frac{1}{2\pi} \theta \mathcal{A} \right)dy \\
&=\omega_{\theta,(A,p,q;\mathcal{A})}(x, y)+\mathbf{i}\bm{e}_i^t \mathcal{A} \theta Adx-\mathbf{i} \bm{e}_i^t \mathcal{A} dy \\
&=\omega_{\theta,(A,p,q;\mathcal{A})}(x, y+\bm{e}_i). 
\end{align*}
Thus, the relation (\ref{connection_theta_2}) is also satisfied. 
\end{proof}
By Lemma \ref{connection_theta}, we see that the differential operator $\nabla_{\theta,(A,p,q;\mathcal{A})}$ is compatible with the factor of automorphy $j_{\theta,(A;\mathcal{A})}$, namely, the differential operator $\nabla_{\theta,(A,p,q;\mathcal{A})}$ defines a connection on $E_{\theta,(A;\mathcal{A})}$. Then, although we can calculate the curvature form $\Omega_{\theta,(A,p,q;\mathcal{A})}$ of $\nabla_{\theta,(A,p,q;\mathcal{A})}$ locally, it is a little complicated. Explicitly, we have the following.
\begin{lemma} \label{curvature_theta}
The curvature form $\Omega_{\theta,(A,p,q;\mathcal{A})}$ of $\nabla_{\theta,(A,p,q;\mathcal{A})}$ is expressed locally as
\begin{equation*}
\Omega_{\theta,(A,p,q;\mathcal{A})}=\pi \mathbf{i}dx^t A^t \left( I_n-\left( \frac{1}{2\pi}\theta \mathcal{A} \right)^2 \right) \theta Adx-2\pi \mathbf{i} dx^t A^t \left( I_n-\left( \frac{1}{2\pi}\theta \mathcal{A} \right)^2 \right)dy+\frac{\mathbf{i}}{4\pi}dy^t \mathcal{A} \theta \mathcal{A} dy.
\end{equation*} 
\end{lemma}
\begin{proof}
We can check that
\begin{align*}
d\omega_{\theta,(A,p,q;\mathcal{A})}&=\pi \mathbf{i}dx^t A^t \left( I_n+\frac{1}{\pi}\theta \mathcal{A} \right) \theta Adx-2\pi \mathbf{i}dx^t A^t \left( I_n+\frac{1}{2\pi}\theta \mathcal{A} \right)dy+\mathbf{i}dy^t \mathcal{A} \theta Adx -\mathbf{i}dy^t \mathcal{A}dy \\
&=\pi \mathbf{i}dx^t A^t \theta Adx-2\pi \mathbf{i}dx^t A^t \left( I_n+\frac{1}{2\pi}\theta \mathcal{A} \right)dy+\mathbf{i}dx^t A^t \theta \mathcal{A}dy.
\end{align*}
Here, note that $dx^t A^t \theta^t \mathcal{A} \theta Adx$ and $dy^t \mathcal{A}dy$ vanish since $A^t \theta^t \mathcal{A} \theta A$, $\mathcal{A}\in \mathrm{Sym}(n;\mathbb{R})$. Moreover, we have
\begin{equation*}
\omega_{\theta,(A,p,q;\mathcal{A})}\overset{\star}{\wedge} \omega_{\theta,(A,p,q;\mathcal{A})}=-\frac{\mathbf{i}}{4\pi}dx^t A^t \theta \mathcal{A} \theta \mathcal{A} \theta Adx+\frac{\mathbf{i}}{2\pi}dx^t A^t \theta \mathcal{A} \theta \mathcal{A} dy+\frac{\mathbf{i}}{4\pi}dy^t \mathcal{A} \theta \mathcal{A}dy.
\end{equation*}
Thus, 
\begin{equation*}
\Omega_{\theta,(A,p,q;\mathcal{A})}=d\omega_{\theta,(A,p,q;\mathcal{A})}+\omega_{\theta,(A,p,q;\mathcal{A})}\overset{\star}{\wedge} \omega_{\theta,(A,p,q;\mathcal{A})}
\end{equation*}
can be calculated as follows:
\begin{align*}
\Omega_{\theta,(A,p,q;\mathcal{A})}&=\pi \mathbf{i}dx^t A^t \theta Adx-\frac{\mathbf{i}}{4\pi}dx^t A^t \theta \mathcal{A} \theta \mathcal{A} \theta Adx \\
&\hspace{3.8mm} -2\pi \mathbf{i}dx^t A^t \left( I_n+\frac{1}{2\pi}\theta \mathcal{A} \right)dy+\mathbf{i}dx^t A^t \theta \mathcal{A}dy+\frac{\mathbf{i}}{2\pi}dx^t A^t \theta \mathcal{A} \theta \mathcal{A} dy \\
&\hspace{3.8mm} +\frac{\mathbf{i}}{4\pi}dy^t \mathcal{A} \theta \mathcal{A}dy \\
&=\pi \mathbf{i}dx^t A^t \left( I_n-\left( \frac{1}{2\pi} \theta \mathcal{A} \right)^2 \right) \theta Adx \\
&\hspace{3.8mm} -2\pi \mathbf{i}dx^t A^t \left( I_n+\frac{1}{2\pi}\theta \mathcal{A} \right)dy+\mathbf{i}dx^t A^t \left( I_n+\frac{1}{2\pi}\theta \mathcal{A} \right) \theta \mathcal{A} dy \\
&\hspace{3.8mm} +\frac{\mathbf{i}}{4\pi}dy^t \mathcal{A} \theta \mathcal{A}dy \\
&=\pi \mathbf{i}dx^t A^t \left( I_n-\left( \frac{1}{2\pi} \theta \mathcal{A} \right)^2 \right) \theta Adx-2\pi \mathbf{i}dx^t A^t \left( I_n+\frac{1}{2\pi}\theta \mathcal{A} \right) \left( I_n-\frac{1}{2\pi}\theta \mathcal{A} \right) dy \\
&\hspace{3.8mm} +\frac{\mathbf{i}}{4\pi}dy^t \mathcal{A} \theta \mathcal{A}dy \\
&=\pi \mathbf{i}dx^t A^t \left( I_n-\left( \frac{1}{2\pi} \theta \mathcal{A} \right)^2 \right) \theta Adx-2\pi \mathbf{i}dx^t A^t \left( I_n-\left( \frac{1}{2\pi}\theta \mathcal{A} \right)^2 \right) dy \\
&\hspace{3.8mm} +\frac{\mathbf{i}}{4\pi}dy^t \mathcal{A} \theta \mathcal{A}dy. 
\end{align*}
This completes the proof.
\end{proof}
We denote a smooth complex line bundle $E_{\theta,(A;\mathcal{A})} \to X_{\theta}^n$ with a connection $\nabla_{\theta,(A,p,q;\mathcal{A})}$ by $E_{\theta,(A,p,q;\mathcal{A})}^{\nabla}$, i.e., $E_{\theta,(A,p,q;\mathcal{A})}^{\nabla}:=(E_{\theta,(A;\mathcal{A})}, \nabla_{\theta,(A,p,q;\mathcal{A})})$. 

For later convenience, we would like to explain the definition of $E_{\theta,(A,p,q;\mathcal{A})}^{\nabla}$ and $E_{\theta,(B,u,v;\mathcal{A})}^{\nabla}$ being isomorphic to each other. 

Before explaining it, let us first recall the commutative case. Let $M$ be a complex manifold and $\{ U_i \}_{i\in I}$ an open covering of $M$. We take two holomorphic vector bundles $\mathcal{E}^{\nabla}=( \{ \varphi_{ij} \}_{i,j\in I}, \{ \bar{\partial}+\omega_i^{(0,1)} \}_{i\in I})$, $\mathcal{F}^{\nabla}=( \{ \psi_{ij} \}_{i,j\in I}, \{ \bar{\partial}+\xi_i^{(0,1)} \}_{i\in I})$ on $M$, where $\{ \varphi_{ij} \}_{i,j\in I}$, $\{ \psi_{ij} \}_{i,j\in I}$ are families of transition functions, and $\{ \bar{\partial}+\omega_i^{(0,1)} \}_{i\in I}$, $\{ \bar{\partial}+\xi_i^{(0,1)} \}_{i\in I}$ are holomorphic structures. Then we say that $\mathcal{E}^{\nabla}$ is isomorphic to $\mathcal{F}^{\nabla}$ if there exists a family of invertible morphisms (invertible locally defined smooth functions) $\{ \phi_i : \mathcal{E}^{\nabla}|_{U_i} \to \mathcal{F}^{\nabla}|_{U_i} \}_{i\in I}$ satisfying the following conditions. Firstly, the equality $\phi_j \varphi_{ij}=\psi_{ij} \phi_i$ holds for any $i$, $j=1, \ldots, n$. Secondly, the relation $\bar{\partial}\phi_i+\xi_i^{(0,1)}\wedge \phi_i-\phi_i\wedge \omega_i^{(0,1)}=0$ is satisfied for each $i=1, \ldots, n$. 

In light of this definition, if there exists an invertible locally defined smooth function $\varphi$ satisfying the conditions
\begin{align}
&j_{\theta,(B;\mathcal{A})}(\bm{e}_i, x, y)\star \varphi(x, y)=\varphi(x+\bm{e}_i, y)\star j_{\theta,(A;\mathcal{A})}(\bm{e}_i, x, y), \notag \\
&j_{\theta,(B;\mathcal{A})}(T\bm{e}_i, x, y)\star \varphi(x, y)=\varphi(x, y+\bm{e}_i)\star j_{\theta,(A;\mathcal{A})}(T\bm{e}_i, x, y), \notag \\
&\bar{\partial}\varphi+\omega_{\theta,(B,u,v;\mathcal{A})}^{(0,1)} \overset{\star}{\wedge} \varphi-\varphi \overset{\star}{\wedge} \omega_{\theta,(A,p,q;\mathcal{A})}^{(0,1)}=0, \label{kernel}
\end{align} 
where $i=1, \ldots, n$, we say that $E_{\theta,(A,p,q;\mathcal{A})}^{\nabla}$ is isomorphic to $E_{\theta,(B,u,v;\mathcal{A})}^{\nabla}$, and write
\begin{equation*}
E_{\theta,(A,p,q;\mathcal{A})}^{\nabla}\cong E_{\theta,(B,u,v;\mathcal{A})}^{\nabla}.
\end{equation*} 
In particular, under this situation, $\varphi$ is called an isomorphism from $E_{\theta,(A,p,q;\mathcal{A})}^{\nabla}$ to $E_{\theta,(B,u,v;\mathcal{A})}^{\nabla}$: 
\begin{equation*}
\varphi : E_{\theta,(A,p,q;\mathcal{A})}^{\nabla} \stackrel{\sim}{\to} E_{\theta,(B,u,v;\mathcal{A})}^{\nabla}.
\end{equation*}

Here, we give a remark on this definition. As discussed in \cite{nc} (see also \cite{kajiura}), we can actually construct a curved dg-category consisting of noncommutative objects $E_{\theta,(A,p,q;\mathcal{A})}^{\nabla}$ as a deformation of the dg-category $DG_{X^n}$. Essentially, the curvature in this context is given by $\Omega_{\theta,(A,p,q;\mathcal{A})}^{(0,2)}$, and in general, $\Omega_{\theta,(A,p,q;\mathcal{A})}^{(0,2)}$ does not necessarily vanish, i.e., the holomorphicity of $E_{(A,p,q)}^{\nabla}$ does not necessarily preserved under the noncommutative deformation from $X^n$ to $X_{\theta}^n$. On the other hand, for a given curved dg-category $\mathscr{C}$, we can consider a dg-category consisting of objects which have the same (fixed) curvature as a full subcategory of $\mathscr{C}$. For this dg-category, we can take the cohomology of each space of morphisms. We now assume that $E_{\theta,(A,p,q;\mathcal{A})}^{\nabla}$ and $E_{\theta,(B,u,v;\mathcal{A})}^{\nabla}$ have the same curvature. Let us denote the 0-th cohomology of the space of morphisms from $E_{\theta,(A,p,q;\mathcal{A})}^{\nabla}$ to $E_{\theta,(B,u,v;\mathcal{A})}^{\nabla}$ in this context\footnote{In fact, the spaces of morphisms in this context are obtained by replacing the usual product in the spaces of morphisms of $DG_{X^n}$ with the Moyal star product (\ref{nonformal_moyal}). See \cite{kajiura} for details.} by $H^0(E_{\theta,(A,p,q;\mathcal{A})}^{\nabla}, E_{\theta,(B,u,v;\mathcal{A})}^{\nabla})$. Then a non-trivial representative of $H^0(E_{\theta,(A,p,q;\mathcal{A})}^{\nabla}, E_{\theta,(B,u,v;\mathcal{A})}^{\nabla})$ gives an isomorphism $E_{\theta,(A,p,q;\mathcal{A})}^{\nabla} \stackrel{\sim}{\to} E_{\theta,(B,u,v;\mathcal{A})}^{\nabla}$ (the condition (\ref{kernel}) indicates that this representative belongs to the kernel of the corresponding differential). 

As remarked at the beginning of this subsection, in general, $E_{\theta,(A,p,q)}^{\nabla}\not\cong E_{\theta,(A,p,q;\mathcal{A})}^{\nabla}$ even though $E_{(A,p,q)}^{\nabla}\cong E_{(A,p,q;\mathcal{A})}^{\nabla}$. However, exceptionally, there exists a case that the relation $E_{(A,p,q)}^{\nabla}\cong E_{(A,p,q;\mathcal{A})}^{\nabla}$ is preserved under the deformation from $X^n$ to $X_{\theta}^n$. We now explain this example. Let us take $(A,p,q)\in M(n;\mathbb{Z}) \times \mathbb{R}^n \times \mathbb{R}^n$, $\mathcal{A}\in \mathrm{Sym}(n;\mathbb{R})$, and fix them. Then, if $\mathcal{A}\theta \mathcal{A}=O$, we see that $\Omega_{\theta,(A,p,q;\mathcal{A})}$ coincides with $\Omega_{\theta,(A,p,q)}$ by Lemma \ref{curvature_theta}. Hence, it is expected that $E_{\theta,(A,p,q;\mathcal{A})}^{\nabla}$ and $E_{\theta,(A,p,q)}^{\nabla}$ are isomorphic to each other in this case. Actually, this expectation is correct as shown in the following proposition (cf. Remark \ref{main_proposition_analogue}). 
\begin{proposition} \label{main_proposition}
Assume that $\mathcal{A}\theta \mathcal{A}=O$. Then we have $E_{\theta,(A,p,q)}^{\nabla}\cong E_{\theta,(A,p,q;\mathcal{A})}^{\nabla}$, where an isomorphism $\varphi_{\theta, \mathcal{A}} : E_{\theta,(A,p,q)}^{\nabla} \stackrel{\sim}{\to} E_{\theta,(A,p,q;\mathcal{A})}^{\nabla}$ is expressed locally as
\begin{equation*}
\varphi_{\theta, \mathcal{A}}(x, y)=\mathrm{exp}\left( \frac{\mathbf{i}}{2}x^t A^t \theta^t \mathcal{A} \theta A x+\frac{\mathbf{i}}{2}y^t \mathcal{A} y+\mathbf{i} x^t A^t \theta \mathcal{A} y \right).
\end{equation*}  
\end{proposition}
\begin{proof}
We may prove that the invertible locally defined smooth function
\begin{equation*}
\varphi_{\theta, \mathcal{A}}(x, y):=\mathrm{exp}\left( \frac{\mathbf{i}}{2}x^t A^t \theta^t \mathcal{A} \theta A x+\frac{\mathbf{i}}{2}y^t \mathcal{A} y+\mathbf{i} x^t A^t \theta \mathcal{A} y \right)
\end{equation*}
satisfies the relations
\begin{align}
&j_{\theta,(A;\mathcal{A})}(\bm{e}_i, x, y)\star \varphi_{\theta, \mathcal{A}}(x, y)=\varphi_{\theta, \mathcal{A}}(x+\bm{e}_i, y)\star j_{\theta,A}(\bm{e}_i, x, y), \label{main_proposition_1} \\
&j_{\theta,(A;\mathcal{A})}(T\bm{e}_i, x, y)\star \varphi_{\theta, \mathcal{A}}(x, y)=\varphi_{\theta, \mathcal{A}}(x, y+\bm{e}_i)\star j_{\theta,A}(T\bm{e}_i, x, y), \label{main_proposition_2} \\
&\bar{\partial}\varphi_{\theta, \mathcal{A}}+\omega_{\theta,(A,p,q;\mathcal{A})}^{(0,1)} \overset{\star}{\wedge} \varphi_{\theta, \mathcal{A}}-\varphi_{\theta, \mathcal{A}} \overset{\star}{\wedge} \omega_{\theta,(A,p,q)}^{(0,1)}=0, \label{main_proposition_3}
\end{align}
where $i=1, \ldots, n$. Let us take an arbitrary $i=1, \ldots, n$, and fix it.

First, we consider the equality (\ref{main_proposition_1}). By the definitions of $j_{\theta,A}$ and $j_{\theta,(A;\mathcal{A})}$, it is enough to prove that the equality
\begin{equation*}
\mathrm{exp}\left( 2\pi \mathbf{i} \bm{e}_i^t A^t y \right)\star \varphi_{\theta, \mathcal{A}}(x, y)=\varphi_{\theta, \mathcal{A}}(x+\bm{e}_i, y)\star \mathrm{exp}\left( 2 \pi \mathbf{i} \bm{e}_i^t A^t y \right)
\end{equation*}
holds. The left hand side turns out to be
\begin{align*}
&\mathrm{exp}\left( 2\pi \mathbf{i} \bm{e}_i^t A^t y \right)\star \varphi_{\theta, \mathcal{A}}(x, y) \\
&=\mathrm{exp}\left( 2\pi \mathbf{i} \bm{e}_i^t A^t y+\left( \frac{\mathbf{i}}{2}x^t A^t \theta^t \mathcal{A} \theta A x+\frac{\mathbf{i}}{2}\left( y-\frac{1}{2}\theta A \bm{e}_i \right)^t \mathcal{A} \left( y-\frac{1}{2}\theta A \bm{e}_i \right)+\mathbf{i} x^t A^t \theta \mathcal{A} \left( y-\frac{1}{2}\theta A \bm{e}_i \right) \right) \right) \\
&=\mathrm{exp}\Biggl( 2\pi \mathbf{i} \bm{e}_i^t A^t y+\left( \frac{\mathbf{i}}{8}\bm{e}_i^t A^t \theta^t \mathcal{A} \theta A \bm{e}_i+\frac{\mathbf{i}}{2}\bm{e}_i^t A^t \theta^t \mathcal{A} \theta A x+\frac{\mathbf{i}}{2}\bm{e}_i^t A^t \theta \mathcal{A} y \right) \\
&\hspace{3.8mm} +\left( \frac{\mathbf{i}}{2}x^t A^t \theta^t \mathcal{A} \theta A x+\frac{\mathbf{i}}{2}y^t \mathcal{A} y+\mathbf{i} x^t A^t \theta \mathcal{A} y \right) \Biggr).
\end{align*}
On the other hand, the right hand side turns out to be
\begin{align*}
&\varphi_{\theta, \mathcal{A}}(x+\bm{e}_i, y)\star \mathrm{exp}\left( 2 \pi \mathbf{i} \bm{e}_i^t A^t y \right) \\
&=\mathrm{exp}\left( \left( \frac{\mathbf{i}}{2}\bm{e}_i^t A^t \theta^t \mathcal{A} \theta A \bm{e}_i+\mathbf{i}\bm{e}_i^t A^t \theta^t \mathcal{A} \theta A x+\mathbf{i}\bm{e}_i^t A^t \theta \mathcal{A} y \right)+\left( \frac{\mathbf{i}}{2}x^t A^t \theta^t \mathcal{A} \theta A x+\frac{\mathbf{i}}{2}y^t \mathcal{A} y+\mathbf{i} x^t A^t \theta \mathcal{A} y \right) \right) \\
&\hspace{3.8mm} \star \mathrm{exp}\left( 2 \pi \mathbf{i} \bm{e}_i^t A^t y \right) \\
&=\mathrm{exp}\Biggl( \left( \mathbf{i}\bm{e}_i^t A^t \theta \mathcal{A} \left( y+\frac{1}{2}\theta A\bm{e}_i \right)+\frac{\mathbf{i}}{2}\left( y+\frac{1}{2}\theta A \bm{e}_i \right)^t \mathcal{A} \left( y+\frac{1}{2}\theta A \bm{e}_i \right)+\mathbf{i} x^t A^t \theta \mathcal{A} \left( y+\frac{1}{2}\theta A \bm{e}_i \right) \right) \\
&\hspace{3.8mm} +\left( \frac{\mathbf{i}}{2}\bm{e}_i^t A^t \theta^t \mathcal{A} \theta A \bm{e}_i +\mathbf{i}\bm{e}_i^t A^t \theta^t \mathcal{A} \theta A x+\frac{\mathbf{i}}{2} x^t A^t \theta^t \mathcal{A} \theta A x \right)+2\pi \mathbf{i}\bm{e}_i^t A^t y \Biggr) \\
&=\mathrm{exp}\Biggl( 2\pi \mathbf{i} \bm{e}_i^t A^t y+\left( \frac{\mathbf{i}}{8}\bm{e}_i^t A^t \theta^t \mathcal{A} \theta A \bm{e}_i+\frac{\mathbf{i}}{2}\bm{e}_i^t A^t \theta^t \mathcal{A} \theta A x+\frac{\mathbf{i}}{2}\bm{e}_i^t A^t \theta \mathcal{A} y \right) \\
&\hspace{3.8mm} +\left( \frac{\mathbf{i}}{2}x^t A^t \theta^t \mathcal{A} \theta A x+\frac{\mathbf{i}}{2}y^t \mathcal{A} y+\mathbf{i} x^t A^t \theta \mathcal{A} y \right) \Biggr).
\end{align*}
Therefore, indeed, the equality
\begin{equation*}
\mathrm{exp}\left( 2\pi \mathbf{i} \bm{e}_i^t A^t y \right)\star \varphi_{\theta, \mathcal{A}}(x, y)=\varphi_{\theta, \mathcal{A}}(x+\bm{e}_i, y)\star \mathrm{exp}\left( 2 \pi \mathbf{i} \bm{e}_i^t A^t y \right)
\end{equation*}
holds.

Next, we consider the equality (\ref{main_proposition_2}). Note that 
\begin{equation*}
\mathcal{A}^{\theta}=\mathcal{A}\left( I_n+\frac{1}{2\pi}\theta \mathcal{A} \right)^{-1}=\mathcal{A}\left( I_n-\frac{1}{2\pi}\theta \mathcal{A} \right)=\mathcal{A}
\end{equation*}
since $\mathcal{A}\theta \mathcal{A}=O$, so we have
\begin{equation*}
j_{\theta,(A;\mathcal{A})}(T\bm{e}_i, x, y)=\mathrm{exp}\left( \frac{\mathbf{i}}{2}\bm{e}_i^t \mathcal{A} \bm{e}_i-\mathbf{i}\bm{e}_i^t \mathcal{A} \theta A x+\mathbf{i}\bm{e}_i^t \mathcal{A} y \right).
\end{equation*}
Hence, the equality (\ref{main_proposition_2}) holds if and only if the equality 
\begin{equation*}
\varphi_{\theta, \mathcal{A}}(x, y+\bm{e}_i)=\mathrm{exp}\left( \frac{\mathbf{i}}{2}\bm{e}_i^t \mathcal{A} \bm{e}_i-\mathbf{i}\bm{e}_i^t \mathcal{A} \theta A x+\mathbf{i}\bm{e}_i^t \mathcal{A} y \right)\star \varphi_{\theta, \mathcal{A}}(x, y)
\end{equation*}
holds. By direct calculations, although we see
\begin{align}
&\mathrm{exp}\left( \frac{\mathbf{i}}{2}\bm{e}_i^t \mathcal{A} \bm{e}_i-\mathbf{i}\bm{e}_i^t \mathcal{A} \theta A x+\mathbf{i}\bm{e}_i^t \mathcal{A} y \right)\star \varphi_{\theta, \mathcal{A}}(x, y) \notag \\
&=\mathrm{exp}\Biggl( \left( \frac{\mathbf{i}}{2}\bm{e}_i^t \mathcal{A} \bm{e}_i-\mathbf{i}\bm{e}_i^t \mathcal{A} \theta A x+\mathbf{i}\bm{e}_i^t \mathcal{A} y \right) \notag \\
&\hspace{3.8mm} +\left( \frac{\mathbf{i}}{2}x^t A^t \theta^t \mathcal{A} \theta A x+\frac{\mathbf{i}}{2}\left( y-\frac{1}{4\pi}\theta \mathcal{A} \bm{e}_i \right)^t \mathcal{A} \left( y-\frac{1}{4\pi}\theta \mathcal{A} \bm{e}_i \right)+\mathbf{i} x^t A^t \theta \mathcal{A} \left(y-\frac{1}{4\pi}\theta \mathcal{A} \bm{e}_i \right) \right) \Biggr), \label{main_proposition_2_1}
\end{align}
in fact, the local expression (\ref{main_proposition_2_1}) turns out to be
\begin{align*}
&\mathrm{exp}\left( \left( \frac{\mathbf{i}}{2}\bm{e}_i^t \mathcal{A} \bm{e}_i-\mathbf{i}\bm{e}_i^t \mathcal{A} \theta A x+\mathbf{i}\bm{e}_i^t \mathcal{A} y \right)+\left( \frac{\mathbf{i}}{2}x^t A^t \theta^t \mathcal{A} \theta A x+\frac{\mathbf{i}}{2}y^t \mathcal{A} y+\mathbf{i} x^t A^t \theta \mathcal{A} y \right) \right) \\
&=\varphi_{\theta, \mathcal{A}}(x, y+\bm{e}_i)
\end{align*}
since we assume that $\mathcal{A}\theta \mathcal{A}=O$.

Finally, we verify that $\varphi_{\theta, \mathcal{A}}$ satisfies the relation (\ref{main_proposition_3}). By the definition of $\varphi_{\theta, \mathcal{A}}$, we have 
\begin{equation*}
\bar{\partial}\varphi_{\theta, \mathcal{A}}=\mathbf{i} \Bigl( x^t A^t \theta^t \mathcal{A} \theta A-y^t \mathcal{A} \theta A \Bigr) \varphi_{\theta, \mathcal{A}} dx^{(0,1)}+\mathbf{i} \Bigl( y^t \mathcal{A}+x^t A^t \theta \mathcal{A} \Bigr) (-T^{-1}) \varphi_{\theta, \mathcal{A}} dx^{(0,1)}.
\end{equation*}
On the other hand, by the assumption $\mathcal{A}\theta \mathcal{A}=O$, $\omega_{\theta,(A,p,q;\mathcal{A})}^{(0,1)} \overset{\star}{\wedge} \varphi_{\theta, \mathcal{A}}-\varphi_{\theta, \mathcal{A}} \overset{\star}{\wedge} \omega_{\theta,(A,p,q)}^{(0,1)}$ turns out to be
\begin{align*}
&\omega_{\theta,(A,p,q;\mathcal{A})}^{(0,1)} \overset{\star}{\wedge} \varphi_{\theta, \mathcal{A}}-\varphi_{\theta, \mathcal{A}} \overset{\star}{\wedge} \omega_{\theta, (A,p,q)}^{(0,1)} \\
&=\mathbf{i} \left( -x^t A^t \theta^t \mathcal{A} \theta A+y^t \mathcal{A} \theta A+\frac{1}{4\pi} \Bigl( y^t \mathcal{A}+x^t A^t \theta^t \mathcal{A} \Bigr) \theta \mathcal{A} \theta A \right) \varphi_{\theta, \mathcal{A}} dx^{(0,1)} \\
&\hspace{3.8mm} +\mathbf{i} \left( -x^t A^t \theta \mathcal{A}-y^t \mathcal{A}-\frac{1}{4\pi} \Bigl( y^t \mathcal{A}+x^t A^t \theta^t \mathcal{A} \Bigr) \theta \mathcal{A} \right) (-T^{-1}) \varphi_{\theta, \mathcal{A}} dx^{(0,1)} \\
&=\mathbf{i} \Bigl( -x^t A^t \theta^t \mathcal{A} \theta A+y^t \mathcal{A} \theta A \Bigr) \varphi_{(\mathcal{A}, \theta)} dx^{(0,1)}+\mathbf{i} \Bigl( -y^t \mathcal{A}-x^t A^t \theta \mathcal{A} \Bigr) (-T^{-1}) \varphi_{\theta, \mathcal{A}} dx^{(0,1)}.
\end{align*}
Thus, we see
\begin{equation*}
\bar{\partial}\varphi_{\theta, \mathcal{A}}+\omega_{\theta,(A,p,q;\mathcal{A})}^{(0,1)} \overset{\star}{\wedge} \varphi_{\theta, \mathcal{A}}-\varphi_{\theta, \mathcal{A}} \overset{\star}{\wedge} \omega_{\theta,(A,p,q)}^{(0,1)}=0,
\end{equation*}
namely, $\varphi_{\theta, \mathcal{A}}$ satisfies the relation (\ref{main_proposition_3}).
\end{proof}
Clearly, the isomorphism $\varphi_{\theta, \mathcal{A}}$ in Proposition \ref{main_proposition} is a generalization of the isomorphism (\ref{trivialization}) to the noncommutative case, i.e., we can obtain the isomorphism (\ref{trivialization}) by considering the case $\theta=O$ for $\varphi_{\theta, \mathcal{A}}$ in Proposition \ref{main_proposition}: 
\begin{equation*}
\varphi_{\theta=O, \mathcal{A}}=\varphi_{\mathcal{A}}.
\end{equation*}

\subsection{The moduli space of $E_{\theta,(A,p,q;\mathcal{A})}^{\nabla}$}
The purpose of this subsection is to specify the moduli space of $E_{\theta,(A,p,q;\mathcal{A})}^{\nabla}$ for fixed matrices $A\in M(n;\mathbb{Z})$, $\mathcal{A}\in \mathrm{Sym}(n;\mathbb{R})$. More precisely, in the following Theorem \ref{main_theorem_1}, for fixed matrices $A\in M(n;\mathbb{Z})$, $\mathcal{A}\in \mathrm{Sym}(n;\mathbb{R})$, we determine the condition with respect to $p$, $q$, $p'$, $q'\in \mathbb{R}^n$ such that $E_{\theta,(A,p,q;\mathcal{A})}^{\nabla}\cong E_{\theta,(A,p',q';\mathcal{A})}^{\nabla}$. In particular, this result is also used in section 4 in order to construct the object which is mirror dual to each $E_{\theta,(A,p,q;\mathcal{A})}^{\nabla}$.
\begin{theo} \label{main_theorem_1}
We take matrices $A\in M(n;\mathbb{Z})$, $\mathcal{A}\in \mathrm{Sym}(n;\mathbb{R})$ arbitrary, and fix them. Then two objects $E_{\theta,(A,p,q;\mathcal{A})}^{\nabla}$, $E_{\theta,(A,p',q';\mathcal{A})}^{\nabla}$ are isomorphic to each other,
\begin{equation*}
E_{\theta,(A,p,q;\mathcal{A})}^{\nabla}\cong E_{\theta,(A,p',q';\mathcal{A})}^{\nabla},
\end{equation*} 
if and only if there exist $k$, $l\in \mathbb{Z}^n$ such that
\begin{equation*}
\left( \begin{array}{cc} p' \\ q' \end{array} \right)=\left( \begin{array}{cc} p \\ q \end{array} \right)+\left( \begin{array}{cc} \left( I_n-\left( \frac{1}{2\pi}\theta \mathcal{A} \right)^2 \right)^t & O \\ -A^t \theta \left( I_n-\left( \frac{1}{2\pi}\theta \mathcal{A} \right)^2 \right)^t & I_n \end{array} \right) \left( \begin{array}{cc} k \\ l \end{array} \right). 
\end{equation*}
\end{theo}
\begin{proof}
Let us compute a morphism $\phi_{\theta} : E_{\theta,(A,p,q;\mathcal{A})}^{\nabla} \to E_{\theta,(A,p',q';\mathcal{A})}^{\nabla}$ explicitly, namely, it is enough to compute a locally defined smooth function $\phi_{\theta}$ satisfying the conditions
\begin{align}
&\phi_{\theta}(x+\bm{e}_i, y)=\mathrm{exp}\left( 2\pi \mathbf{i} \bm{e}_i^t A^t y \right)\star \phi_{\theta}(x, y)\star \mathrm{exp}\left( -2\pi \mathbf{i} \bm{e}_i^t A^t y \right), \label{main_theorem_1_1} \\
&\phi_{\theta}(x, y+\bm{e}_i)=\mathrm{exp}\left( \mathbf{i} \bm{e}_i^t \mathcal{A}^{\theta} y \right)\star \phi_{\theta}(x, y)\star \mathrm{exp}\left( -\mathbf{i} \bm{e}_i^t \mathcal{A}^{\theta} y \right), \label{main_theorem_1_2} \\
&\bar{\partial}\phi_{\theta}+\omega_{\theta,(A,p',q';\mathcal{A})}^{(0,1)} \overset{\star}{\wedge} \phi_{\theta}-\phi_{\theta} \overset{\star}{\wedge} \omega_{\theta,(A,p,q;\mathcal{A})}^{(0,1)}=0, \label{main_theorem_1_3}   
\end{align}
where $i=1, \ldots, n$. Let us take an arbitrary $i=1, \ldots, n$, and fix it.

First, we consider the condition (\ref{main_theorem_1_2}). By the formula (\ref{bopp}), we see that $\phi_{\theta}$ can be expressed locally as follows (the Fourier expansion): 
\begin{equation}
\phi_{\theta}(x, y)=\sum_{k\in \mathbb{Z}^n} \phi_{\theta,k}(x) \ \mathrm{exp}\left( 2\pi \mathbf{i} k^t \left(I_n-\frac{1}{2\pi}\theta \mathcal{A} \right) y \right), \label{fourier_theta}
\end{equation} 
where $\phi_{\theta,k}$ is a locally defined smooth function of $x$. Then, in fact, we have
\begin{align*}
&\mathrm{exp}\left( \mathbf{i} \bm{e}_i^t \mathcal{A}^{\theta} y \right)\star \phi_{\theta}(x, y) \\
&=\sum_{k\in \mathbb{Z}^n} \phi_{\theta,k}(x) \ \mathrm{exp}\left( \mathbf{i}\bm{e}_i^t \mathcal{A}^{\theta} y+2\pi \mathbf{i}k^t \left( I_n-\frac{1}{2\pi}\theta \mathcal{A} \right) \left( y-\frac{1}{4\pi}\theta (\mathcal{A}^{\theta})^t \bm{e}_i \right) \right) \\
&=\sum_{k\in \mathbb{Z}^n} \phi_{\theta,k}(x) \ \mathrm{exp}\left( -\frac{\mathbf{i}}{2}k^t \theta \mathcal{A} \bm{e}_i+\mathbf{i}\left( \bm{e}_i^t \mathcal{A}^{\theta} +2\pi k^t \left( I_n-\frac{1}{2\pi}\theta \mathcal{A} \right) \right)y \right),
\end{align*}
so 
\begin{align*}
&\mathrm{exp}\left( \mathbf{i} \bm{e}_i^t \mathcal{A}^{\theta} y \right)\star \phi_{\theta}(x, y)\star \mathrm{exp}\left( -\mathbf{i} \bm{e}_i^t \mathcal{A}^{\theta} y \right) \\
&=\sum_{k\in \mathbb{Z}^n} \phi_{\theta,k}(x) \ \mathrm{exp}\left( -\frac{\mathbf{i}}{2}k^t \theta \mathcal{A} \bm{e}_i+\mathbf{i}\left( \bm{e}_i^t \mathcal{A}^{\theta} +2\pi k^t \left( I_n-\frac{1}{2\pi}\theta \mathcal{A} \right) \right) \left( y-\frac{1}{4\pi}\theta (\mathcal{A}^{\theta})^t \bm{e}_i \right)-\mathbf{i} \bm{e}_i^t \mathcal{A}^{\theta} y \right) \\
&=\sum_{k\in \mathbb{Z}^n} \phi_{\theta,k}(x) \ \mathrm{exp}\left( -\mathbf{i}k^t \theta \mathcal{A} \bm{e}_i+2\pi \mathbf{i} k^t \left(I_n-\frac{1}{2\pi}\theta \mathcal{A} \right) y \right) \\
&=\phi_{\theta}(x, y+\bm{e}_i).
\end{align*}
Therefore, indeed, the local expression (\ref{fourier_theta}) satisfies the condition (\ref{main_theorem_1_2}).

Next, we consider the condition (\ref{main_theorem_1_3}). We put
\begin{equation*}
\alpha^{\theta}:=\mathcal{A}(-T^{-1}-\theta A).
\end{equation*}
By the definitions of $\omega_{\theta,(A,p,q;\mathcal{A})}^{(0,1)}$ and $\omega_{\theta,(A,p',q';\mathcal{A})}^{(0,1)}$, now, it is enough to consider the following relation:
\begin{align}
\bar{\partial}\phi_{\theta}&+\left( -\mathbf{i}y^t \alpha^{\theta}-2\pi \mathbf{i} \left( p'^t (-T^{-1})-q'^t \right) \right) dx^{(0,1)}\overset{\star}{\wedge} \phi_{\theta} \notag \\
&-\phi_{\theta} \overset{\star}{\wedge} \left( -\mathbf{i}y^t \alpha^{\theta}-2\pi \mathbf{i} \left( p^t (-T^{-1})-q^t \right) \right) dx^{(0,1)}=0. \label{diff_eq_phi_theta}
\end{align}
Furthermore, by substituting the local expression (\ref{fourier_theta}) to the relation (\ref{diff_eq_phi_theta}), we obtain the following system of differential equations, where $i=1, \ldots, n$:
\begin{align*}
\frac{\partial \phi_{\theta,k}}{\partial x_i}(x)+\Biggl( &2\pi \mathbf{i}k^t \left( I_n-\frac{1}{2\pi}\theta \mathcal{A} \right) \left( \left( I_n+\frac{1}{2\pi}\theta \mathcal{A} \right) (-T^{-1})-\frac{1}{2\pi}\theta \mathcal{A} \theta A \right) \\
&-2\pi \mathbf{i}\left( \left( p'^t-p^t \right) (-T^{-1})-q'^t+q^t \right) \Biggr)\bm{e}_i \ \phi_{\theta,k}(x)=0.   
\end{align*} 
Hence, we obtain the solution
\begin{align*}
\phi_{\theta,k}(x)=C_k \ \mathrm{exp}\Biggl( &2\pi \mathbf{i} \left( k^t \left( I_n-\frac{1}{2\pi}\theta \mathcal{A} \right) \left( I_n+\frac{1}{2\pi}\theta \mathcal{A} \right)-p'^t+p^t \right)T^{-1}x \\
&-2\pi \mathbf{i} \left( -\frac{1}{2\pi}k^t \left( I_n-\frac{1}{2\pi}\theta \mathcal{A} \right)\theta \mathcal{A} \theta A+q'^t-q^t \right)x \Biggr), 
\end{align*}
where $C_k\in \mathbb{C}$ is an arbitrary constant. For simplicity, we set
\begin{align*}
&\bm{a}_{(p,p')}:=\left( I_n+\frac{1}{2\pi}\theta \mathcal{A} \right)^t \left( I_n-\frac{1}{2\pi}\theta \mathcal{A} \right)^t k-p'+p, \\
&\bm{b}_{(q,q')}:=-\frac{1}{2\pi}A^t \theta^t \mathcal{A} \theta^t \left( I_n-\frac{1}{2\pi}\theta \mathcal{A} \right)^t k+q'-q,
\end{align*}
namely,
\begin{equation*}
\phi_{\theta,k}(x)=C_k \ \mathrm{exp}\left( 2\pi \mathbf{i} \bm{a}_{(p,p')}^t T^{-1}x-2\pi \mathbf{i} \bm{b}_{(q,q')}^t x \right).
\end{equation*}

Finally, we consider the condition (\ref{main_theorem_1_1}). Simlarly as in the case of the condition (\ref{main_theorem_1_2}), we can compute as follows:
\begin{align*}
&\mathrm{exp}\left( 2\pi \mathbf{i} \bm{e}_i^t A^t y \right)\star \phi_{\theta}(x, y)\star \mathrm{exp}\left( -2\pi \mathbf{i} \bm{e}_i^t A^t y \right) \\
&=\sum_{k\in \mathbb{Z}^n} \phi_{\theta,k}(x) \ \mathrm{exp}\left( -2\pi \mathbf{i} k^t \left( I_n-\frac{1}{2\pi}\theta \mathcal{A} \right) \theta A \bm{e}_i+2\pi \mathbf{i} k^t \left(I_n-\frac{1}{2\pi}\theta \mathcal{A} \right) y \right).
\end{align*}
On the other hand, it is clear that 
\begin{equation*}
\phi_{\theta}(x+\bm{e}_i, y)=\sum_{k\in \mathbb{Z}^n} \phi_{\theta,k}(x+\bm{e}_i) \ \mathrm{exp}\left( 2\pi \mathbf{i} k^t \left(I_n-\frac{1}{2\pi}\theta \mathcal{A} \right) y \right).
\end{equation*}
Thus, if there exist $k$, $l\in \mathbb{Z}^n$ such that 
\begin{equation*}
2\pi \mathbf{i} \bm{a}_{(p,p')}^t T^{-1}-2\pi \mathbf{i} \bm{b}_{(q,q')}^t =-2\pi \mathbf{i} k^t \left( I_n-\frac{1}{2\pi}\theta \mathcal{A} \right) \theta A-2\pi \mathbf{i}l^t, 
\end{equation*}
i.e., if there exist $k$, $l\in \mathbb{Z}^n$ such that
\begin{equation*}
\left( \begin{array}{cc} p' \\ q' \end{array} \right)=\left( \begin{array}{cc} p \\ q \end{array} \right)+\left( \begin{array}{cc} \left( I_n-\left( \frac{1}{2\pi}\theta \mathcal{A} \right)^2 \right)^t & O \\ -A^t \theta \left( I_n-\left( \frac{1}{2\pi}\theta \mathcal{A} \right)^2 \right)^t & I_n \end{array} \right) \left( \begin{array}{cc} k \\ l \end{array} \right),
\end{equation*}
then we should put $C_{k'}=0$ for any $k'\not=k$ ($k'\in \mathbb{Z}^n$):
\begin{equation*}
\phi_{\theta}(x, y)=C \ \mathrm{exp}\left( -2\pi \mathbf{i} \left( k^t \left( I_n-\frac{1}{2\pi}\theta \mathcal{A} \right)\theta A+l^t \right)x+2\pi \mathbf{i}k^t \left( I_n-\frac{1}{2\pi}\theta \mathcal{A} \right)y \right), 
\end{equation*}
where $C:=C_k\in \mathbb{C}$. If not, we must set $C_k=0$ for any $k\in \mathbb{Z}^n$ in order to construct $\phi_{\theta}$ satisfying the conditions (\ref{main_theorem_1_1}), (\ref{main_theorem_1_2}), (\ref{main_theorem_1_3}), namely, $\phi_{\theta}$ which satisfies the conditions (\ref{main_theorem_1_1}), (\ref{main_theorem_1_2}), (\ref{main_theorem_1_3}) is $\phi_{\theta}=0$ only. This completes the proof. 

In particular, the statement of this theorem implies that the moduli space of $E_{\theta,(A,p,q;\mathcal{A})}^{\nabla}$ for fixed matrices $A\in M(n;\mathbb{Z})$, $\mathcal{A}\in \mathrm{Sym}(n;\mathbb{R})$ is homeomorphic to the $2n$-dimensional real torus
\begin{equation*}
\Biggl. \mathbb{R}^{2n} \Biggr/ \left( \begin{array}{cc} \left( I_n-\left( \frac{1}{2\pi}\theta \mathcal{A} \right)^2 \right)^t & O \\ -A^t \theta \left( I_n-\left( \frac{1}{2\pi}\theta \mathcal{A} \right)^2 \right)^t & I_n \end{array} \right) \mathbb{Z}^{2n}. 
\end{equation*}
\end{proof}

We denote the set of isomorphism classes (the moduli space) of noncommutative objects $E_{\theta,(A,p,q;\mathcal{A})}^{\nabla}$ for fixed matrices $A\in M(n;\mathbb{Z})$, $\mathcal{A}\in \mathrm{Sym}(n;\mathbb{R})$ by
\begin{equation*}
\mathcal{M}_{\theta}(A;\mathcal{A}).
\end{equation*} 

We comment on several results related to Theorem \ref{main_theorem_1}. By the theorem of Appell-Humbert, we can define the holomorphic line bundle on $X^n=\mathbb{C}^n/(\mathbb{Z}^n\oplus T\mathbb{Z}^n)$ for a given pair consisting of a Hermitian metric $H$ on $\mathbb{C}^n$ satisfying an appropriate condition and a semi-representation $U : \mathbb{Z}^n\oplus T\mathbb{Z}^n \to U(1)$. Also, the corresponding factor of automorphy is described explicitly by using those $H$ and $U$. Moreover, the theorem of Appell-Humbert can be generalized to the higher rank case when we focus on projectively flat (semi-homogeneous) bundles on $X^n$ \cite{dg-vect}. In our prior results \cite{bijection}, by using the classification result of simple projectively flat bundles on complex tori \cite[Theorem 6.1]{proj-flat}, \cite[Proposition 6.17 (1)]{semi-hom}, semi-representations (that cannot be determined by the first Chern character) of (a certain) simple projectively flat bundles on $X^n$ are investigated, including the higher rank case\footnote{Precisely speaking, we need to rescale the lattice $\mathbb{Z}^n \oplus T\mathbb{Z}^n$ in $\mathbb{C}^n$ to the lattice $2\pi (\mathbb{Z}^n \oplus T\mathbb{Z}^n)$ in $\mathbb{C}^n$ in order to make our setting correspond to the arguments in \cite{bijection}.}. This result is also used in order to determine the data in $Fuk(\check{X}^n)$ corresponding to semi-representations \cite{bijection}. The statement of Theorem \ref{main_theorem_1} in the case $\theta=O$ can be regarded as a special case of \cite[Theorem 3.4]{bijection}. On the other hand, when we consider the case of elliptic curves, i.e., the case $n=1$, the terms in the Poisson bivector (\ref{P_bivector}) corresponding $\theta_1$, $\theta_3\in \mathrm{Alt}(n;\mathbb{R})$ vanish, so it becomes
\begin{equation*}
\sum_{i, j=1}^n (\theta_2)_{ij} \frac{\partial}{\partial x_i}\wedge \frac{\partial}{\partial y_j}.
\end{equation*}
The noncommutative deformation of $X^1$ treated in \cite{p-s-nc} is essentially associated to the nonformal deformation quantization by this Poisson bivector, and \cite[Proposition 2.1 (b)]{p-s-nc} is an analogue of Theorem \ref{main_theorem_1} in this case.

\section{A deformation of objects of the Fukaya category over $\check{X}^n$}
In this section, we construct the deformation of objects $\mathcal{L}_{(A,p,q)}^{\nabla}$ of $Fuk_{\rm sub}(\check{X}^n)$ which are mirror dual to noncommutative objects $E_{\theta,(A,p,q;\mathcal{A})}^{\nabla}$, and study some properties of them. In particular, the analogue of Theorem \ref{main_theorem_1} is given in Theorem \ref{main_theorem_2}, and a generalization of the SYZ transform 
\begin{equation*}
\mathcal{L}_{(A,p,q)}^{\nabla} \ \longmapsto \ E_{(A,p,q)}^{\nabla}
\end{equation*}
to the setting which includes the noncommutative parameter $\theta$ is given in Theorem \ref{main_theorem_3}.

\subsection{A deformation of $\mathcal{L}_{(A,p,q)}^{\nabla}$}
The purpose of this subsection is to construct the deformation of objects $\mathcal{L}_{(A,p,q)}^{\nabla}$ of $Fuk_{\rm sub}(\check{X}^n)$ associated to the deformation from $\check{X}^n$ to the mirror partner of $X_{\theta}^n$. This construction is based on an idea which is proposed in our previous paper \cite{b-field}.

In order to achieve this purpose, we first need to consider the mirror partner of $X_{\theta}^n$ according to the mirror transform (\ref{mirror}). In the complex geometry side, the Poisson bivector $\Pi_{\theta}$ (see the formula (\ref{theta})) causes the $\beta$-field transform 
\begin{align}
&\mathcal{I}_T(\Pi_{\theta}) \left( \frac{\partial}{\partial x}^t, \frac{\partial}{\partial y}^t, dx^t, dy^t \right) \notag \\
&=\left( \frac{\partial}{\partial x}^t, \frac{\partial}{\partial y}^t, dx^t, dy^t \right) \left( \begin{array}{cccc} I_n & O & O & O \\ O & I_n & O & -\theta \\ O & O & I_n & O \\ O & O & O & I_n \end{array} \right) \mathcal{I}_T \left( \begin{array}{cccc} I_n & O & O & O \\ O & I_n & O & \theta \\ O & O & I_n & O \\ O & O & O & I_n \end{array} \right) \label{g_c_c_theta}
\end{align}
of $\mathcal{I}_T$ over $X^n$ (see the formula (\ref{g_c_c}) for the details of the definition of $\mathcal{I}_T$) from the viewpoint of generalized complex geometry. Hence, we can calculate
\begin{align}
&\check{\mathcal{I}}_T(\Pi_{\theta}) \left( \frac{\partial}{\partial \check{x}}^t, \frac{\partial}{\partial \check{y}}^t, d\check{x}^t, d\check{y}^t \right) \notag \\
&=\left( \frac{\partial}{\partial \check{x}}^t, \frac{\partial}{\partial \check{y}}^t, d\check{x}^t, d\check{y}^t \right) \left( \begin{array}{cccc} I_n & O & O & O \\ O & I_n & O & O \\ O & O & I_n & O \\ O & -\theta & O & I_n \end{array} \right) \check{\mathcal{I}}_T \left( \begin{array}{cccc} I_n & O & O & O \\ O & I_n & O & O \\ O & O & I_n & O \\ O & \theta & O & I_n \end{array} \right) \label{g_c_s_theta}
\end{align}
via the mirror transform (\ref{mirror}) (see the formula (\ref{g_c_s}) for the details of the definition of $\check{\mathcal{I}}_T$). This fact implies that the mirror partner of $X_{\theta}^n$ is given by the $2n$-dimensional real torus $\mathbb{R}^{2n}/\mathbb{Z}^{2n}$ with the complexified symplectic form $\tilde{\omega}_{\theta}^{\vee}$ which is obtained by twisting $\tilde{\omega}^{\vee}$ with a B-field depending on $\theta$: 
\begin{equation*}
\tilde{\omega}_{\theta}^{\vee}:=\tilde{\omega}^{\vee}+\frac{1}{2}d\check{y}^t \theta d\check{y}=\mathbf{i}\omega^{\vee}+\left(B^{\vee}+\frac{1}{2}d\check{y}^t \theta d\check{y} \right).
\end{equation*}
In particular, note that the symplectic form $\omega^{\vee}$ is preserved. For later convenience, let us put
\begin{equation*}
B_{\theta}^{\vee}:=\frac{1}{2}d\check{y}^t \theta d\check{y},
\end{equation*}
namely,
\begin{equation*}
\tilde{\omega}_{\theta}^{\vee}=\mathbf{i}\omega^{\vee}+\bigl( B^{\vee}+B_{\theta}^{\vee} \bigr).
\end{equation*}
Hereafter, we denote this complexified symplectic torus (the mirror partner of $X_{\theta}^n$ which is obtained by using the mirror transform (\ref{mirror})) by
\begin{equation*}
\check{X}_{\theta}^n:=\left( \mathbb{R}^{2n}/\mathbb{Z}^{2n}, \ \tilde{\omega}_{\theta}^{\vee}=d\check{x}^t (-(T^{-1})^t) d\check{y}+\frac{1}{2}d\check{y}^t \theta d\check{y} \right).
\end{equation*}
Of course, we can regard $\{ \check{O}_m^l \}_{(l;m)\in I}$ as an open covering of $\check{X}_{\theta}^n$.

We construct the deformation of objects $\mathcal{L}_{(A,p,q)}^{\nabla}$ of $Fuk_{\rm sub}(\check{X}^n)$ associated to the deformation from $\check{X}^n$ to $\check{X}_{\theta}^n$. They are also mirrors of noncommutative objects $E_{\theta,(A,p,q;\mathcal{A})}^{\nabla}$. 

Before starting main discussions, we need to give a remark. As recalled at the beginning of subsection 2.3, we usually take a smooth complex line bundle $\mathcal{L}\to L$ with a unitary connection $\nabla_{\mathcal{L}}$ which satisfies the condition (\ref{curv_b_field}):
\begin{equation*}
\Omega_{\nabla_{\mathcal{L}}}=\Bigl. 2\pi \mathbf{i}B \Bigr|_L
\end{equation*}
as an object of the Fukaya category over a given complexified symplectic manifold $M_{\omega, B}$, where $L$ is a Lagrangian submanifold in $M_{\omega, B}$, and $\Omega_{\nabla_{\mathcal{L}}}$ denotes the curvature form of $\nabla_{\mathcal{L}}$ (see also Remark \ref{fuk_ob_tori}). Here, note that there exists such a unitary connection $\nabla_{\mathcal{L}}$ if and only if 
\begin{equation}
[B]\in H^2(L,\mathbb{Z}). \label{integrity}
\end{equation}
On the other hand, under the deformation from $\check{X}^n$ to $\check{X}_{\theta}^n$, although the symplectic form $\omega^{\vee}$ is preserved, the B-field $B^{\vee}$ is twisted by $B_{\theta}^{\vee}$. In particular, we see
\begin{equation*}
\Bigl. 2\pi \mathbf{i} \bigl( B^{\vee}+B_{\theta}^{\vee} \bigr) \Bigr|_{L_{(A,p)}}=\pi \mathbf{i} d\check{x}^t A^t \theta A d\check{x}
\end{equation*}
by the assumption $AT\in \mathrm{Sym}(n;\mathbb{C})$. This implies that it is natural to consider $L_{(A,p)}$ is preserved and $\mathcal{O}_{(A,p,q)}^{\nabla}$ is deformed to something. However, in general, for a given $\theta$, there exists a case such that a given Lagrangian submanifold $L_{(A,p)}$ in $\check{X}_{\theta}^n$ does not satisfy the integrity condition (\ref{integrity}):
\begin{equation*}
\bigl[ B^{\vee}+B_{\theta}^{\vee} \bigr]\not\in H^2(L_{(A,p)},\mathbb{Z})
\end{equation*}  
even though we can consider the deformation $E_{\theta,(A,p,q;\mathcal{A})}^{\nabla}$ of each $E_{(A,p,q)}^{\nabla}$ depending on $\theta$ (consider the case that $\theta \in M(n;\mathbb{R})\backslash M(n;\mathbb{Q})$ and $A=I_n\in M(n;\mathbb{Z})$ for instance). In order to overcome this problem, in our previous paper \cite{b-field}, we proposed to employ a ``twisted'' line bundle on each Lagrangian submanifold $L$ associated to the flat gerbe whose 1-connection is $2\pi \mathbf{i}B|_L$, instead of ``usual'' line bundles. Concerning these facts, in this paper, we construct the deformation of objects $\mathcal{L}_{(A,p,q)}^{\nabla}$ based on this idea\footnote{Strictly speaking, we do not discuss the compatibility of our proposal and some structures required to define (an analogue of) the Fukaya category over $\check{X}_{\theta}^n$. However, maybe, our proposal is actually compatible with such requirements. See also subsection 5.4 in \cite{b-field}.}. Also, perhapes it is no coincidence that we can apply the idea proposed in \cite{b-field} to our setting. We explain the relation between this paper and \cite{b-field} in section 5. 

Hereafter, for simplicity, we denote the symplectic torus $(\mathbb{R}^{2n}/\mathbb{Z}^{2n}, \omega^{\vee})$ by $\check{X}_{\omega^{\vee}}^n$. By using the B-field $B^{\vee}+B_{\theta}^{\vee}$, we define a flat gerbe on the Lagrangian submanifold $L_{(A,p)}$ as follows. In fact, the effect of $B^{\vee}$ vanishes on $L_{(A,p)}$, i.e.,
\begin{equation*}
\Bigl. B^{\vee} \Bigr|_{L_{(A,p)}}=0
\end{equation*}  
since we assume that $AT\in \mathrm{Sym}(n;\mathbb{C})$, so essentially, it is enough to consider the effect of $B_{\theta}^{\vee}$ only on $L_{(A,p)}$. Moreover, note that the flat gerbe given in this subsection can also be regarded as the restriction of the flat gerbe defined globally on $\check{X}_{\omega^{\vee}}^n$ whose 1-connection is $2\pi \mathbf{i}(B^{\vee}+B_{\theta}^{\vee})$ to $L_{(A,p)}$. Let us consider an open covering of $L_{(A,p)}$. By taking a suitable subset $I'$ of $I$, we can construct an open covering $\{ \check{O}_m^l(A,p) \}_{(l;m)\in I'}$ of $L_{(A,p)}$:
\begin{equation*}
\check{O}_m^l(A,p):=\check{O}_m^l \cap L_{(A,p)}, 
\end{equation*}
where $(l;m)\in I'$. We set
\begin{equation*}
B_{\theta,A}^{\vee}:=\Bigl. B_{\theta}^{\vee} \Bigr|_{L_{(A,p)}},
\end{equation*}
i.e.,
\begin{equation*}
B_{\theta,A}^{\vee}=\frac{1}{2}d\check{x}^t A^t \theta A d\check{x}.
\end{equation*}
For $B_{\theta,A}^{\vee}$, let us take a 1-form $\beta_{\theta,A}^{\vee}$ which is expressed locally as
\begin{equation*}
\beta_{\theta,A}^{\vee}=\frac{1}{2}\check{x}^t A^t \theta A d\check{x},
\end{equation*}
and it is clear that $B_{\theta,A}^{\vee}=d\beta_{\theta,A}^{\vee}$. We sometimes use the notation $\beta_{\theta,A}^{\vee}(\check{x})$ instead of $\beta_{\theta,A}^{\vee}$ in order to specify the coordinate system $\check{x}$. For each $i=1, \ldots, n$, let us locally define $\omega_{\theta,A}^{\vee}(\bm{e}_i)$ by
\begin{equation*}
\omega_{\theta,A}^{\vee}(\bm{e}_i)=\beta_{\theta,A}^{\vee}(\check{x}+\bm{e}_i)-\beta_{\theta,A}^{\vee}(\check{x})=\frac{1}{2}\bm{e}_i^t A^t \theta A d\check{x}.
\end{equation*}
Now, for an arbitrary $i=1, \ldots, n$ and points $([\check{x}], [A\check{x}+p])$, $([\check{x}+\bm{e}_i], [A(\check{x}+\bm{e}_i)+p])\in L_{(A,p)}$, there exist $(l;m)$, $(l';m')\in I'$ such that $([\check{x}], [A\check{x}+p])\in \check{O}_m^l(A,p)$, $([\check{x}+\bm{e}_i], [A(\check{x}+\bm{e}_i)+p])\in \check{O}_{m'}^{l'}(A,p)$, so let us denote the open set $\check{O}_m^l(A,p)$ including $([\check{x}], [A\check{x}+p])$ and the open set $\check{O}_{m'}^{l'}(A,p)$ including $([\check{x}+\bm{e}_i], [A(\check{x}+\bm{e}_i)+p])$ by $\check{O}_{[\check{x}]}$ and $\check{O}_{[\check{x}+\bm{e}_i]}$, respectively. Then $\omega_{\theta,A}^{\vee}(\bm{e}_1), \ldots, \omega_{\theta,A}^{\vee}(\bm{e}_n)$ can be considered as 1-forms which are defined on $\check{O}_{[\check{x}]}\cap \check{O}_{[\check{x}+\bm{e}_1]}, \ldots, \check{O}_{[\check{x}]}\cap \check{O}_{[\check{x}+\bm{e}_n]}$, respectively. For each $i=1, \ldots, n$, let us consider a function
\begin{equation*}
\xi_{\theta,A}^{\vee}(\bm{e}_i,\check{x}):=\mathrm{exp}\Bigl( -\pi \mathbf{i} \bm{e}_i^t A^t \theta A \check{x} \Bigr)
\end{equation*}
defined on $\check{O}_{[\check{x}]}\cap \check{O}_{[\check{x}+\bm{e}_i]}$ which satisfies the differential equation
\begin{equation*}
\left( d+2\pi \mathbf{i}\omega_{\theta,A}^{\vee}(\bm{e}_i) \right) \left( \xi_{\theta,A}^{\vee}(\bm{e}_i,\check{x}) \right)=0.
\end{equation*}  
We further define $\xi_{\theta,A}^{\vee}(\bm{e}_i+\bm{e}_j,\check{x})$ by
\begin{equation*}
\xi_{\theta,A}^{\vee}(\bm{e}_i+\bm{e}_j,\check{x})=\xi_{\theta,A}^{\vee}(\bm{e}_i,\check{x}+\bm{e}_j)\xi_{\theta,A}^{\vee}(\bm{e}_j,\check{x}), 
\end{equation*}
where $i$, $j=1, \ldots, n$. Now, by using them, we consider the following constant functions defined on $\check{O}_{[\check{x}]}\cap \check{O}_{[\check{x}+\bm{e}_i]}\cap \check{O}_{[\check{x}+\bm{e}_i+\bm{e}_j]}$, where $i$, $j=1, \ldots, n$:
\begin{equation*}
\alpha_{\theta,A}^{\vee}(\bm{e}_i,\bm{e}_j):=\xi_{\theta,A}^{\vee}(\bm{e}_i+\bm{e}_j,\check{x})^{-1}\xi_{\theta,A}^{\vee}(\bm{e}_j,\check{x}+\bm{e}_i)\xi_{\theta,A}^{\vee}(\bm{e}_i,\check{x})=\mathrm{exp}\Bigl( 2\pi \mathbf{i}\bm{e}_i^t A^t \theta A \bm{e}_j \Bigr).
\end{equation*}
Moreover, we have the relations 
\begin{equation*}
2\pi \mathbf{i}\omega_{\theta,A}^{\vee}(\bm{e}_i)+2\pi \mathbf{i}\omega_{\theta,A}^{\vee}(\bm{e}_j)-2\pi \mathbf{i}\omega_{\theta,A}^{\vee}(\bm{e}_i+\bm{e}_j)=\alpha_{\theta,A}^{\vee}(\bm{e}_i,\bm{e}_j)d\alpha_{\theta,A}^{\vee}(\bm{e}_i,\bm{e}_j)^{-1}=0,
\end{equation*}
where 
\begin{equation*}
\omega_{\theta,A}^{\vee}(\bm{e}_i+\bm{e}_j):=\omega_{\theta,A}^{\vee}(\bm{e}_i)+\omega_{\theta,A}^{\vee}(\bm{e}_j), 
\end{equation*}
and $i$, $j=1, \ldots, n$. For simplicity, let us put
\begin{equation*}
\alpha_{\theta,A}^{\vee}:=\Bigl\{ \alpha_{\theta,A}^{\vee}(\bm{e}_i,\bm{e}_j) \Bigr\}_{i,j=1,\ldots,n}, \ \ \ \nabla_{\theta,A}^{\vee}:=\Bigl\{ d+2\pi \mathbf{i}\omega_{\theta,A}^{\vee}(\bm{e}_i) \Bigr\}_{i=1,\ldots,n}.
\end{equation*}
Summarizing the above arguments, we see that the data $(\alpha_{\theta,A}^{\vee}, \nabla_{\theta,A}^{\vee}, 2\pi \mathbf{i}B_{\theta,A}^{\vee})$ defines a flat gerbe $\check{\mathcal{G}}_{\theta,A}^{\nabla}$ in the sense of Hitchin-Chatterjee \cite{chat, h} on $L_{(A,p)}$:
\begin{equation*}
\check{\mathcal{G}}_{\theta,A}^{\nabla}:=\Bigl( \alpha_{\theta,A}^{\vee}, \ \nabla_{\theta,A}^{\vee}, \ 2\pi \mathbf{i}B_{\theta,A}^{\vee} \Bigr).
\end{equation*}
Here, the family $\nabla_{\theta,A}^{\vee}$ gives a 0-connection over the gerbe on $L_{(A,p)}$ which is determined by the family $\alpha_{\theta,A}^{\vee}$, and the globally defined 2-form $2\pi \mathbf{i}B_{\theta,A}^{\vee}$ is the 1-connection which is compatible with the 0-connection $\nabla_{\theta,A}^{\vee}$.

\begin{rem} \label{lieblich}
Originally, gerbes are introduced in \textup{\cite{giraud}} as stacks satisfying appropriate conditions, and Hitchin-Chatterjee gives a differential geometric interpretation for gerbes in \textup{\cite{chat, h}}. As mentioned in the above, the definition of the flat gerbe $\check{\mathcal{G}}_{\theta,A}^{\nabla}$ is based on Hitchin-Chatterjee's work \textup{\cite{chat, h}}. Let $X$ be a smooth manifold and $\mathscr{A}$ an abelian sheaf on $X$. We consider an $\mathscr{A}$-gerbe $\mathscr{G}$ \textup{(}as a stack\textup{)} on $X$ and an element $\alpha \in H^2(X,\mathscr{A})$ which corresponds to $\mathscr{G}$. Then, as proved in \textup{\cite[Proposition 2.1.3.3]{lieb}}, it is known that twisted sheaves on $\mathscr{G}$ can be interpreted as $\alpha$-twisted sheaves on $X$ in the sense of C\u{a}ld\u{a}raru \textup{\cite{cal}}. 
\end{rem}

We consider the deformation of each object $\mathcal{L}_{(A,p,q)}^{\nabla}$ of $Fuk_{\rm sub}(\check{X}^n)$ which is mirror dual to $E_{\theta,(A,p,q;\mathcal{A})}^{\nabla}$. As remarked in the above, it is natural to consider $L_{(A,p)}$ is preserved since the symplectic form $\omega^{\vee}$ is preserved under the deformation from $\check{X}^n$ to $\check{X}_{\theta}^n$. Hence, we may consider the deformation of $\mathcal{O}_{(A,p,q)}^{\nabla}$ only, and it is realized as the following $\alpha_{\theta,A}^{\vee}$-twisted smooth complex line bundle on $L_{(A,p)}$ (see also Remark \ref{lieblich}). First, the family of the trivial transition functions is twisted as follows, where $j_{\theta,(A;\mathcal{A})}^{\vee}$ is a map from $\mathbb{Z}^n\times \mathbb{R}^n$ to $\mathbb{C}^{\times}$, and $i=1, \cdots, n$:
\begin{align*}
1 \mapsto j_{\theta,(A;\mathcal{A})}^{\vee}(\bm{e}_i, \check{x})&:=\mathrm{exp}\left( -\frac{\mathbf{i}}{4\pi}\bm{e}_i^t \mathcal{A}^{\theta} \theta \left( \mathcal{A}^{\theta} \right)^t \check{x} \right) \xi_{\theta,A}^{\vee}(\bm{e}_i, \check{x}) \\
&=\mathrm{exp}\left( -\frac{\mathbf{i}}{4\pi}\bm{e}_i^t \mathcal{A}^{\theta} \theta \left( \mathcal{A}^{\theta} \right)^t \check{x}-\pi \mathbf{i}\bm{e}_i^t A^t \theta A \check{x} \right). 
\end{align*}
Furthermore, $j_{\theta,(A;\mathcal{A})}^{\vee}(\bm{e}_i+\bm{e}_j, \check{x})$ is defined by $j_{\theta,(A;\mathcal{A})}^{\vee}(\bm{e}_i, \check{x}+\bm{e}_j) j_{\theta,(A;\mathcal{A})}^{\vee}(\bm{e}_j, \check{x})$:
\begin{equation*}
j_{\theta,(A;\mathcal{A})}^{\vee}(\bm{e}_i+\bm{e}_j, \check{x}):=j_{\theta,(A;\mathcal{A})}^{\vee}(\bm{e}_i, \check{x}+\bm{e}_j) j_{\theta,(A;\mathcal{A})}^{\vee}(\bm{e}_j, \check{x}),
\end{equation*}
where $i$, $j=1, \ldots, n$. By the assumption (\ref{cocycle_theta_mat}), we can easily check that the map $j_{\theta,(A;\mathcal{A})}^{\vee} : \mathbb{Z}^n\times \mathbb{R}^n \to \mathbb{C}^{\times}$ satisfies the relations 
\begin{equation*}
j_{\theta,(A;\mathcal{A})}^{\vee}(\bm{e}_i+\bm{e}_j, \check{x})^{-1} j_{\theta,(A;\mathcal{A})}^{\vee}(\bm{e}_j, \check{x}+\bm{e}_i) j_{\theta,(A;\mathcal{A})}^{\vee}(\bm{e}_i, \check{x})=\alpha_{\theta,A}^{\vee}(\bm{e}_i, \bm{e}_j)\cdot 1,
\end{equation*}
where $i$, $j=1, \ldots, n$. They imply that the map $j_{\theta,(A;\mathcal{A})}^{\vee}$ defines an $\alpha_{\theta,A}^{\vee}$-twisted smooth complex line bundle on $L_{(A,p)}$, so let us denote it by $\mathcal{O}_{\theta,(A,p;\mathcal{A})}\to L_{(A,p)}$ or $\mathcal{O}_{\theta,(A,p;\mathcal{A})}$ for short. On the other hand, as the deformation of the 1-form $\omega_{(A,p,q)}^{\vee}$, we consider the 1-form $\omega_{\theta,(A,p,q;\mathcal{A})}^{\vee}$ which is locally defined by
\begin{equation*}
\omega_{\theta,(A,p,q;\mathcal{A})}^{\vee}=\omega_{(A,p,q)}^{\vee}+\frac{\mathbf{i}}{4\pi} \check{x}^t \mathcal{A}^{\theta} \theta \left( \mathcal{A}^{\theta} \right)^t d\check{x}=2\pi \mathbf{i}q^t d\check{x}+\frac{\mathbf{i}}{4\pi} \check{x}^t \mathcal{A}^{\theta} \theta \left( \mathcal{A}^{\theta} \right)^t d\check{x}.
\end{equation*}
Actually, this 1-form $\omega_{\theta,(A,p,q;\mathcal{A})}^{\vee}$ satisfies the relation
\begin{equation*}
\omega_{\theta,(A,p,q;\mathcal{A})}^{\vee}(\check{x}+\bm{e}_i)=\omega_{\theta,(A,p,q;\mathcal{A})}^{\vee}(\check{x})+j_{\theta,(A;\mathcal{A})}^{\vee}(\bm{e}_i, \check{x})d\hspace{0.3mm} j_{\theta,(A;\mathcal{A})}^{\vee}(\bm{e}_i, \check{x})^{-1}-2\pi \mathbf{i}\omega_{\theta,A}^{\vee}(\bm{e}_i)
\end{equation*}
for each $i=1, \ldots, n$. Also, we use the notation $\omega_{\theta,(A,p,q;\mathcal{A})}^{\vee}(\check{x})$ instead of $\omega_{\theta,(A,p,q;\mathcal{A})}^{\vee}$ in order to specify the coordinate system $\check{x}$. Therefore, 
\begin{equation*}
\nabla_{\theta,(A,p,q;\mathcal{A})}^{\vee}:=d+\omega_{\theta,(A,p,q;\mathcal{A})}^{\vee}
\end{equation*}
gives a connection on $\mathcal{O}_{\theta,(A,p,;\mathcal{A})}$. Here, we set $\mathcal{O}_{\theta,(A,p,q;\mathcal{A})}^{\nabla}:=(\mathcal{O}_{\theta,(A,p;\mathcal{A})}, \nabla_{\theta,(A,p,q;\mathcal{A})}^{\vee})$. Let us denote the pair of the Lagrangian submanifold $L_{(A,p)}$ in $\check{X}_{\theta}^n$ and the $\alpha_{\theta,A}^{\vee}$-twisted smooth complex line bundle with the connection $\mathcal{O}_{\theta,(A,p,q;\mathcal{A})}^{\nabla}$ by $\mathcal{L}_{\theta,(A,p,q;\mathcal{A})}^{\nabla}$: $\mathcal{L}_{\theta,(A,p,q;\mathcal{A})}^{\nabla}:=(L_{(A,p)}, \mathcal{O}_{\theta,(A,p,q;\mathcal{A})}^{\nabla})$. 
\begin{rem} \label{main_proposition_analogue}
When we denote $\mathcal{L}_{\theta,(A,p,q;\mathcal{A})}^{\nabla}=(L_{(A,p)}, \mathcal{O}_{\theta,(A,p,q;\mathcal{A})}^{\nabla})$ in the case $\mathcal{A}=O$ by $\mathcal{L}_{\theta,(A,p,q)}^{\nabla}=(L_{(A,p)}, \mathcal{O}_{\theta,(A,p,q)}^{\nabla})$, it is clear that $\mathcal{O}_{\theta,(A,p,q;\mathcal{A})}^{\nabla}$ is identified with $\mathcal{O}_{\theta,(A,p,q)}^{\nabla}$ under the assumption $\mathcal{A}\theta \mathcal{A}=O$. This fact indicates that $\mathcal{L}_{\theta,(A,p,q;\mathcal{A})}^{\nabla}$ coincides with $\mathcal{L}_{\theta,(A,p,q)}^{\nabla}$ in the case $\mathcal{A}\theta \mathcal{A}=O$. It is also an analogue of Proposition \ref{main_proposition} in the symplectic geometry side. 
\end{rem}
The curvature form $\Omega_{\theta,(A,p,q;\mathcal{A})}^{\vee}$ of the connection $\nabla_{\theta,(A,p,q;\mathcal{A})}^{\vee}$ is expressed locally as
\begin{align*}
\Omega_{\theta,(A,p,q;\mathcal{A})}^{\vee}&=d\omega_{\theta,(A,p,q;\mathcal{A})}^{\vee}+\omega_{\theta,(A,p,q;\mathcal{A})}^{\vee}\wedge \omega_{\theta,(A,p,q;\mathcal{A})}^{\vee}+2\pi \mathbf{i}B_{(\theta,A)}^{\vee} \\
&=\frac{\mathbf{i}}{4\pi} d\check{x}^t \mathcal{A}^{\theta} \theta \left( \mathcal{A}^{\theta} \right)^t d\check{x}+\pi \mathbf{i}d\check{x}^t A^t \theta A d\check{x}.
\end{align*}
In this context, sometimes the 2-form $d\omega_{\theta,(A,p,q;\mathcal{A})}^{\vee}+\omega_{\theta,(A,p,q;\mathcal{A})}^{\vee}\wedge \omega_{\theta,(A,p,q;\mathcal{A})}^{\vee}$ is called the local curvature form of $\nabla_{\theta,(A,p,q;\mathcal{A})}^{\vee}$ \cite{k-theory}. Then, we can easily verify that the equality
\begin{equation*}
\Omega_{\theta,(A,p,q;\mathcal{A})}^{\vee}-d\omega_{\theta,(A,p,q;\mathcal{A})}^{\vee}=\Bigl. 2\pi \mathbf{i} \bigl( B^{\vee}+B_{\theta}^{\vee} \bigr) \Bigr|_{L_{(A,p)}}
\end{equation*}
holds, and this is a generalization of the condition (\ref{curv_b_field}) to the setting which includes the deformation parameter $\theta$.

\subsection{The moduli space of $\mathcal{L}_{\theta,(A,p,q;\mathcal{A})}^{\nabla}$}
The purpose of this subsection is to specify the moduli space of $\mathcal{L}_{\theta,(A,p,q;\mathcal{A})}^{\nabla}$ for fixed matrices $A\in M(n;\mathbb{Z})$, $\mathcal{A}\in \mathrm{Sym}(n;\mathbb{R})$. In particular, in Theorem \ref{main_theorem_2}, we prove the analogue of Theorem \ref{main_theorem_1} in the symplectic geometry side. Moreover, by using those Theorem \ref{main_theorem_1} and Theorem \ref{main_theorem_2}, in Theorem \ref{main_theorem_3}, we give a generalization of the SYZ transform 
\begin{equation*}
\mathcal{L}_{(A,p,q)}^{\nabla} \ \longmapsto \ E_{(A,p,q)}^{\nabla}
\end{equation*}
to the setting which includes the noncommutative parameter $\theta$.

First, we give the following lemma.
\begin{lemma} \label{curvature_compare}
For given parameters $(A,p,q)\in M(n;\mathbb{Z})\times \mathbb{R}^n\times \mathbb{R}^n$, $\mathcal{A}$, $\mathcal{B}\in \mathrm{Sym}(n;\mathbb{R})$, 
\begin{equation*}
\Omega_{\theta,(A,p,q;\mathcal{A})}=\Omega_{\theta,(A,p,q;\mathcal{B})}
\end{equation*}
holds if and only if
\begin{equation*}
\Omega_{\theta,(A,p,q;\mathcal{A})}^{\vee}=\Omega_{\theta,(A,p,q;\mathcal{B})}^{\vee}
\end{equation*}
holds.
\end{lemma}
\begin{proof}
Let us first recall the definitions of $\Omega_{\theta,(A,p,q;\mathcal{A})}$ and $\Omega_{\theta,(A,p,q;\mathcal{A})}^{\vee}$. As computed in Lemma \ref{curvature_theta}, $\Omega_{\theta,(A,p,q;\mathcal{A})}$ is expressed locally as
\begin{equation*}
\Omega_{\theta,(A,p,q;\mathcal{A})}=\pi \mathbf{i}dx^t A^t \left( I_n-\left( \frac{1}{2\pi}\theta \mathcal{A} \right)^2 \right) \theta Adx-2\pi \mathbf{i} dx^t A^t \left( I_n-\left( \frac{1}{2\pi}\theta \mathcal{A} \right)^2 \right)dy+\frac{\mathbf{i}}{4\pi}dy^t \mathcal{A} \theta \mathcal{A} dy.
\end{equation*} 
On the other hand, as verified at the last of subsection 4.1, $\Omega_{\theta,(A,p,q;\mathcal{A})}^{\vee}$ is expressed locally as
\begin{equation*}
\Omega_{\theta,(A,p,q;\mathcal{A})}^{\vee}=\frac{\mathbf{i}}{4\pi} d\check{x}^t \mathcal{A}^{\theta} \theta \left( \mathcal{A}^{\theta} \right)^t d\check{x}+\pi \mathbf{i}d\check{x}^t A^t \theta A d\check{x}.
\end{equation*}
In particular, we see that $\Omega_{\theta,(A,p,q;\mathcal{A})}^{(0,2)}$ and $\Omega_{\theta,(A,p,q;\mathcal{A})}^{\vee}$ are determined by the same matrix $A^t \theta A\in \mathrm{Alt}(n;\mathbb{R})$ in the case $\mathcal{A}=O$:
\begin{equation*}
\Omega_{\theta,(A,p,q;\mathcal{A})}^{(0,2)}=\pi \mathbf{i} \left( dx^{(0,1)} \right)^t A^t \theta A dx^{(0,1)}, \ \ \ \Omega_{\theta,(A,p,q;\mathcal{A})}^{\vee}=\pi \mathbf{i}d\check{x}^t A^t \theta A d\check{x}.
\end{equation*}
Note that $\mathcal{A}^{\theta} \theta (\mathcal{A}^{\theta})^t$ and $\mathcal{B}^{\theta} \theta (\mathcal{B}^{\theta})^t$ turn out to be
\begin{align*}
\mathcal{A}^{\theta} \theta \left( \mathcal{A}^{\theta} \right)^t&=\left( I_n+\frac{1}{2\pi}\mathcal{A} \theta \right)^{-1} \left( I_n-\frac{1}{2\pi}\mathcal{A} \theta \right)^{-1} \mathcal{A} \theta \mathcal{A} \\
&=\left( I_n-\left( \frac{1}{2\pi}\mathcal{A} \theta \right)^2 \right)^{-1} \mathcal{A} \theta \mathcal{A}
\end{align*}
and
\begin{align*}
\mathcal{B}^{\theta} \theta \left( \mathcal{B}^{\theta} \right)^t&=\mathcal{B} \theta \mathcal{B} \left( I_n+\frac{1}{2\pi}\theta \mathcal{B} \right)^{-1} \left( I_n-\frac{1}{2\pi}\theta \mathcal{B} \right)^{-1} \\
&=\mathcal{B} \theta \mathcal{B} \left( I_n-\left( \frac{1}{2\pi}\theta \mathcal{B} \right)^2 \right)^{-1}, 
\end{align*}
respectively. Hence, $\mathcal{A}^{\theta} \theta (\mathcal{A}^{\theta})^t=\mathcal{B} \theta (\mathcal{B})^t$ if and only if
\begin{equation*}
\left( I_n-\left( \frac{1}{2\pi}\mathcal{A} \theta \right)^2 \right)^{-1} \mathcal{A} \theta \mathcal{A}=\mathcal{B} \theta \mathcal{B} \left( I_n-\left( \frac{1}{2\pi}\theta \mathcal{B} \right)^2 \right)^{-1},
\end{equation*} 
i.e., $\mathcal{A} \theta \mathcal{A}=\mathcal{B} \theta \mathcal{B}$. In other words, the condition that $\Omega_{\theta,(A,p,q;\mathcal{A})}^{\vee}=\Omega_{\theta,(A,p,q;\mathcal{B})}^{\vee}$ holds is equivalent to that the equality $\mathcal{A} \theta \mathcal{A}=\mathcal{B} \theta \mathcal{B}$ holds. Moreover, by the definition of $\Omega_{\theta,(A,p,q;\mathcal{A})}$, it is clear that $\mathcal{A} \theta \mathcal{A}=\mathcal{B} \theta \mathcal{B}$ holds if and only if $\Omega_{\theta,(A,p,q;\mathcal{A})}=\Omega_{\theta,(A,p,q;\mathcal{B})}$ holds. This completes the proof.
\end{proof}

Now, based on \cite{bijection}, for two fixed matrices $A\in M(n;\mathbb{Z})$, $\mathcal{A}\in \mathrm{Sym}(n;\mathbb{R})$ and $p$, $q$, $p'$, $q'\in \mathbb{R}^n$, we give the definition of $\mathcal{L}_{\theta,(A,p,q;\mathcal{A})}^{\nabla}$ and $\mathcal{L}_{\theta,(A,p',q';\mathcal{A})}^{\nabla}$ being isomorphic to each other as follows. If there exists a symplectic automorphism $\Phi : \check{X}_{\theta}^n \stackrel{\sim}{\to} \check{X}_{\theta}^n$ such that
\begin{align}
&\Phi^{-1}\left( L_{(A,p')} \right)=L_{(A,p)}, \label{L_equiv_1} \\
&\Phi^*\mathcal{O}_{\theta,(A,p',q';\mathcal{A})}^{\nabla}\cong \mathcal{O}_{\theta,(A,p,q;\mathcal{A})}^{\nabla}, \label{L_equiv_2}
\end{align} 
we say that $\mathcal{L}_{\theta,(A,p,q;\mathcal{A})}^{\nabla}$ is isomorphic to $\mathcal{L}_{\theta,(A,p',q';\mathcal{A})}^{\nabla}$, and write
\begin{equation*}
\mathcal{L}_{\theta,(A,p,q;\mathcal{A})}^{\nabla}\cong \mathcal{L}_{\theta,(A,p',q';\mathcal{A})}^{\nabla}.
\end{equation*}
Here, we present the analogue of Theorem \ref{main_theorem_1} in the symplectic geometry side.
\begin{theo} \label{main_theorem_2}
We take matrices $A\in M(n;\mathbb{Z})$, $\mathcal{A}\in \mathrm{Sym}(n;\mathbb{R})$ arbitrary, and fix them. Then two objects $\mathcal{L}_{\theta,(A,p,q;\mathcal{A})}^{\nabla}$ and $\mathcal{L}_{\theta,(A,p',q';\mathcal{A})}^{\nabla}$ are isomorphic to each other,
\begin{equation*}
\mathcal{L}_{\theta,(A,p,q;\mathcal{A})}^{\nabla}\cong \mathcal{L}_{\theta,(A,p',q';\mathcal{A})}^{\nabla},
\end{equation*}
if and only if there exist $k$, $l\in \mathbb{Z}^n$ such that
\begin{equation*}
\left( \begin{array}{cc} p' \\ q' \end{array} \right)=\left( \begin{array}{cc} p \\ q \end{array} \right)+\left( \begin{array}{cc} k \\ l \end{array} \right).
\end{equation*}
\end{theo}
\begin{proof}
Since two Lagrangian submanifolds $L_{(A,p)}$ and $L_{(A,p')}$ which are defined by using a fixed matrix $A\in M(n;\mathbb{Z})$ are affine, we can take the map
\begin{equation*}
\Phi=\mathrm{id}_{\check{X}_{\theta}^n}
\end{equation*}
as a symplectic automorphism $\Phi : \check{X}_{\theta}^n \stackrel{\sim}{\to} \check{X}_{\theta}^n$ satisfying the condition (\ref{L_equiv_1}), in the case $L_{(A,p)}=L_{(A,p')}$ only. In particular, $L_{(A,p)}=L_{(A,p')}$ holds if and only if there exists $k\in \mathbb{Z}^n$ such that $p'=p+k$. Then the condition (\ref{L_equiv_2}) becomes
\begin{equation*}
\mathcal{O}_{\theta,(A,p,q;\mathcal{A})}^{\nabla}\cong \mathcal{O}_{\theta,(A,p',q';\mathcal{A})}^{\nabla},
\end{equation*}   
so it is enough to discuss when $\mathcal{O}_{\theta,(A,p,q;\mathcal{A})}^{\nabla}\cong \mathcal{O}_{\theta,(A,p',q';\mathcal{A})}^{\nabla}$ holds on $L_{(A,p)}=L_{(A,p')}$. 

Below, we compute an isomorphism $\varphi : \mathcal{O}_{\theta,(A,p,q;\mathcal{A})}^{\nabla}\stackrel{\sim}{\to} \mathcal{O}_{\theta,(A,p',q';\mathcal{A})}^{\nabla}$ explicitly. By solving the differential equation
\begin{equation*}
d\varphi+2\pi \mathbf{i}q'^t d\check{x}\wedge \varphi-\varphi \wedge 2\pi \mathbf{i}q^t d\check{x}=0
\end{equation*}
which is equivalent to $\nabla_{\theta,(A,p',q';\mathcal{A})}^{\vee}\varphi-\varphi \nabla_{\theta,(A,p,q;\mathcal{A})}^{\vee}=0$, we have the solution which is expressed locally as
\begin{equation*}
\varphi(\check{x})=C \ \mathrm{exp} \left( -2\pi \mathbf{i} (q'^t-q^t)\check{x} \right), 
\end{equation*}
where $C\in \mathbb{C}$ is an arbitrary constant. Besides, although we need to consider the condition such that $\varphi$ is compatible with the factors of automorphy corresponding to $\mathcal{O}_{\theta,(A,p,q;\mathcal{A})}^{\nabla}$ and $\mathcal{O}_{\theta,(A,p',q';\mathcal{A})}^{\nabla}$, it is easy since $\mathcal{O}_{\theta,(A,p,q;\mathcal{A})}^{\nabla}$ coincides with $\mathcal{O}_{\theta,(A,p',q';\mathcal{A})}^{\nabla}$ when we forget the differential structures $\nabla_{\theta,(A,p,q;\mathcal{A})}^{\vee}$, $\nabla_{\theta,(A,p',q';\mathcal{A})}^{\vee}$ of them. Namely, we may consider when the equality
\begin{equation}
\varphi(\check{x}+\bm{e}_i)=\varphi(\check{x}) \label{main_theorem_2_1}
\end{equation} 
holds for each $i=1, \ldots, n$. Clearly, for each fixed $i=1, \ldots, n$, we see 
\begin{equation*}
\varphi(\check{x}+\bm{e}_i)=\mathrm{exp} \left( -2\pi \mathbf{i} (q'^t-q^t)\bm{e}_i \right) \varphi(\check{x}).
\end{equation*}
Hence, if there exists $l\in \mathbb{Z}^n$ such that $q'-q=l$, the relation (\ref{main_theorem_2_1}) is satisfied since $\mathrm{exp}(-2\pi \mathbf{i} (q'^t-q^t)\bm{e}_i)=\mathrm{exp}(-2\pi \mathbf{i} l^t \bm{e}_i)=1$, and then, the isomorphism $\varphi$ is expressed locally as
\begin{equation*}
\varphi(\check{x})=C \ \mathrm{exp} \left( -2\pi \mathbf{i} l^t \check{x} \right)
\end{equation*}
with $C\not=0$. If not, we must set $C=0$ in order to satisfy the relation (\ref{main_theorem_2_1}) since there exists $i=1, \ldots, n$ such that $\mathrm{exp}(-2\pi \mathbf{i} (q'^t-q^t)\bm{e}_i)\not=1$. This completes the proof. 

In particular, the statement of this theorem implies that the moduli space of $\mathcal{L}_{\theta,(A,p,q;\mathcal{A})}^{\nabla}$ for fixed matrices $A\in M(n;\mathbb{Z})$, $\mathcal{A}\in \mathrm{Sym}(n;\mathbb{R})$ is homeomorphic to the $2n$-dimensional real torus
\begin{equation*}
\mathbb{R}^{2n}/\mathbb{Z}^{2n}\approx \bigl( \mathbb{R}^n/\mathbb{Z}^n \bigr) \times \bigl( \mathbb{R}^n/\mathbb{Z}^n \bigr).
\end{equation*}  
\end{proof}

We denote the set of isomorphism classes (the moduli space) of objects $\mathcal{L}_{\theta,(A,p,q;\mathcal{A})}^{\nabla}$ for fixed matrices $A\in M(n;\mathbb{Z})$, $\mathcal{A}\in \mathrm{Sym}(n;\mathbb{R})$ by
\begin{equation*}
\mathcal{M}_{\theta}^{\vee}(A;\mathcal{A}).
\end{equation*}

As a conclusion, for fixed matrices $A\in M(n;\mathbb{Z})$, $\mathcal{A}\in \mathrm{Sym}(n;\mathbb{R})$, we see that both the moduli space of $E_{\theta,(A,p,q;\mathcal{A})}^{\nabla}$ and the moduli space of $\mathcal{L}_{\theta,(A,p,q;\mathcal{A})}^{\nabla}$ are homeomorphic to $2n$-dimensional real tori by Theorem \ref{main_theorem_1} and Theorem \ref{main_theorem_2}. Thus, we obtain the following theorem which is also a part of the statement of the homological mirror symmetry for $(X_{\theta}^n, \check{X}_{\theta}^n)$.   
\begin{theo} \label{main_theorem_3}
We take matrices $A\in M(n;\mathbb{Z})$, $\mathcal{A}\in \mathrm{Sym}(n;\mathbb{R})$ arbitrary, and fix them. Then there exists a bijection between $\mathcal{M}_{\theta}^{\vee}(A;\mathcal{A})$ and $\mathcal{M}_{\theta}(A;\mathcal{A})$.
\end{theo}
\begin{proof}
This is essentially a corollary of Theorem \ref{main_theorem_1} and Theorem \ref{main_theorem_2}. By Theorem \ref{main_theorem_1} and Theorem \ref{main_theorem_2}, it is easy to verify that the map which is defined by
\begin{equation*}
\mathcal{L}_{\theta,(A,p,q;\mathcal{A})}^{\nabla} \ \longmapsto \ E_{\theta,(A,(I_n-(\frac{1}{2\pi}\theta \mathcal{A})^2)^t p,-A^t \theta (I_n-(\frac{1}{2\pi}\theta \mathcal{A})^2)^t p+q;\mathcal{A})}^{\nabla}
\end{equation*}
induces a bijection $\mathcal{M}_{\theta}^{\vee}(A;\mathcal{A}) \stackrel{\sim}{\to} \mathcal{M}_{\theta}(A;\mathcal{A})$. Actually, the map which is defined by
\begin{equation*}
E_{\theta,(A,p,q;\mathcal{A})}^{\nabla} \ \longmapsto \ \mathcal{L}_{\theta,(A,((I_n-(\frac{1}{2\pi}\theta \mathcal{A})^2)^{-1})^t p,A^t \theta p+q;\mathcal{A})}^{\nabla}.
\end{equation*}
induces a map $\mathcal{M}_{\theta}(A;\mathcal{A}) \to \mathcal{M}_{\theta}^{\vee}(A;\mathcal{A})$, and this is the inverse map of the above map $\mathcal{M}_{\theta}^{\vee}(A;\mathcal{A}) \to \mathcal{M}_{\theta}(A;\mathcal{A})$.
\end{proof}
Finally, by combining Theorem \ref{main_theorem_3} with Lemma \ref{curvature_compare}, we see that the correspondence
\begin{equation*}
\mathcal{L}_{\theta,(A,p,q;\mathcal{A})}^{\nabla} \ \longmapsto \ E_{\theta,(A,(I_n-(\frac{1}{2\pi}\theta \mathcal{A})^2)^t p,-A^t \theta (I_n-(\frac{1}{2\pi}\theta \mathcal{A})^2)^t p+q;\mathcal{A})}^{\nabla}
\end{equation*}
is a generalization of the SYZ transform
\begin{equation*}
\mathcal{L}_{(A,p,q)}^{\nabla} \ \longmapsto \ E_{(A,p,q)}^{\nabla}
\end{equation*}
on a mirror pair $(X^n, \check{X}^n)$ to the setting which includes the noncommutative parameter $\theta$.

\section{Concluding remarks}
In this paper, based on \cite{nc}, we have constructed the deformation $E_{\theta,(A,p,q;\mathcal{A})}^{\nabla}$ of $E_{(A,p,q)}^{\nabla}$ including the effect by ambiguity arising from the isomorphism (\ref{trivialization}). Moreover, we have computed the objects $\mathcal{L}_{\theta,(A,p,q;\mathcal{A})}^{\nabla}$ which are mirror dual to noncommutative objects $E_{\theta,(A,p,q;\mathcal{A})}^{\nabla}$, and actually, the SYZ transform between them is also described in Theorem \ref{main_theorem_3} explicitly.

On the other hand, as mentioned in subsection 4.1, the construction of the deformed objects $\mathcal{L}_{\theta,(A,p,q;\mathcal{A})}^{\nabla}$ stands on the idea which is proposed in our previous paper \cite{b-field}. Perhaps it is no coincidence that we can apply the idea proposed in \cite{b-field} to our setting, and it seems that they are related by the Fourier-Mukai transform. Below, by employing generalized complex geometry, we explain this perspective briefly.

As a toy model, let us first focus on the mirror pair of elliptic curves $(X^1, \check{X}^1)$. Note that $X^1$ is not deformed in the sense described in this paper\footnote{As discussed in \cite{p-s-nc}, we can consider the (non-trivial) noncommutative deformation of $X^1$ associated to the Poisson bivector (\ref{P_bivector}) such that $\theta_2 \in M(n;\mathbb{R})$ and $\theta_1=\theta_3=O$.} since $\Pi_{\theta}=0$. In general, we can consider the $SL(2;\mathbb{Z})$-action (up to shifts) over $D^b(Coh(X^1))$. On the other hand, in the symplectic geometry side, this $SL(2;\mathbb{Z})$-action over $D^b(Coh(X^1))$ is translated to the $SL(2;\mathbb{Z})$-action on $\check{X}^1$ via the homological mirror symmetry. Precisely speaking, for a given element $g\in SL(2;\mathbb{Z})$, we can define a symplectic automorphism $\varphi_g : \check{X}^1 \stackrel{\sim}{\to} \check{X}^1$ by $\varphi_g(\check{x})=g\check{x}$, and this $\varphi_g$ induces autoequivalences over $Fuk(\check{X}^1)$ and $Tr(Fuk(\check{X}^1))$:
\begin{equation*}
\left( L_{(A,p)}, \ \mathcal{O}_{(A,p,q)}^{\nabla} \right) \ \longmapsto \ \left( \varphi_g \left( L_{(A,p)} \right), \ \left( \varphi_g^{-1} \right)^* \mathcal{O}_{(A,p,q)}^{\nabla} \right).
\end{equation*}
In particular, a generator
\begin{equation*}
g_{FM}:=\left( \begin{array}{cc} 0 & 1 \\ -1 & 0 \end{array} \right)\in SL(2;\mathbb{Z})
\end{equation*}
corresponds to the Fourier-Mukai transform over $D^b(Coh(X^1))$ and its mirror transform $\varphi_{g_{FM}} : \check{X}^1 \stackrel{\sim}{\to} \check{X}^1$. Here, we denote the representation matrices with respect to the standard bases in the formulas (\ref{g_c_c}), (\ref{g_c_s}) by the same notations $\mathcal{I}_T$, $\check{\mathcal{I}}_T$, respectively. Then, from the viewpoint of generalized complex geometry, the symplectic automorphism $\varphi_{g_{FM}}$ induces the transform of $\check{\mathcal{I}}_T$:
\begin{equation*}
\check{\mathcal{I}}_T \ \longmapsto \ \hat{\check{\mathcal{I}}}_T:=\left( \begin{array}{cccc} 0 & 1 & 0 & 0 \\ -1 & 0 & 0 & 0 \\ 0 & 0 & 0 & 1 \\ 0 & 0 & -1 & 0 \end{array} \right) \check{\mathcal{I}}_T \left( \begin{array}{cccc} 0 & -1 & 0 & 0 \\ 1 & 0 & 0 & 0 \\ 0 & 0 & 0 & -1 \\ 0 & 0 & 1 & 0 \end{array} \right),
\end{equation*}
and by the mirror transform (\ref{mirror}), we have
\begin{equation*}
M_{(1)} \hspace{0.5mm} \hat{\check{\mathcal{I}}}_T \hspace{0.5mm} M_{(1)}=\left( \begin{array}{cccc} 0 & 0 & 0 & 1 \\ 0 & 0 & -1 & 0 \\ 0 & 1 & 0 & 0 \\ -1 & 0 & 0 & 0 \end{array} \right) \mathcal{I}_T \left( \begin{array}{cccc} 0 & 0 & 0 & -1 \\ 0 & 0 & 1 & 0 \\ 0 & -1 & 0 & 0 \\ 1 & 0 & 0 & 0 \end{array} \right).
\end{equation*}
Therefore, it is natural to expect that the transform
\begin{equation*}
\mathcal{I}_T \ \longmapsto \ \left( \begin{array}{cccc} 0 & 0 & 0 & 1 \\ 0 & 0 & -1 & 0 \\ 0 & 1 & 0 & 0 \\ -1 & 0 & 0 & 0 \end{array} \right) \mathcal{I}_T \left( \begin{array}{cccc} 0 & 0 & 0 & -1 \\ 0 & 0 & 1 & 0 \\ 0 & -1 & 0 & 0 \\ 1 & 0 & 0 & 0 \end{array} \right)
\end{equation*}
corresponds to the Fourier-Mukai transform over $D^b(Coh(X^1))$ from the viewpoint of generalized complex geometry.

Now, based on this toy model, let us consider the following transform:
\begin{equation*}
\mathcal{I}_T \ \longmapsto \ \hat{\mathcal{I}}_T:=\left( \begin{array}{cccc} O & O & O & I_n \\ O & O & -I_n & O \\ O & I_n & O & O \\ -I_n & O & O & O \end{array} \right) \mathcal{I}_T \left( \begin{array}{cccc} O & O & O & -I_n \\ O & O & I_n & O \\ O & -I_n & O & O \\ I_n & O & O & O \end{array} \right)
\end{equation*}
It seems that this transform corresponds to the Fourier-Mukai transform. In fact, we can interpret $\hat{\mathcal{I}}_T$ as the generalized complex structure induced from the complex structure of $\mathbb{C}^n/(\mathbb{Z}^n\oplus T^t \mathbb{Z}^n)$ which is the dual of $X^n$ (the dual of $X^1$ is $X^1$ itself). Moreover, we have
\begin{equation}
M_{(n)} \hspace{0.5mm} \hat{\mathcal{I}}_T \hspace{0.5mm} M_{(n)}=\left( \begin{array}{cccc} O & I_n & O & O \\ -I_n & O & O & O \\ O & O & O & I_n \\ O & O & -I_n & O \end{array} \right) \check{\mathcal{I}}_T \left( \begin{array}{cccc} O & -I_n & O & O \\ I_n & O & O & O \\ O & O & O & -I_n \\ O & O & I_n & O \end{array} \right) \label{f_m_s}
\end{equation}
via the mirror transform (\ref{mirror}). For simplicity, we put
\begin{equation*}
\hat{\check{\mathcal{I}}}_T:=\left( \begin{array}{cccc} O & I_n & O & O \\ -I_n & O & O & O \\ O & O & O & I_n \\ O & O & -I_n & O \end{array} \right) \check{\mathcal{I}}_T \left( \begin{array}{cccc} O & -I_n & O & O \\ I_n & O & O & O \\ O & O & O & -I_n \\ O & O & I_n & O \end{array} \right).
\end{equation*}
The right hand side of the equality (\ref{f_m_s}) implies that the transform $\check{\mathcal{I}}_T \mapsto \hat{\check{\mathcal{I}}}_T$ is induced from the map which is defined by
\begin{equation}
\check{x} \ \longmapsto \ \left( \begin{array}{cc} O & I_n \\ -I_n & O \end{array} \right) \check{x}, \label{FM_symplectic}
\end{equation}
and this is a generalization of the $SL(2;\mathbb{Z})$-action $\varphi_{g_{FM}}$ on $\check{X}^1$ to the higher dimensional case. Here, note that it is no longer a symplectic automorphism on $\check{X}^n$.

Let us consider the above transforms on the mirror pair $(X_{\theta}^n, \check{X}_{\theta}^n)$. Similarly as in the cases of $\mathcal{I}_T$ and $\check{\mathcal{I}}_T$, we denote the representation matrices with respect to the standard bases in the formulas (\ref{g_c_c_theta}), (\ref{g_c_s_theta}) by $\mathcal{I}_T(\Pi_{\theta})$, $\check{\mathcal{I}}_T(\Pi_{\theta})$, respectively. Then, by direct calculations, we see that the equality
\begin{align*}
&M_{(n)} \left\{ \left( \begin{array}{cccc} O & O & O & I_n \\ O & O & -I_n & O \\ O & I_n & O & O \\ -I_n & O & O & O \end{array} \right) \mathcal{I}_T(\Pi_{\theta}) \left( \begin{array}{cccc} O & O & O & -I_n \\ O & O & I_n & O \\ O & -I_n & O & O \\ I_n & O & O & O \end{array} \right) \right\} M_{(n)} \\
&=\left( \begin{array}{cccc} O & I_n & O & O \\ -I_n & O & O & O \\ O & O & O & I_n \\ O & O & -I_n & O \end{array} \right) \check{\mathcal{I}}_T(\Pi_{\theta}) \left( \begin{array}{cccc} O & -I_n & O & O \\ I_n & O & O & O \\ O & O & O & -I_n \\ O & O & I_n & O \end{array} \right) 
\end{align*}
holds. This equality implies that the mirror transform (\ref{mirror}) is compatible with both the Fourier-Mukai transforms in the above sense. Moreover, we have
\begin{align*}
&\left( \begin{array}{cccc} O & I_n & O & O \\ -I_n & O & O & O \\ O & O & O & I_n \\ O & O & -I_n & O \end{array} \right) \check{\mathcal{I}}_T(\Pi_{\theta}) \left( \begin{array}{cccc} O & -I_n & O & O \\ I_n & O & O & O \\ O & O & O & -I_n \\ O & O & I_n & O \end{array} \right) \\
&=\left( \begin{array}{cccc} I_n & O & O & O \\ O & I_n & O & O \\ -\theta & O & I_n & O \\ O & O & O & I_n \end{array} \right) \hat{\check{\mathcal{I}}}_T \left( \begin{array}{cccc} I_n & O & O & O \\ O & I_n & O & O \\ \theta & O & I_n & O \\ O & O & O & I_n \end{array} \right),
\end{align*}
and the right hand side can be interpreted as the B-field transform of $\hat{\check{\mathcal{I}}}_T$ in the symplectic geometry side in the context of our previous paper \cite{b-field}. On the other hand, in the complex geometry side, we have 
\begin{align*}
&\left( \begin{array}{cccc} O & O & O & I_n \\ O & O & -I_n & O \\ O & I_n & O & O \\ -I_n & O & O & O \end{array} \right) \mathcal{I}_T(\Pi_{\theta}) \left( \begin{array}{cccc} O & O & O & -I_n \\ O & O & I_n & O \\ O & -I_n & O & O \\ I_n & O & O & O \end{array} \right) \\
&=\left( \begin{array}{cccc} I_n & O & O & O \\ O & I_n & O & O \\ -\theta & O & I_n & O \\ O & O & O & I_n \end{array} \right) \hat{\mathcal{I}}_T \left( \begin{array}{cccc} I_n & O & O & O \\ O & I_n & O & O \\ \theta & O & I_n & O \\ O & O & O & I_n \end{array} \right).
\end{align*}
Here, the right hand side can be interpreted as the B-field transform of $\hat{\mathcal{I}}_T$ in the context of our previous paper \cite{b-field}. Thus, by summarizing the above discussions, we can expect that the Fourier-Mukai partner of the noncommutative complex torus $X_{\theta}^n$ is a gerby deformation of the Fourier-Mukai partner of $X^n$ in the sense of \cite{b-field}. In particular, in the symplectic geometry side, the flat gerbe $\check{\mathcal{G}}_{\theta,A}^{\nabla}$ on $L_{(A,p)}$ and a flat gerbe such as in \cite{b-field} are related by the map (\ref{FM_symplectic}). We will comment on this description again at the last of subsection 6.1.  

Finally, we comment on several related works. In \cite{nc-fm}, by focusing on the noncommutative deformations associated to the formal deformation quantizations, the Fourier-Mukai partners of noncommutative complex tori are investigated. On the other hand, in light of \cite{nc-fm}, the Fourier-Mukai partners of noncommutative complex tori for nonformal parameters are discussed in \cite{okuda} (see also \cite{block-1, block-2}). 

\section{Appendices}
As commented in section 5, our previous paper \cite{b-field} is closely related to this paper, and in \cite{gerby}, another gerby deformation of $X^n$ is treated. However, during this study period, similar errors are found in both the articles \cite{b-field, gerby} even though they are already published. Throughout this section, we improve such errors.

In order to explain the details, let us consider a B-field on $X^n$ which is expressed locally as
\begin{equation}
\sum_{i, j=1}^n \left( \frac{1}{2} (\tau_1)_{ij} dx_i \wedge dx_j +(\tau_2)_{ij} dx_i \wedge dy_j +\frac{1}{2} (\tau_3)_{ij} dy_i \wedge dy_j \right), \label{b_field_X_n} 
\end{equation}
where $\tau_1$, $\tau_3 \in \mathrm{Alt}(n;\mathbb{R})$ and $\tau_2 \in M(n;\mathbb{R})$. Explicitly, in \cite{b-field}, the case of $\tau_1$ type, i.e., the case that $\tau_1 \in \mathrm{Alt}(n;\mathbb{R})$ and $\tau_2 =\tau_3 =O$ is discussed, and in \cite{gerby}, we focus on the case of $\tau_2$ type, i.e., the case that $\tau_2 \in M(n;\mathbb{R})$ and $\tau_1 =\tau_3 =O$. In those papers, anyway, we use the flat gerbes whose 1-connections are determined by the B-field (\ref{b_field_X_n}). However, precisely, in this context, what we should consider is the flat gerbes whose 1-connections are determined by the (0,2)-part of the B-field (\ref{b_field_X_n}) rather than the flat gerbes which are employed in \cite{b-field, gerby}. In particular, although both \cite[Proposition 5.6]{b-field} and \cite[Proposition 4.4]{gerby} state that the holomorphicity of each $E_{(A,p,q)}^{\nabla}$ is preserved under the corresponding gerby deformations, actually, we obtain the conclusion that those statements are also incorrect when we consider them correctly. In this section, we modify these arguments.

\subsection{The gerby deformation of type $\tau_1$}
We take an arbitrary $\tau \in \mathrm{Alt}(n;\mathbb{R})$, and fix it. In this subsection, we revise several results written in \cite{b-field}.

\subsubsection{Complex geometry side}
Let us consider the B-field (\ref{b_field_X_n}) in the case that $\tau_1=\tau$ and $\tau_2=\tau_3=O$, and denote it by ${}_{\tau}B$:
\begin{equation*}
{}_{\tau}B:=\frac{1}{2} dx^t \tau dx.
\end{equation*}
From the viewpoint of generalized complex geometry, this B-field ${}_{\tau}B$ causes the B-field transform of $\mathcal{I}_T$ over $X^n$:
\begin{align*}
&\mathcal{I}_T({}_{\tau}B) \left( \frac{\partial}{\partial x}^t, \frac{\partial}{\partial y}^t, dx^t, dy^t \right) \\
&=\left( \frac{\partial}{\partial x}^t, \frac{\partial}{\partial y}^t, dx^t, dy^t \right) \left( \begin{array}{cccc} I_n & O & O & O \\ O & I_n & O & O \\ -\tau & O & I_n & O \\ O & O & O & I_n \end{array} \right) \mathcal{I}_T \left( \begin{array}{cccc} I_n & O & O & O \\ O & I_n & O & O \\ \tau & O & I_n & O \\ O & O & O & I_n \end{array} \right).
\end{align*}
Then the following proposition holds.
\begin{proposition} \label{comp_B_field}
The B-field transform associated to the B-field ${}_{\tau}B$ preserves the generalized complex structure $\mathcal{I}_T$ over $X^n$, i.e., 
\begin{equation*}
\mathcal{I}_T({}_{\tau}B)=\mathcal{I}_T
\end{equation*}
if and only if $\tau=O$.
\end{proposition}
\begin{proof}
For simplicity, we set
\begin{equation*}
I_T:=\left( \begin{array}{cc} -XY^{-1} & -Y-XY^{-1}X \\ Y^{-1} & Y^{-1}X \end{array} \right), \ \ \ {}_{\tau}\underline{B}:=\left( \begin{array}{cc} \tau & O \\ O & O \end{array} \right).
\end{equation*}
Then, we may prove the following statement: the equality
\begin{equation}
\left( \begin{array}{cc} I_{2n} & O \\ -{}_{\tau}\underline{B} & I_{2n} \end{array} \right) \left( \begin{array}{cc} I_T & O \\ O & -I_T^t \end{array} \right) \left( \begin{array}{cc} I_{2n} & O \\ {}_{\tau}\underline{B} & I_{2n} \end{array} \right)=\left( \begin{array}{cc} I_T & O \\ O & -I_T^t \end{array} \right) \label{B_eq}
\end{equation}
holds if and only if $\tau=O$. Here, the equality (\ref{B_eq}) holds if and only if ${}_{\tau}\underline{B} I_T+I_T^t {}_{\tau}\underline{B}=O$ since we can calculate
\begin{equation*}
\left( \begin{array}{cc} I_{2n} & O \\ -{}_{\tau}\underline{B} & I_{2n} \end{array} \right) \left( \begin{array}{cc} I_T & O \\ O & -I_T^t \end{array} \right) \left( \begin{array}{cc} I_{2n} & O \\ {}_{\tau}\underline{B} & I_{2n} \end{array} \right)=\left( \begin{array}{cc} I_T & O \\ -({}_{\tau}\underline{B} I_T+I_T^t {}_{\tau}\underline{B}) & -I_T^t \end{array} \right),
\end{equation*}
and the matrix ${}_{\tau}\underline{B} I_T+I_T^t {}_{\tau}\underline{B}$ turns out to be
\begin{equation*}
{}_{\tau}\underline{B} I_T+I_T^t {}_{\tau}\underline{B}=\left( \begin{array}{cc} -\tau XY^{-1}-(Y^{-1})^t X^t \tau & -\tau (Y+XY^{-1}X) \\ -(Y^t+X^t (Y^{-1})^t X^t)\tau & O \end{array} \right).
\end{equation*}
Therefore, we may consider the two conditions
\begin{align}
&-\tau XY^{-1}-(Y^{-1})^t X^t \tau =O, \label{B_cond_1} \\
&-\tau (Y+XY^{-1}X)=O. \label{B_cond_2}
\end{align}
As remarked in subsection 2.1, now, $\mathrm{det}(Y+XY^{-1}X)\not=0$ since $\mathrm{det}T\not=0$. Hence, the condition (\ref{B_cond_2}) indicates $\tau=O$, and then, the condition (\ref{B_cond_1}) is automatically satisfied. Conversely, it is clear that the two conditions (\ref{B_cond_1}), (\ref{B_cond_2}) are satisfied if $\tau=O$. Thus, indeed, the equality (\ref{B_eq}) holds if and only if $\tau=O$.
\end{proof}
On the other hand, by the definition of ${}_{\tau}B$, the vanishing of the (0,2)-part ${}_{\tau}B^{(0,2)}$ is equivalent to $\tau \in \mathrm{Sym}(n;\mathbb{R})$, i.e., $\tau=O$ (note that $\tau \in \mathrm{Alt}(n;\mathbb{R})$). As a result, we immediately obtain the following as a corollary of Proposition \ref{comp_B_field}. 
\begin{corollary} \label{comp_B_field_cor}
The B-field transform associated to the B-field ${}_{\tau}B$ preserves the generalized complex structure $\mathcal{I}_T$ over $X^n$, i.e., 
\begin{equation*}
\mathcal{I}_T({}_{\tau}B)=\mathcal{I}_T
\end{equation*}
if and only if the \textup{(}0,2\textup{)}-part ${}_{\tau}B^{(0,2)}$ vanishes.
\end{corollary}

Below, we give the strict definition of the flat gerbe whose 1-connection is determined by ${}_{\tau}B^{(0,2)}$. We set
\begin{equation*}
\Lambda(T,\tau):=((T-\bar{T})^{-1})^t T^t \tau T (T-\bar{T})^{-1},
\end{equation*}
namely,
\begin{equation*}
{}_{\tau}B^{(0,2)}=\frac{1}{2}d\bar{z}^t \Lambda(T,\tau) d\bar{z}.
\end{equation*}
For ${}_{\tau}B^{(0,2)}$, let us take a 1-form ${}_{\tau}\beta$ which is expressed locally as
\begin{equation*}
{}_{\tau}\beta=\frac{1}{2}\bar{z}^t \Lambda(T,\tau) d\bar{z},
\end{equation*}
and it is clear that ${}_{\tau}B^{(0,2)}=\bar{\partial}{}_{\tau}\beta$. We sometimes use the notation ${}_{\tau}\beta(z)$ instead of ${}_{\tau}\beta$ in order to specify the complex coordinate system $z$. For each $i=1, \ldots, n$, let us locally define ${}_{\tau}\omega(\bm{e}_i)$, ${}_{\tau}\omega(T\bm{e}_i)$ by
\begin{align*}
&{}_{\tau}\omega(\bm{e}_i)={}_{\tau}\beta(z+\bm{e}_i)-{}_{\tau}\beta(z)=\frac{1}{2}\bm{e}_i^t \Lambda(T,\tau) d\bar{z}, \\
&{}_{\tau}\omega(T\bm{e}_i)={}_{\tau}\beta(z+T\bm{e}_i)-{}_{\tau}\beta(z)=\frac{1}{2}\bm{e}_i^t \bar{T}^t \Lambda(T,\tau) d\bar{z},
\end{align*}
respectively. Now, for an arbitrary $i=1, \ldots, n$ and points $[z]$, $[z+\bm{e}_i]\in X^n$, there exist $(l;m)$, $(l';m')\in I$ such that $[z]\in O_m^l$, $[z+\bm{e}_i]\in O_{m'}^{l'}$, so let us denote the open set $O_m^l$ including $[z]$ and the open set $O_{m'}^{l'}$ including $[z+\bm{e}_i]$ by $O_{[z]}$ and $O_{[z+\bm{e}_i]}$, respectively. We use the notations $O_{[z+T\bm{e}_1]}, \ldots, O_{[z+T\bm{e}_n]}$ in this sense. Then ${}_{\tau}\omega(\bm{e}_1), \ldots, {}_{\tau}\omega(\bm{e}_n)$, ${}_{\tau}\omega(T\bm{e}_1), \ldots, {}_{\tau}\omega(T\bm{e}_n)$ can be considered as 1-forms which are defined on $O_{[z]}\cap O_{[z+\bm{e}_1]}, \ldots, O_{[z]}\cap O_{[z+\bm{e}_n]}$, $O_{[z]}\cap O_{[z+T\bm{e}_1]}, \ldots, O_{[z]}\cap O_{[z+T\bm{e}_n]}$, respectively. Let us consider the following locally defined functions for each $i=1, \ldots, n$:
\begin{equation*}
{}_{\tau}\xi(\bm{e}_i,z):=\mathrm{exp}\Bigl( -\pi \mathbf{i} \bm{e}_i^t \Lambda(T,\tau) \bar{z} \Bigr), \ \ \ {}_{\tau}\xi(T\bm{e}_i,z):=\mathrm{exp}\Bigl( -\pi \mathbf{i} \bm{e}_i^t \bar{T}^t \Lambda(T,\tau) \bar{z} \Bigr).
\end{equation*}
Since they satisfy the differential equations
\begin{equation*}
\left( \bar{\partial}+2\pi \mathbf{i}{}_{\tau}\omega(\bm{e}_i) \right) \bigl( \xi_{\tau}(\gamma_i,z) \bigr)=0, \ \ \ \left( \bar{\partial}+2\pi \mathbf{i}{}_{\tau}\omega(T\bm{e}_i) \right) \bigl( {}_{\tau}\xi(T\bm{e}_i,z) \bigr)=0,
\end{equation*}  
we can regard ${}_{\tau}\xi(\bm{e}_i,z)$ and ${}_{\tau}\xi(T\bm{e}_i,z)$ as a holomorphic section of the trivial complex line bundle on $O_{[z]}\cap O_{[z+\bm{e}_i]}$ equipped with the holomorphic structure $\bar{\partial}+2\pi \mathbf{i}{}_{\tau}\omega(\bm{e}_i)$ and a holomorphic section of the trivial complex line bundle on $O_{[z]}\cap O_{[z+T\bm{e}_i]}$ equipped with the holomophic structure $\bar{\partial}+2\pi \mathbf{i}{}_{\tau}\omega(T\bm{e}_i)$, respectively. We further define ${}_{\tau}\xi(\bm{e}_i+\bm{e}_j,z)$, ${}_{\tau}\xi(\bm{e}_i+T\bm{e}_j,z)$, ${}_{\tau}\xi(T\bm{e}_i+\bm{e}_j,z)$, ${}_{\tau}\xi(T\bm{e}_i+T\bm{e}_j,z)$ by
\begin{align*}
&{}_{\tau}\xi(\bm{e}_i+\bm{e}_j,z)={}_{\tau}\xi(\bm{e}_i,z+\bm{e}_j){}_{\tau}\xi(\bm{e}_j,z), \\
&{}_{\tau}\xi(\bm{e}_i+T\bm{e}_j,z)={}_{\tau}\xi(\bm{e}_i,z+T\bm{e}_j){}_{\tau}\xi(T\bm{e}_j,z), \\
&{}_{\tau}\xi(T\bm{e}_i+\bm{e}_j,z)={}_{\tau}\xi(T\bm{e}_i,z+\bm{e}_j){}_{\tau}\xi(\bm{e}_j,z), \\
&{}_{\tau}\xi(T\bm{e}_i+T\bm{e}_j,z)={}_{\tau}\xi(T\bm{e}_i,z+T\bm{e}_j){}_{\tau}\xi(T\bm{e}_j,z), 
\end{align*}
respectively, where $i$, $j=1, \ldots, n$. Now, by using them, we consider the following locally defined constant functions, where $i$, $j=1, \ldots, n$:
\begin{align*}
&{}_{\tau}\alpha(\bm{e}_i,\bm{e}_j):={}_{\tau}\xi(\bm{e}_i+\bm{e}_j,z)^{-1}{}_{\tau}\xi(\bm{e}_j,z+\bm{e}_i){}_{\tau}\xi(\bm{e}_i,z)=\mathrm{exp}\Bigl( 2\pi \mathbf{i}\bm{e}_i^t \Lambda(T,\tau) \bm{e}_j \Bigr), \\
&{}_{\tau}\alpha(\bm{e}_i,T\bm{e}_j):={}_{\tau}\xi(\bm{e}_i+T\bm{e}_j,z)^{-1}{}_{\tau}\xi(T\bm{e}_j,z+\bm{e}_i){}_{\tau}\xi(\bm{e}_i,z)=\mathrm{exp}\Bigl( 2\pi \mathbf{i}\bm{e}_i^t \Lambda(T,\tau) \bar{T}\bm{e}_j \Bigr), \\
&{}_{\tau}\alpha(T\bm{e}_i,\bm{e}_j):={}_{\tau}\xi(T\bm{e}_i+\bm{e}_j,z)^{-1}{}_{\tau}\xi(\bm{e}_j,z+T\bm{e}_i){}_{\tau}\xi(T\bm{e}_i,z)=\mathrm{exp}\Bigl( 2\pi \mathbf{i}\bm{e}_i^t \bar{T}^t \Lambda(T,\tau) \bm{e}_j \Bigr), \\
&{}_{\tau}\alpha(T\bm{e}_i,T\bm{e}_j):={}_{\tau}\xi(T\bm{e}_i+T\bm{e}_j,z)^{-1}{}_{\tau}\xi(T\bm{e}_j,z+T\bm{e}_i){}_{\tau}\xi(T\bm{e}_i,z)=\mathrm{exp}\Bigl( 2\pi \mathbf{i}\bm{e}_i^t \bar{T}^t \Lambda(T,\tau) \bar{T}\bm{e}_j \Bigr).
\end{align*}
In particular, they are non-trivial if and only if $\Lambda(T,\tau)\not\in \mathrm{Sym}(n;\mathbb{R})$. These ${}_{\tau}\alpha(\bm{e}_i,\bm{e}_j)$, ${}_{\tau}\alpha(\bm{e}_i,T\bm{e}_j)$, ${}_{\tau}\alpha(T\bm{e}_i,\bm{e}_j)$, ${}_{\tau}\alpha(T\bm{e}_i,T\bm{e}_j)$ can be regarded as holomorphic sections of the trivial complex line bundles on $O_{[z]}\cap O_{[z+\bm{e}_i]}\cap O_{[z+\bm{e}_i+\bm{e}_j]}$, $O_{[z]}\cap O_{[z+\bm{e}_i]}\cap O_{[z+\bm{e}_i+T\bm{e}_j]}$, $O_{[z]}\cap O_{[z+T\bm{e}_i]}\cap O_{[z+T\bm{e}_i+\bm{e}_j]}$, $O_{[z]}\cap O_{[z+T\bm{e}_i]}\cap O_{[z+T\bm{e}_i+T\bm{e}_j]}$ equipped with the trivial holomorphic structures, respectively. Moreover, we have the relations 
\begin{align*}
&2\pi \mathbf{i}{}_{\tau}\omega(\bm{e}_i)+2\pi \mathbf{i}{}_{\tau}\omega(\bm{e}_j)-2\pi \mathbf{i}{}_{\tau}\omega(\bm{e}_i+\bm{e}_j)={}_{\tau}\alpha(\bm{e}_i,\bm{e}_j)\bar{\partial}{}_{\tau}\alpha(\bm{e}_i,\bm{e}_j)^{-1}=0, \\
&2\pi \mathbf{i}{}_{\tau}\omega(\bm{e}_i)+2\pi \mathbf{i}{}_{\tau}\omega(T\bm{e}_j)-2\pi \mathbf{i}{}_{\tau}\omega(\bm{e}_i+T\bm{e}_j)={}_{\tau}\alpha(\bm{e}_i,T\bm{e}_j)\bar{\partial}{}_{\tau}\alpha(\bm{e}_i,T\bm{e}_j)^{-1}=0, \\
&2\pi \mathbf{i}{}_{\tau}\omega(T\bm{e}_i)+2\pi \mathbf{i}{}_{\tau}\omega(\bm{e}_j)-2\pi \mathbf{i}{}_{\tau}\omega(T\bm{e}_i+\bm{e}_j)={}_{\tau}\alpha(T\bm{e}_i,\bm{e}_j)\bar{\partial}{}_{\tau}\alpha(T\bm{e}_i,\bm{e}_j)^{-1}=0, \\
&2\pi \mathbf{i}{}_{\tau}\omega(T\bm{e}_i)+2\pi \mathbf{i}{}_{\tau}\omega(T\bm{e}_j)-2\pi \mathbf{i}{}_{\tau}\omega(T\bm{e}_i+T\bm{e}_j)={}_{\tau}\alpha(T\bm{e}_i,T\bm{e}_j)\bar{\partial}{}_{\tau}\alpha(T\bm{e}_i,T\bm{e}_j)^{-1}=0,
\end{align*}
where 
\begin{align*}
&{}_{\tau}\omega(\bm{e}_i+\bm{e}_j):={}_{\tau}\omega(\bm{e}_i)+{}_{\tau}\omega(\bm{e}_j), \ \ \ {}_{\tau}\omega(\bm{e}_i+T\bm{e}_j):={}_{\tau}\omega(\bm{e}_i)+{}_{\tau}\omega(T\bm{e}_j), \\
&{}_{\tau}\omega(T\bm{e}_i+\bm{e}_j):={}_{\tau}\omega(T\bm{e}_i)+{}_{\tau}\omega(\bm{e}_j), \ \ \ {}_{\tau}\omega(T\bm{e}_i+T\bm{e}_j):={}_{\tau}\omega(T\bm{e}_i)+{}_{\tau}\omega(T\bm{e}_j),
\end{align*}
and $i$, $j=1, \ldots, n$. For simplicity, let us put
\begin{align*}
&{}_{\tau}\alpha:=\Bigl\{ {}_{\tau}\alpha(\bm{e}_i,\bm{e}_j), \ {}_{\tau}\alpha(\bm{e}_i,T\bm{e}_j), \ {}_{\tau}\alpha(T\bm{e}_i,\bm{e}_j), \ {}_{\tau}\alpha(T\bm{e}_i,T\bm{e}_j) \Bigr\}_{i,j=1,\ldots,n}, \\
&{}_{\tau}\nabla:=\Bigl\{ \bar{\partial}+2\pi \mathbf{i}{}_{\tau}\omega(\bm{e}_i), \ \bar{\partial}+2\pi \mathbf{i}{}_{\tau}\omega(T\bm{e}_i) \Bigr\}_{i=1,\ldots,n}.
\end{align*}
Summarizing the above arguments, we see that the data $({}_{\tau}\alpha, {}_{\tau}\nabla, 2\pi \mathbf{i}{}_{\tau}B^{(0,2)})$ defines a flat gerbe ${}_{\tau}\mathcal{G}^{\nabla}$ in the sense of Hitchin-Chatterjee \cite{chat, h} on $X^n$:
\begin{equation*}
{}_{\tau}\mathcal{G}^{\nabla}:=\Bigl( {}_{\tau}\alpha, \ {}_{\tau}\nabla, \ 2\pi \mathbf{i}{}_{\tau}B^{(0,2)} \Bigr).
\end{equation*}
Here, the family ${}_{\tau}\nabla$ gives a 0-connection over the gerbe on $X^n$ which is determined by the family ${}_{\tau}\alpha$, and the globally defined 2-form $2\pi \mathbf{i}{}_{\tau}B^{(0,2)}$ is the 1-connection which is compatible with the 0-connection ${}_{\tau}\nabla$.

As remarked in the above, ${}_{\tau}\alpha$ is non-trivial if and only if $\Lambda(T, \tau)\not\in \mathrm{Sym}(n;\mathbb{R})$, and the condition $\Lambda(T,\tau)\in \mathrm{Sym}(n;\mathbb{R})$ is equivalent to the condition $\tau\in \mathrm{Sym}(n;\mathbb{R})$, i.e., $\tau=O$ (note that $\tau\in \mathrm{Sym}(n;\mathbb{R})$). As a consequence, twisting $X^n$ by using ${}_{\tau}\mathcal{G}^{\nabla}$ and considering the B-field transform $\mathcal{I}_T({}_{\tau}B)$ of $\mathcal{I}_T$ associated to ${}_{\tau}B$ are essentially the same (see Proposition \ref{comp_B_field} and Corollary \ref{comp_B_field_cor}). Thus, when we denote this gerby deformed complex torus by ${}_{\tau}X^n$, the following notation makes sense:
\begin{equation*}
{}_{\tau}X^n=\Bigl( X^n, \ \mathcal{I}_T({}_{\tau}B) \Bigr).
\end{equation*}

Associated to the deformation from $X^n$ to ${}_{\tau}X^n$, we can twist each holomorphic line bundle $E_{(A,p,q)}^{\nabla}$ to the following ${}_{\tau}\alpha$-twisted smooth complex line bundle on $X^n$ with the connection. First, the family of the transition functions, i.e., the factor of automorphy $j_A$ is twisted as follows, where $i=1, \ldots, n$, and note that $\mathrm{exp}(2\pi \mathbf{i}\bm{e}_i^t A^t (T-\bar{T})^{-1}(z-\bar{z}))=\mathrm{exp}(2\pi \mathbf{i}\bm{e}_i^t A^t y)$:
\begin{align*}
&j_A(\bm{e}_i, z) \\
&\mapsto {}_{\tau}j_A(\bm{e}_i, z):={}_{\tau}\xi(\bm{e}_i, z)j_A(\bm{e}_i, z)=\mathrm{exp}\left( -\pi \mathbf{i}\bm{e}_i^t \Lambda(T,\tau)\bar{z}+2\pi \mathbf{i}\bm{e}_i^t A^t \left( T-\bar{T} \right)^{-1}(z-\bar{z}) \right), \\
& j_A(T\bm{e}_i, z) \\
&\mapsto {}_{\tau}j_A(T\bm{e}_i, z):={}_{\tau}\xi(T\bm{e}_i, z)j_A(T\bm{e}_i, z)=\mathrm{exp}\left( -\pi \mathbf{i}\bm{e}_i^t \bar{T}^t \Lambda(T,\tau)\bar{z} \right).
\end{align*}  
We further define
\begin{align*}
&{}_{\tau}j_A(\bm{e}_i+\bm{e}_j,z):={}_{\tau}j_A(\bm{e}_i,z+\bm{e}_j){}_{\tau}j_A(\bm{e}_j,z), \\
&{}_{\tau}j_A(\bm{e}_i+T\bm{e}_j,z):={}_{\tau}j_A(\bm{e}_i,z+T\bm{e}_j){}_{\tau}j_A(T\bm{e}_j,z), \\
&{}_{\tau}j_A(T\bm{e}_i+\bm{e}_j,z):={}_{\tau}j_A(T\bm{e}_i,z+\bm{e}_j){}_{\tau}j_A(\bm{e}_j,z), \\
&{}_{\tau}j_A(T\bm{e}_i+T\bm{e}_j,z):={}_{\tau}j_A(T\bm{e}_i,z+T\bm{e}_j){}_{\tau}j_A(T\bm{e}_j,z),
\end{align*}
and then, we can easily check that the map ${}_{\tau}j_A : (\mathbb{Z}^n \oplus T\mathbb{Z}^n)\times \mathbb{C}^n \to \mathbb{C}^{\times}$ satisfies the relations
\begin{align*}
&{}_{\tau}j_A(\bm{e}_i+\bm{e}_j, z)^{-1}{}_{\tau}j_A(\bm{e}_j, z+\bm{e}_i){}_{\tau}j_A(\bm{e}_i, z)={}_{\tau}\alpha(\bm{e}_i, \bm{e}_j)\cdot 1, \\
&{}_{\tau}j_A(\bm{e}_i+T\bm{e}_j, z)^{-1}{}_{\tau}j_A(T\bm{e}_j, z+\bm{e}_i){}_{\tau}j_A(\bm{e}_i, z)={}_{\tau}\alpha(\bm{e}_i, T\bm{e}_j)\cdot 1, \\
&{}_{\tau}j_A(T\bm{e}_i+\bm{e}_j, z)^{-1}{}_{\tau}j_A(\bm{e}_j, z+T\bm{e}_i){}_{\tau}j_A(T\bm{e}_i, z)={}_{\tau}\alpha(T\bm{e}_i, \bm{e}_j)\cdot 1, \\
&{}_{\tau}j_A(T\bm{e}_i+T\bm{e}_j, z)^{-1}{}_{\tau}j_A(T\bm{e}_j, z+T\bm{e}_i){}_{\tau}j_A(T\bm{e}_i, z)={}_{\tau}\alpha(T\bm{e}_i, T\bm{e}_j)\cdot 1, 
\end{align*}
where $i$, $j=1, \ldots, n$. They imply that the map ${}_{\tau}j_A$ defines an ${}_{\tau}\alpha$-twisted smooth complex line bundle on $X^n$, so let us denote it by ${}_{\tau}E_A \to X^n$ or ${}_{\tau}E_A$ for short. Moreover, as will be explained later, we can check that the differential operator of the form (\ref{conn}) again defines a connection on ${}_{\tau}E_A$. However, we denote this differential operator by ${}_{\tau}\nabla_{(A,p,q)}=d+{}_{\tau}\omega_{(A,p,q)}$ because we consider the deformation depending on $\tau$ in this context. It is easy to check that the 1-form ${}_{\tau}\omega_{(A,p,q)}$ satisfies the following relations for each $i=1, \ldots, n$:
\begin{align*}
&{}_{\tau}\omega_{(A,p,q)}(x+\bm{e}_i, y)={}_{\tau}\omega_{(A,p,q)}(x, y)+{}_{\tau}j_A(\bm{e}_i, z)d\hspace{0.3mm}{}_{\tau}j_A(\bm{e}_i, z)^{-1}-2\pi \mathbf{i}{}_{\tau}\omega(\bm{e}_i), \\
&{}_{\tau}\omega_{(A,p,q)}(x, y+\bm{e}_i)={}_{\tau}\omega_{(A,p,q)}(x, y)+{}_{\tau}j_A(T\bm{e}_i, z)d\hspace{0.3mm}{}_{\tau}j_A(T\bm{e}_i, z)^{-1}-2\pi \mathbf{i}{}_{\tau}\omega(T\bm{e}_i).
\end{align*}
Here, we use the notation ${}_{\tau}\omega_{(A,p,q)}(x, y)$ instead of ${}_{\tau}\omega_{(A,p,q)}$ in order to specify the coordinate system $(x,y)\in \mathbb{R}^{2n}$, and for each $i=1, \ldots, n$, note that the shifts $x\mapsto x+\bm{e}_i$, $y\mapsto y+\bm{e}_i$ correspond to the shifts $z\mapsto z+\bm{e}_i$, $z\mapsto z+T\bm{e}_i$, respectively. We denote the ${}_{\tau}\alpha$-twisted smooth complex line bundle ${}_{\tau}E_A\to X^n$ with the connection ${}_{\tau}\nabla_{(A,p,q)}$ by ${}_{\tau}E_{(A,p,q)}^{\nabla}$, i.e., ${}_{\tau}E_{(A,p,q)}^{\nabla}:=({}_{\tau}E_A, {}_{\tau}\nabla_{(A,p,q)})$. 

Let us consider the holomorphicity of ${}_{\tau}E_{(A,p,q)}^{\nabla}$. The curvature form ${}_{\tau}\Omega_{(A,p,q)}$ of the connection ${}_{\tau}\nabla_{(A,p,q)}$ is expressed locally as
\begin{align*}
{}_{\tau}\Omega_{(A,p,q)}&=d{}_{\tau}\omega_{(A,p,q)}+{}_{\tau}\omega_{(A,p,q)}\wedge {}_{\tau}\omega_{(A,p,q)}+2\pi \mathbf{i}{}_{\tau}B^{(0,2)} \\
&=-2\pi \mathbf{i}dx^t A^t dy+\pi \mathbf{i}d\bar{z}^t \Lambda(T,\tau) d\bar{z}. 
\end{align*}
Now, we prepare the following proposition.
\begin{proposition} \label{twisted_hol}
The \textup{(}0,2\textup{)}-part ${}_{\tau}\Omega_{(A,p,q)}^{(0,2)}$ vanishes if and only if $\tau=O$.
\end{proposition}
\begin{proof}
The (0,2)-part ${}_{\tau}\Omega_{(A,p,q)}^{(0,2)}$ turns out to be 
\begin{align*}
{}_{\tau}\Omega_{(A,p,q)}^{(0,2)}&=2\pi \mathbf{i}d\bar{z}^t ((T-\bar{T})^{-1})^t (AT)^t (T-\bar{T})^{-1} d\bar{z}+\pi \mathbf{i}d\bar{z}^t \Lambda(T,\tau) d\bar{z} \\
&=\pi \mathbf{i} d\bar{z}^t \Lambda(T,\tau) d\bar{z}
\end{align*}
since we assume that $AT\in \mathrm{Sym}(n;\mathbb{C})$. By the definition of $\Lambda(T,\tau)$, it is clear that $\Lambda(T,\tau)\in \mathrm{Sym}(n;\mathbb{C})$ if and only if $\tau\in \mathrm{Sym}(n;\mathbb{R})$, i.e., $\tau=O$ (note that $\tau\in \mathrm{Alt}(n;\mathbb{R})$). This completes the proof.
\end{proof}
By \cite[Lemma 2.18, Lemma 2.19]{k-htwisted}, ${}_{\tau}E_{(A,p,q)}^{\nabla}$ becomes an ${}_{\tau}\alpha$-twisted holomorphic line bundle on $X^n$ with a connection if and only if ${}_{\tau}\Omega_{(A,p,q)}^{(0,2)}$ vanishes. Thus, Proposition \ref{twisted_hol} implies that ${}_{\tau}E_{(A,p,q)}^{\nabla}$ is not holomorphic when we consider the non-trivial gerby deformation, i.e., the case $\tau \not=O$.

\subsubsection{Symplectic geometry side}
Let us first compute the mirror partner of ${}_{\tau}X^n$ by using the mirror transform (\ref{mirror}). As a result, we have 
\begin{align*}
&\check{\mathcal{I}}_T({}_{\tau}B) \left( \frac{\partial}{\partial \check{x}}^t, \frac{\partial}{\partial \check{y}}^t, d\check{x}^t, d\check{y}^t \right) \notag \\
&=\left( \frac{\partial}{\partial \check{x}}^t, \frac{\partial}{\partial \check{y}}^t, d\check{x}^t, d\check{y}^t \right) \left( \begin{array}{cccc} I_n & O & O & O \\ O & I_n & O & O \\ -\tau & O & I_n & O \\ O & O & O & I_n \end{array} \right) \check{\mathcal{I}}_T \left( \begin{array}{cccc} I_n & O & O & O \\ O & I_n & O & O \\ \tau & O & I_n & O \\ O & O & O & I_n \end{array} \right), 
\end{align*}
and this implies that the symplectic form $\omega^{\vee}$ is preserved and the B-field $B^{\vee}$ is twisted by
\begin{equation*}
{}_{\tau}B^{\vee}:=\frac{1}{2}d\check{x}^t \tau d\check{x}.
\end{equation*}
Let us put
\begin{equation*}
{}_{\tau}\tilde{\omega}^{\vee}:=\tilde{\omega}^{\vee}+{}_{\tau}B^{\vee}=\mathbf{i}\omega^{\vee}+\bigl( B^{\vee}+{}_{\tau}B^{\vee} \bigr).
\end{equation*}
Hereafter, we denote this complexified symplectic torus (the mirror partner of ${}_{\tau}X^n$ which is obtained by using the mirror transform (\ref{mirror})) by
\begin{equation*}
{}_{\tau}\check{X}^n:=\left( \mathbb{R}^{2n}/\mathbb{Z}^{2n}, \ {}_{\tau}\tilde{\omega}^{\vee}=d\check{x}^t (-(T^{-1})^t) d\check{y}+\frac{1}{2}d\check{x}^t \tau d\check{x} \right).
\end{equation*}

We construct the deformation of objects $\mathcal{L}_{(A,p,q)}^{\nabla}$ of $Fuk_{\rm sub}(\check{X}^n)$ which are mirror dual to ${}_{\tau}\alpha$-twisted smooth complex line bundles ${}_{\tau}E_{(A,p,q)}^{\nabla}$. It is reasonable to consider that $L_{(A,p)}$ is preserved and $\mathcal{O}_{(A,p,q)}^{\nabla}$ is deformed to something since the symplectic form $\omega^{\vee}$ is preserved and the B-field $B^{\vee}$ is twisted by ${}_{\tau}B^{\vee}$. Anyway, we should be careful with the condition (\ref{integrity}) when we consider the deformation of $\mathcal{O}_{(A,p,q)}^{\nabla}$. For example, in the case $\tau \in M(n;\mathbb{R})\backslash M(n;\mathbb{Q})$, it is easy to check that 
\begin{equation*}
\bigl[ B^{\vee}+{}_{\tau}B^{\vee} \bigr]\not\in H^2(L_{(A,p)},\mathbb{Z})
\end{equation*}  
holds for any Lagrangian submanifold $L_{(A,p)}$ in ${}_{\tau}\check{X}^n$. Therefore, it is natural to employ a twisted line bundle as a deformation of each $\mathcal{O}_{(A,p,q)}^{\nabla}$ as discussed in subsection 4.1. In fact, the effect of $B^{\vee}$ vanishes on $L_{(A,p)}$ since we assume that $AT\in \mathrm{Sym}(n;\mathbb{C})$, so essentially, it is enough to consider the effect of ${}_{\tau}B^{\vee}$ only on $L_{(A,p)}$. Moreover, note that the flat gerbe given here can also be regarded as the restriction of the flat gerbe defined globally on $\check{X}_{\omega^{\vee}}^n$ whose 1-connection is $2\pi \mathbf{i}(B^{\vee}+{}_{\tau}B^{\vee})$ to $L_{(A,p)}$. We set
\begin{equation*}
{}_{\tau}B_A^{\vee}:=\Bigl. {}_{\tau}B^{\vee} \Bigr|_{L_{(A,p)}},
\end{equation*}
i.e.,
\begin{equation*}
{}_{\tau}B_A^{\vee}=\frac{1}{2}d\check{x}^t \tau d\check{x}.
\end{equation*}
For ${}_{\tau}B_A^{\vee}$, let us take a 1-form ${}_{\tau}\beta_A^{\vee}$ which is expressed locally as
\begin{equation*}
{}_{\tau}\beta_A^{\vee}=\frac{1}{2}\check{x}^t \tau d\check{x},
\end{equation*}
and it is clear that ${}_{\tau}B_A^{\vee}=d{}_{\tau}\beta_A^{\vee}$. We sometimes use the notation ${}_{\tau}\beta_A^{\vee}(\check{x})$ instead of ${}_{\tau}\beta_A^{\vee}$ in order to specify the coordinate system $\check{x}$. Then, similarly as in the cases of $\check{\mathcal{G}}_{\theta,A}^{\nabla}$, ${}_{\tau}\mathcal{G}^{\nabla}$, the locally expressed data ${}_{\tau}\alpha_A^{\vee}:=\{ {}_{\tau}\alpha_A^{\vee}(\bm{e}_i,\bm{e}_j) \}_{i,j=1,\ldots,n}$, ${}_{\tau}\nabla_A^{\vee}:=\{ d+2\pi \mathbf{i}{}_{\tau}\omega_A^{\vee}(\bm{e}_i) \}_{i=1,\ldots,n}$: 
\begin{align*}
&{}_{\tau}\alpha_A^{\vee}(\bm{e}_i,\bm{e}_j):={}_{\tau}\xi_A^{\vee}(\bm{e}_i+\bm{e}_j,\check{x})^{-1}{}_{\tau}\xi_A^{\vee}(\bm{e}_j,\check{x}+\bm{e}_i){}_{\tau}\xi_A^{\vee}(\bm{e}_i,\check{x})=\mathrm{exp}\Bigl( 2\pi \mathbf{i}\bm{e}_i^t \tau \bm{e}_j \Bigr), \\
&{}_{\tau}\omega_A^{\vee}(\bm{e}_i):={}_{\tau}\beta_A^{\vee}(\check{x}+\bm{e}_i)-{}_{\tau}\beta_A^{\vee}(\check{x})=\frac{1}{2}\bm{e}_i^t \tau d\check{x}
\end{align*}
and $2\pi \mathbf{i}{}_{\tau}B_A^{\vee}$ defines a flat gerbe ${}_{\tau}\check{\mathcal{G}}_A^{\nabla}$ in the sense of Hitchin-Chatterjee \cite{chat, h} on $L_{(A,p)}$:
\begin{equation*}
{}_{\tau}\check{\mathcal{G}}_A^{\nabla}:=\Bigl( {}_{\tau}\alpha_A^{\vee}, \ {}_{\tau}\nabla_A^{\vee}, \ 2\pi \mathbf{i}{}_{\tau}B_A^{\vee} \Bigr),
\end{equation*}
where 
\begin{equation*}
{}_{\tau}\xi_A^{\vee}(\bm{e}_i,\check{x}):=\mathrm{exp}\Bigl( -\pi \mathbf{i} \bm{e}_i^t \tau \check{x} \Bigr), \ \ \ {}_{\tau}\xi_A^{\vee}(\bm{e}_i+\bm{e}_j,\check{x}):={}_{\tau}\xi_A^{\vee}(\bm{e}_i,\check{x}+\bm{e}_j){}_{\tau}\xi_A^{\vee}(\bm{e}_j,\check{x}) 
\end{equation*}
for any $i$, $j=1, \ldots, n$. In particular, the family ${}_{\tau}\nabla_A^{\vee}$ gives a 0-connection over the gerbe on $L_{(A,p)}$ which is determined by the family ${}_{\tau}\alpha_A^{\vee}$, and the globally defined 2-form $2\pi \mathbf{i}{}_{\tau}B_A^{\vee}$ is the 1-connection which is compatible with the 0-connection ${}_{\tau}\nabla_A^{\vee}$.

Let us twist each $\mathcal{O}_{(A,p,q)}^{\nabla} \to L_{(A,p)}$ by using the flat gerbe ${}_{\tau}\check{\mathcal{G}}_A^{\nabla}$. We consider a map ${}_{\tau}j_A^{\vee} : \mathbb{Z}^n \times \mathbb{R}^n \to \mathbb{C}^{\times}$ which is defined by
\begin{equation*}
{}_{\tau}j_A^{\vee}(\bm{e}_i, \check{x}):={}_{\tau}\xi_A^{\vee}(\bm{e}_i, \check{x})=\mathrm{exp}\Bigl( -\pi \mathbf{i} \bm{e}_i^t \tau \check{x} \Bigr),
\end{equation*}
where $i=1, \ldots, n$. We further define
\begin{equation*}
{}_{\tau}j_A^{\vee}(\bm{e}_i+\bm{e}_j, \check{x}):={}_{\tau}j_A^{\vee}(\bm{e}_i, \check{x}+\bm{e}_j) {}_{\tau}j_A^{\vee}(\bm{e}_j, \check{x})
\end{equation*}
for any $i$, $j=1, \ldots, n$, and then, it is clear that the map ${}_{\tau}j_A^{\vee}$ satisfies the relations
\begin{equation*}
{}_{\tau}j_A^{\vee}(\bm{e}_i+\bm{e}_j, \check{x})^{-1} {}_{\tau}j_A^{\vee}(\bm{e}_j, \check{x}+\bm{e}_i) {}_{\tau}j_A^{\vee}(\bm{e}_i, \check{x})={}_{\tau}\alpha_A^{\vee}(\bm{e}_i, \bm{e}_j)\cdot 1
\end{equation*}
by the definition of ${}_{\tau}\alpha_A^{\vee}$. Hence, the map ${}_{\tau}j_A^{\vee}$ defines an ${}_{\tau}\alpha_A^{\vee}$-twisted smooth complex line bundle on $L_{(A,p)}$, so let us denote it by ${}_{\tau}\mathcal{O}_{(A,p)}\to L_{(A,p)}$ or ${}_{\tau}\mathcal{O}_{(A,p)}$ for short. On the other hand, it is natural to consider that the connection $\nabla_{(A,p,q)}^{\vee}$ on $\mathcal{O}_{(A,p)}$ is preserved:
\begin{equation*}
{}_{\tau}\omega_{(A,p,q)}^{\vee}:=\omega_{(A,p,q)}^{\vee}, \ \ \ {}_{\tau}\nabla_{(A,p,q)}^{\vee}:=d+{}_{\tau}\omega_{(A,p,q)}^{\vee}.
\end{equation*}
In fact, we have the relation
\begin{equation*}
{}_{\tau}\omega_{(A,p,q)}^{\vee}(\check{x}+\bm{e}_i)={}_{\tau}\omega_{(A,p,q)}^{\vee}(\check{x})+{}_{\tau}j_A^{\vee}(\bm{e}_i, \check{x})d{}_{\tau}j_A^{\vee}(\bm{e}_i, \check{x})^{-1}-2\pi \mathbf{i}{}_{\tau}\omega_A^{\vee}(\bm{e}_i)
\end{equation*}
for each $i=1, \ldots, n$, and we use the notation ${}_{\tau}\omega_A^{\vee}(\check{x})$ instead of ${}_{\tau}\omega_A^{\vee}$ in order to specify the coordinate system $\check{x}$. Here, we set ${}_{\tau}\mathcal{O}_{(A,p,q)}^{\nabla}:=({}_{\tau}\mathcal{O}_{\theta,(A,p)}, {}_{\tau}\nabla_{(A,p,q)}^{\vee})$. Let us denote the pair of the Lagrangian submanifold $L_{(A,p)}$ in ${}_{\tau}\check{X}^n$ and the ${}_{\tau}\alpha_A^{\vee}$-twisted smooth complex line bundle with the connection ${}_{\tau}\mathcal{O}_{(A,p,q)}^{\nabla}$ by ${}_{\tau}\mathcal{L}_{(A,p,q)}^{\nabla}$: ${}_{\tau}\mathcal{L}_{(A,p,q)}^{\nabla}:=(L_{(A,p)}, {}_{\tau}\mathcal{O}_{(A,p,q)}^{\nabla})$.

The curvature form ${}_{\tau}\Omega_{(A,p,q)}^{\vee}$ of the connection ${}_{\tau}\nabla_{(A,p,q)}^{\vee}$ is expressed locally as
\begin{equation*}
{}_{\tau}\Omega_{(A,p,q)}^{\vee}=d{}_{\tau}\omega_{(A,p,q)}^{\vee}+{}_{\tau}\omega_{(A,p,q)}^{\vee}\wedge {}_{\tau}\omega_{(A,p,q)}^{\vee}+2\pi \mathbf{i}{}_{\tau}B_A^{\vee}=\pi \mathbf{i}d\check{x}^t \tau d\check{x}.
\end{equation*}
Then we can easily verify that the equality
\begin{equation*}
{}_{\tau}\Omega_{(A,p,q)}^{\vee}=\Bigl. 2\pi \mathbf{i} \bigl( B^{\vee}+{}_{\tau}B^{\vee} \bigr) \Bigr|_{L_{(A,p)}}
\end{equation*}
holds. In particular, we see that ${}_{\tau}\Omega_{(A,p,q)}^{(0,2)}$ and ${}_{\tau}\Omega_{(A,p,q)}^{\vee}$ are determined by the same matrix $\tau \in \mathrm{Alt}(n;\mathbb{R})$:
\begin{equation*}
{}_{\tau}\Omega_{(A,p,q)}^{(0,2)}=\pi \mathbf{i} \left( dx^{(0,1)} \right)^t \tau dx^{(0,1)}, \ \ \ {}_{\tau}\Omega_{(A,p,q)}^{\vee}=\pi \mathbf{i} d\check{x}^t \tau d\check{x}.
\end{equation*}

We explain the relation between ${}_{\tau}\check{\mathcal{G}}_A^{\nabla}$ and $\check{\mathcal{G}}_{\theta,A}^{\nabla}$ briefly. For simplicity, we assume that $\mathrm{det}A\not=0$. Let us denote the Fourier-Mukai partner of $X^n$ by $\hat{X}^n$, i.e., $\hat{X}^n=\mathbb{C}^n/(\mathbb{Z}^n \oplus T^t \mathbb{Z}^n)$. The symplectic form of its mirror partner $\check{\hat{X}}^n$ is given by $\hat{\omega}^{\vee}:=d\check{x}^t (\underline{\omega}^{\vee})^t d\check{y}$, and the pullback of $\hat{\omega}^{\vee}$ by the map (\ref{FM_symplectic}) coincides with $\omega^{\vee}$. Moreover, each Lagrangian submanifold $L_{(A,p)}$ in $\check{X}_{\omega^{\vee}}^n$ maps to the Lagrangian submanifold $L_{(-A^{-1},A^{-1}p)}$ in the symplectic torus $\check{\hat{X}}_{\hat{\omega}^{\vee}}^n:=(\mathbb{R}^{2n}/\mathbb{Z}^{2n}, \hat{\omega}^{\vee})$ by the map (\ref{FM_symplectic}). Let us focus on the gerby deformation of $\check{\hat{X}}_{\hat{\omega}^{\vee}}^n$ by the B-field ${}_{\theta}B^{\vee}$. We denote the flat gerbe on $\check{\hat{X}}_{\hat{\omega}^{\vee}}^n$ which is determined by ${}_{\theta}B^{\vee}$ by ${}_{\theta}\check{\mathcal{G}}^{\nabla}$, namely, ${}_{\theta}\check{\mathcal{G}}^{\nabla}|_{L_{(-A^{-1},A^{-1}p)}}={}_{\theta}\check{\mathcal{G}}_{-A^{-1}}^{\nabla}$. Then, by direct calculations, we see that the restriction of the pullback\footnote{When we regard ${}_{\theta}\check{\mathcal{G}}^{\nabla}$ as the data $({}_{\theta}\alpha^{\vee}, {}_{\theta}\nabla^{\vee}, 2\pi \mathbf{i}{}_{\theta}B^{\vee})$ similarly as in other examples (${}_{\theta}\alpha^{\vee}\in H^2(\check{\hat{X}}_{\hat{\omega}^{\vee}}^n, C^{\infty}(\check{\hat{X}}_{\hat{\omega}^{\vee}}^n)^{\times})$, ${}_{\theta}\nabla^{\vee}$ is the 0-connection over this gerbe, $C^{\infty}(\check{\hat{X}}_{\hat{\omega}^{\vee}}^n)^{\times}$ is the sheaf of nowhere vanishing smooth functions on $\check{\hat{X}}_{\hat{\omega}^{\vee}}^n$), we call the data (the flat gerbe on $\check{X}_{\omega^{\vee}}^n$) which are given by the pullback of ${}_{\theta}\alpha^{\vee}$, ${}_{\theta}\nabla^{\vee}$, $2\pi \mathbf{i}{}_{\theta}B^{\vee}$ by the map (\ref{FM_symplectic}) the pullback of ${}_{\theta}\check{\mathcal{G}}^{\nabla}$ by the map (\ref{FM_symplectic}).} of ${}_{\theta}\check{\mathcal{G}}^{\nabla}$ by the map (\ref{FM_symplectic}) to $L_{(A,p)}$ coincides with $\check{\mathcal{G}}_{\theta,A}^{\nabla}$. We can expect that this description corresponds to the Fourier-Mukai transform from the noncommutative complex torus $X_{\theta}^n$ to the gerby deformed complex torus ${}_{\theta}\hat{X}^n$.  

\subsection{The gerby deformation of type $\tau_2$}
We take an arbitrary $\tau \in M(n;\mathbb{R})$, and fix it. In this subsection, we revise several results written in \cite{gerby}.

\subsubsection{Complex geometry side}
Let us consider the B-field (\ref{b_field_X_n}) in the case that $\tau_2=\tau$ and $\tau_1=\tau_3=O$, and denote it by ${}_{\tau}B$:
\begin{equation*}
{}_{\tau}B:=dx^t \tau dy.
\end{equation*}
Here, precisely speaking, in \cite{gerby}, note that $\tau$ in this ${}_{\tau}B$ is replaced with $\tau^t$. From the viewpoint of generalized complex geometry, this B-field ${}_{\tau}B$ causes the B-field transform of $\mathcal{I}_T$ over $X^n$:
\begin{align*}
&\mathcal{I}_T({}_{\tau}B) \left( \frac{\partial}{\partial x}^t, \frac{\partial}{\partial y}^t, dx^t, dy^t \right) \\
&=\left( \frac{\partial}{\partial x}^t, \frac{\partial}{\partial y}^t, dx^t, dy^t \right) \left( \begin{array}{cccc} I_n & O & O & O \\ O & I_n & O & O \\ O & -\tau & I_n & O \\ \tau^t & O & O & I_n \end{array} \right) \mathcal{I}_T \left( \begin{array}{cccc} I_n & O & O & O \\ O & I_n & O & O \\ O & \tau & I_n & O \\ -\tau^t & O & O & I_n \end{array} \right).
\end{align*}
Then, the following statements hold, and we omit the proofs of them because they can be proved similarly as in Proposition \ref{comp_B_field} and Corollary \ref{comp_B_field_cor}.
\begin{proposition} \label{comp_B_field_1} 
The B-field transform associated to the B-field ${}_{\tau}B$ preserves the generalized complex structure $\mathcal{I}_T$ over $X^n$, i.e., 
\begin{equation*}
\mathcal{I}_T({}_{\tau}B)=\mathcal{I}_T
\end{equation*}
if and only if $\tau^t T\in \mathrm{Sym}(n;\mathbb{C})$.
\end{proposition}
\begin{corollary} \label{comp_B_field_cor_1} 
The B-field transform associated to the B-field ${}_{\tau}B$ preserves the generalized complex structure $\mathcal{I}_T$ over $X^n$, i.e., 
\begin{equation*}
\mathcal{I}_T({}_{\tau}B)=\mathcal{I}_T
\end{equation*}
if and only if the \textup{(}0,2\textup{)}-part ${}_{\tau}B^{(0,2)}$ vanishes.
\end{corollary}
Below, we give the strict definition of the flat gerbe whose 1-connection is determined by ${}_{\tau}B^{(0,2)}$. We set
\begin{equation*}
\Lambda(T,\tau):=((T-\bar{T})^{-1})^t \tau^t T (T-\bar{T})^{-1},
\end{equation*}
namely,
\begin{equation*}
{}_{\tau}B^{(0,2)}=d\bar{z}^t \Lambda(T,\tau) d\bar{z}.
\end{equation*}
For ${}_{\tau}B^{(0,2)}$, let us take a 1-form ${}_{\tau}\beta$ which is expressed locally as
\begin{equation*}
{}_{\tau}\beta=\bar{z}^t \Lambda(T,\tau) d\bar{z},
\end{equation*}
and it is clear that ${}_{\tau}B^{(0,2)}=\bar{\partial}{}_{\tau}\beta$. We sometimes use the notation ${}_{\tau}\beta(z)$ instead of ${}_{\tau}\beta$ in order to specify the complex coordinate system $z$. Then the locally expressed data ${}_{\tau}\alpha:=\{ {}_{\tau}\alpha(\bm{e}_i,\bm{e}_j), \ {}_{\tau}\alpha(\bm{e}_i,T\bm{e}_j), \ {}_{\tau}\alpha(T\bm{e}_i,\bm{e}_j), \ {}_{\tau}\alpha(T\bm{e}_i,T\bm{e}_j) \}_{i,j=1,\ldots,n}$, ${}_{\tau}\nabla:=\{ \bar{\partial}+2\pi \mathbf{i}{}_{\tau}\omega(\bm{e}_i), \ \bar{\partial}+2\pi \mathbf{i}{}_{\tau}\omega(T\bm{e}_i) \}_{i=1,\ldots,n}$:
\begin{align*}
&{}_{\tau}\alpha(\bm{e}_i,\bm{e}_j):={}_{\tau}\xi(\bm{e}_i+\bm{e}_j,z)^{-1}{}_{\tau}\xi(\bm{e}_j,z+\bm{e}_i){}_{\tau}\xi(\bm{e}_i,z)=\mathrm{exp}\Bigl( 2\pi \mathbf{i}\bm{e}_i^t (\Lambda(T,\tau)-\Lambda(T,\tau)^t) \bm{e}_j \Bigr), \\
&{}_{\tau}\alpha(\bm{e}_i,T\bm{e}_j):={}_{\tau}\xi(\bm{e}_i+T\bm{e}_j,z)^{-1}{}_{\tau}\xi(T\bm{e}_j,z+\bm{e}_i){}_{\tau}\xi(\bm{e}_i,z)=\mathrm{exp}\Bigl( 2\pi \mathbf{i}\bm{e}_i^t (\Lambda(T,\tau)-\Lambda(T,\tau)^t) \bar{T}\bm{e}_j \Bigr), \\
&{}_{\tau}\alpha(T\bm{e}_i,\bm{e}_j):={}_{\tau}\xi(T\bm{e}_i+\bm{e}_j,z)^{-1}{}_{\tau}\xi(\bm{e}_j,z+T\bm{e}_i){}_{\tau}\xi(T\bm{e}_i,z)=\mathrm{exp}\Bigl( 2\pi \mathbf{i}\bm{e}_i^t \bar{T}^t (\Lambda(T,\tau)-\Lambda(T,\tau)^t) \bm{e}_j \Bigr), \\
&{}_{\tau}\alpha(T\bm{e}_i,T\bm{e}_j):={}_{\tau}\xi(T\bm{e}_i+T\bm{e}_j,z)^{-1}{}_{\tau}\xi(T\bm{e}_j,z+T\bm{e}_i){}_{\tau}\xi(T\bm{e}_i,z)=\mathrm{exp}\Bigl( 2\pi \mathbf{i}\bm{e}_i^t \bar{T}^t (\Lambda(T,\tau)-\Lambda(T,\tau)^t) \bar{T}\bm{e}_j \Bigr), \\
&{}_{\tau}\omega(\bm{e}_i):={}_{\tau}\beta(z+\bm{e}_i)-{}_{\tau}\beta(z)=\bm{e}_i^t \Lambda(T,\tau) d\bar{z}, \\
&{}_{\tau}\omega(T\bm{e}_i):={}_{\tau}\beta(z+T\bm{e}_i)-{}_{\tau}\beta(z)=\bm{e}_i^t \bar{T}^t \Lambda(T,\tau) d\bar{z}
\end{align*}
and $2\pi \mathbf{i}{}_{\tau}B^{(0,2)}$ defines a flat gerbe ${}_{\tau}\mathcal{G}^{\nabla}$ in the sense of Hitchin-Chatterjee \cite{chat, h} on $X^n$:
\begin{equation*}
{}_{\tau}\mathcal{G}^{\nabla}:=\Bigl( {}_{\tau}\alpha, \ {}_{\tau}\nabla, \ 2\pi \mathbf{i}{}_{\tau}B^{(0,2)} \Bigr),
\end{equation*}
where 
\begin{align*}
&{}_{\tau}\xi(\bm{e}_i,z):=\mathrm{exp}\Bigl( -2\pi \mathbf{i} \bm{e}_i^t \Lambda(T,\tau) \bar{z} \Bigr), \\
&{}_{\tau}\xi(T\bm{e}_i,z):=\mathrm{exp}\Bigl( -2\pi \mathbf{i} \bm{e}_i^t \bar{T}^t \Lambda(T,\tau) \bar{z} \Bigr), \\
&{}_{\tau}\xi(\bm{e}_i+\bm{e}_j,z)={}_{\tau}\xi(\bm{e}_i,z+\bm{e}_j){}_{\tau}\xi(\bm{e}_j,z), \\
&{}_{\tau}\xi(\bm{e}_i+T\bm{e}_j,z)={}_{\tau}\xi(\bm{e}_i,z+T\bm{e}_j){}_{\tau}\xi(T\bm{e}_j,z), \\
&{}_{\tau}\xi(T\bm{e}_i+\bm{e}_j,z)={}_{\tau}\xi(T\bm{e}_i,z+\bm{e}_j){}_{\tau}\xi(\bm{e}_j,z), \\
&{}_{\tau}\xi(T\bm{e}_i+T\bm{e}_j,z)={}_{\tau}\xi(T\bm{e}_i,z+T\bm{e}_j){}_{\tau}\xi(T\bm{e}_j,z), 
\end{align*}
for any $i$, $j=1, \ldots, n$. In particular, the family ${}_{\tau}\nabla$ gives a 0-connection over the gerbe on $X^n$ which is determined by the family ${}_{\tau}\alpha$, and the globally defined 2-form $2\pi \mathbf{i}{}_{\tau}B^{(0,2)}$ is the 1-connection which is compatible with the 0-connection ${}_{\tau}\nabla$. Moreover, when we denote this gerby deformed complex torus by ${}_{\tau}X^n$, the notation 
\begin{equation*}
{}_{\tau}X^n=\Bigl( X^n, \ \mathcal{I}_T({}_{\tau}B) \Bigr).
\end{equation*}
makes sense since we have Proposition \ref{comp_B_field_1} and Corolary \ref{comp_B_field_cor_1}.

Let us twist each $E_{(A,p,q)}^{\nabla} \to X^n$ by using the flat gerbe ${}_{\tau}\mathcal{G}^{\nabla}$. We consider a map ${}_{\tau}j_A : (\mathbb{Z}^n \oplus T\mathbb{Z}^n)\times \mathbb{C}^n \to \mathbb{C}^{\times}$ which is defined by
\begin{align*}
&{}_{\tau}j_A(\bm{e}_i, z):={}_{\tau}\xi(\bm{e}_i, z)j_A(\bm{e}_i, z)=\mathrm{exp}\left( -2\pi \mathbf{i}\bm{e}_i^t \Lambda(T,\tau)\bar{z}+2\pi \mathbf{i}\bm{e}_i^t A^t \left( T-\bar{T} \right)^{-1}(z-\bar{z}) \right), \\
&{}_{\tau}j_A(T\bm{e}_i, z):={}_{\tau}\xi(T\bm{e}_i, z)j_A(T\bm{e}_i, z)=\mathrm{exp}\left( -2\pi \mathbf{i}\bm{e}_i^t \bar{T}^t \Lambda(T,\tau)\bar{z} \right),
\end{align*}  
where $i=1, \ldots, n$, and note that $\mathrm{exp}(2\pi \mathbf{i}\bm{e}_i^t A^t (T-\bar{T})^{-1}(z-\bar{z}))=\mathrm{exp}(2\pi \mathbf{i}\bm{e}_i^t A^t y)$. We further define
\begin{align*}
&{}_{\tau}j_A(\bm{e}_i+\bm{e}_j,z):={}_{\tau}j_A(\bm{e}_i,z+\bm{e}_j){}_{\tau}j_A(\bm{e}_j,z), \\
&{}_{\tau}j_A(\bm{e}_i+T\bm{e}_j,z):={}_{\tau}j_A(\bm{e}_i,z+T\bm{e}_j){}_{\tau}j_A(T\bm{e}_j,z), \\
&{}_{\tau}j_A(T\bm{e}_i+\bm{e}_j,z):={}_{\tau}j_A(T\bm{e}_i,z+\bm{e}_j){}_{\tau}j_A(\bm{e}_j,z), \\
&{}_{\tau}j_A(T\bm{e}_i+T\bm{e}_j,z):={}_{\tau}j_A(T\bm{e}_i,z+T\bm{e}_j){}_{\tau}j_A(T\bm{e}_j,z),
\end{align*}
and then, we can easily check that the map ${}_{\tau}j_A$ satisfies the relations
\begin{align*}
&{}_{\tau}j_A(\bm{e}_i+\bm{e}_j, z)^{-1}{}_{\tau}j_A(\bm{e}_j, z+\bm{e}_i){}_{\tau}j_A(\bm{e}_i, z)={}_{\tau}\alpha(\bm{e}_i, \bm{e}_j)\cdot 1, \\
&{}_{\tau}j_A(\bm{e}_i+T\bm{e}_j, z)^{-1}{}_{\tau}j_A(T\bm{e}_j, z+\bm{e}_i){}_{\tau}j_A(\bm{e}_i, z)={}_{\tau}\alpha(\bm{e}_i, T\bm{e}_j)\cdot 1, \\
&{}_{\tau}j_A(T\bm{e}_i+\bm{e}_j, z)^{-1}{}_{\tau}j_A(\bm{e}_j, z+T\bm{e}_i){}_{\tau}j_A(T\bm{e}_i, z)={}_{\tau}\alpha(T\bm{e}_i, \bm{e}_j)\cdot 1, \\
&{}_{\tau}j_A(T\bm{e}_i+T\bm{e}_j, z)^{-1}{}_{\tau}j_A(T\bm{e}_j, z+T\bm{e}_i){}_{\tau}j_A(T\bm{e}_i, z)={}_{\tau}\alpha(T\bm{e}_i, T\bm{e}_j)\cdot 1
\end{align*}
by the definition of ${}_{\tau}\alpha$, where $i$, $j=1, \ldots, n$. Hence, the map ${}_{\tau}j_A$ defines an ${}_{\tau}\alpha$-twisted smooth complex line bundle on $X^n$, so let us denote it by ${}_{\tau}E_A \to X^n$ or ${}_{\tau}E_A$ for short. On the other hand, it is natural to consider the connection $\nabla_{(A,p,q)}$ on $E_A$ is preserved:
\begin{equation*}
{}_{\tau}\omega_{(A,p,q)}:=\omega_{(A,p,q)}, \ \ \ {}_{\tau}\nabla_{(A,p,q)}:=d+{}_{\tau}\omega_{(A,p,q)}.
\end{equation*} 
In fact, we have the relations
\begin{align*}
&{}_{\tau}\omega_{(A,p,q)}(x+\bm{e}_i, y)={}_{\tau}\omega_{(A,p,q)}(x, y)+{}_{\tau}j_A(\bm{e}_i, z)d\hspace{0.3mm}{}_{\tau}j_A(\bm{e}_i, z)^{-1}-2\pi \mathbf{i}{}_{\tau}\omega(\bm{e}_i), \\
&{}_{\tau}\omega_{(A,p,q)}(x, y+\bm{e}_i)={}_{\tau}\omega_{(A,p,q)}(x, y)+{}_{\tau}j_A(T\bm{e}_i, z)d\hspace{0.3mm}{}_{\tau}j_A(T\bm{e}_i, z)^{-1}-2\pi \mathbf{i}{}_{\tau}\omega(T\bm{e}_i).
\end{align*}
for each $i=1, \ldots, n$, and we use the notation ${}_{\tau}\omega_{(A,p,q)}(x, y)$ instead of ${}_{\tau}\omega_{(A,p,q)}$ in order to specify the coordinate system $(x,y)\in \mathbb{R}^{2n}$. Here, for each $i=1, \ldots, n$, note that the shifts $x\mapsto x+\bm{e}_i$, $y\mapsto y+\bm{e}_i$ correspond to the shifts $z\mapsto z+\bm{e}_i$, $z\mapsto z+T\bm{e}_i$, respectively. We denote the ${}_{\tau}\alpha$-twisted smooth complex line bundle ${}_{\tau}E_A\to X^n$ with the connection ${}_{\tau}\nabla_{(A,p,q)}$ by ${}_{\tau}E_{(A,p,q)}^{\nabla}$, i.e., ${}_{\tau}E_{(A,p,q)}^{\nabla}:=({}_{\tau}E_A, {}_{\tau}\nabla_{(A,p,q)})$. 

Let us consider the holomorphicity of ${}_{\tau}E_{(A,p,q)}^{\nabla}$. The curvature form ${}_{\tau}\Omega_{(A,p,q)}$ of the connection ${}_{\tau}\nabla_{(A,p,q)}$ is expressed locally as
\begin{align*}
{}_{\tau}\Omega_{(A,p,q)}&=d{}_{\tau}\omega_{(A,p,q)}+{}_{\tau}\omega_{(A,p,q)}\wedge {}_{\tau}\omega_{(A,p,q)}+2\pi \mathbf{i}{}_{\tau}B^{(0,2)} \\
&=-2\pi \mathbf{i}dx^t A^t dy+2\pi \mathbf{i}d\bar{z}^t \Lambda(T,\tau) d\bar{z}. 
\end{align*}
Then the following proposition holds, and we omit its proof because it can be proved similarly as in Proposition \ref{twisted_hol}.
\begin{proposition} \label{twisted_hol_1} 
The \textup{(}0,2\textup{)}-part ${}_{\tau}\Omega_{(A,p,q)}^{(0,2)}$ vanishes if and only if $\tau^t T\in \mathrm{Sym}(n;\mathbb{C})$.
\end{proposition}
Similarly as in the case of the gerby deformation of type $\tau_1$, by \cite[Lemma 2.18, Lemma 2.19]{k-htwisted}, ${}_{\tau}E_{(A,p,q)}^{\nabla}$ becomes an ${}_{\tau}\alpha$-twisted holomorphic line bundle on $X^n$ with a connection if and only if ${}_{\tau}\Omega_{(A,p,q)}^{(0,2)}$ vanishes. Thus, Proposition \ref{twisted_hol_1} implies that ${}_{\tau}E_{(A,p,q)}^{\nabla}$ is not holomorphic when we consider the deformation by ${}_{\tau}\mathcal{G}^{\nabla}$ under the assumption $\tau^t T\not\in \mathrm{Sym}(n;\mathbb{C})$.

\subsubsection{Symplectic geometry side}
Let us first compute the mirror partner of ${}_{\tau}X^n$ by using the mirror transform (\ref{mirror}). As a result, the representation matrix of $\check{\mathcal{I}}_T({}_{\tau}B)$ with respect to the standard bases turns out to be
\begin{align*}
&\left( \begin{array}{cccc} I_n & O & O & O \\ O & I_n & O & O \\ \underline{B}^{\vee}\tau^t-\tau (\underline{B}^{\vee})^t & -\underline{B}^{\vee} & I_n & O \\ (\underline{B}^{\vee})^t & O & O & I_n \end{array} \right) \\ 
&\left( \begin{array}{cccc} O & O & O & -((\underline{\omega}^{\vee})^{-1})^t \\ O & O & (\underline{\omega}^{\vee})^{-1} & (\underline{\omega}^{\vee})^{-1}\tau-\tau^t ((\underline{\omega}^{\vee})^{-1})^t \\ \underline{\omega}^{\vee}\tau^t-\tau (\underline{\omega}^{\vee})^t & -\underline{\omega}^{\vee} & O & O \\ (\underline{\omega}^{\vee})^t & O & O & O \end{array} \right) \\ 
&\left( \begin{array}{cccc} I_n & O & O & O \\ O & I_n & O & O \\ -\underline{B}^{\vee}\tau^t+\tau (\underline{B}^{\vee})^t & \underline{B}^{\vee} & I_n & O \\ -(\underline{B}^{\vee})^t & O & O & I_n \end{array} \right).
\end{align*}
This implies that the symplectic form $\omega^{\vee}$ and the B-field $B^{\vee}$ are twisted by
\begin{equation*}
{}_{\tau}\omega^{\vee}:=-d\check{x}^t \underline{\omega}^{\vee} \tau^t d\check{x}
\end{equation*}
and
\begin{equation*}
{}_{\tau}B^{\vee}:=-d\check{x}^t \underline{B}^{\vee} \tau^t d\check{x},
\end{equation*}
respectively. Let us put
\begin{equation*}
{}_{\tau}\tilde{\omega}^{\vee}:=\tilde{\omega}^{\vee}+\mathbf{i}{}_{\tau}\omega^{\vee}+{}_{\tau}B^{\vee}=\mathbf{i} \bigl( \omega^{\vee}+{}_{\tau}\omega^{\vee} \bigr)+\bigl( B^{\vee}+{}_{\tau}B^{\vee} \bigr).
\end{equation*}
Hereafter, we denote this complexified symplectic torus (the mirror partner of ${}_{\tau}X^n$ which is obtained by using the mirror transform (\ref{mirror})) by
\begin{equation*}
{}_{\tau}\check{X}^n:=\Bigl( \mathbb{R}^{2n}/\mathbb{Z}^{2n}, \ {}_{\tau}\tilde{\omega}^{\vee}=d\check{x}^t (-(T^{-1})^t) d\check{y}-d\check{x}^t (-(T^{-1})^t) \tau^t d\check{x} \Bigr).
\end{equation*}

Finally, we comment on objects which are mirror dual to ${}_{\tau}E_{(A,p,q)}^{\nabla}$. Although we do not assume that Lagrangian submanifolds are affine in \cite{gerby}, here, we focus on affine Lagrangian submanifolds only. Let us consider the Lagrangian submanifold
\begin{equation*}
{}_{\tau}L_{(A,p)}:=\left\{ \left( \begin{array}{cc} \lbrack \check{x} \rbrack \\ \lbrack \left( A+\tau^t \right) \check{x}+p \rbrack \end{array} \right)\in {}_{\tau}\check{X}^n \right\}
\end{equation*}
in ${}_{\tau}\check{X}^n$ as a deformation of each $L_{(A,p)}$, and actually, ${}_{\tau}\omega^{\vee}$ vanishes on this ${}_{\tau}L_{(A,p)}$ since $AT\in \mathrm{Sym}(n;\mathbb{C})$:
\begin{equation*}
\Bigl. {}_{\tau}\omega^{\vee} \Bigr|_{{}_{\tau}L_{(A,p)}}=d\check{x}^t \omega^{\vee} \left( A+\tau^t \right) d\check{x}-d\check{x}^t \omega^{\vee} \tau^t d\check{x}=0.
\end{equation*}
On the other hand, in \cite{gerby}, we conclude that each $\mathcal{O}_{(A,p,q)}^{\nabla}$ is preserved under the deformation from $\check{X}^n$ to ${}_{\tau}\check{X}^n$:
\begin{equation*}
\Bigl. 2\pi \mathbf{i} {}_{\tau}B^{\vee} \Bigr|_{{}_{\tau}L_{(A,p)}}=2\pi \mathbf{i} d\check{x}^t B^{\vee} \left( A+\tau^t \right) d\check{x}-2\pi \mathbf{i} d\check{x}^t B^{\vee} \tau^t d\check{x}=0,
\end{equation*}
where $0$ in the right hand side is interpreted as the curvature form of the flat connection $\nabla_{(A,p,q)}^{\vee}$. Moreover, note that the deformation in the complex geometry side can be interpreted as the tensor product by ${}_{\tau}E_{(O,0,0)}^{\nabla}$. 

This construction is motivated by the $SL(2;\mathbb{Z})$-action on $D^b(Coh(X^1))$ by a generator
\begin{equation*}
\left( \begin{array}{cc} 1 & 0 \\ d & 1 \end{array} \right)\in SL(2;\mathbb{Z})
\end{equation*}
which corresponds to the tensor product by the holomorphic line bundle $E_{(d,0,0)}^{\nabla}$ of degree $d\in \mathbb{Z}$, and the integrity of $\tau$ comes from the integrity of this action. From the viewpoint of the homological mirror symmetry, we can interpret it as the $SL(2;\mathbb{Z})$-action on $\check{X}^1$ which is defined by
\begin{equation*}
\check{x} \ \longmapsto \ \left( \begin{array}{cc} 1 & 0 \\ d & 1 \end{array} \right) \check{x}
\end{equation*}
similarly as in the case of $g_{FM}\in SL(2;\mathbb{Z})$ (see section 5). In particular, the slope $a\in \mathbb{Z}$ of a given Lagrangian submanifold $L_{(a,p)}$ in $\check{X}^1$ is shifted to $a+d\in \mathbb{Z}$ associated to this $SL(2;\mathbb{Z})$ action.

However, concerning the general case $\tau \in M(n;\mathbb{R})$, perhaps, the description in the symplectic geometry side given in \cite{gerby} is not optimal. In general, $\check{X}^n$ is equipped with a foliation structure whose leaves are the fibers when we regard $\check{X}^n$ as the trivial torus fibration $\check{X}^n \to \mathbb{R}^n/\mathbb{Z}^n$ ; $([\check{x}], [\check{y}]) \mapsto [\check{x}]$. Focusing on this point of view, for general $\tau \in M(n;\mathbb{R})$, it seems to be more natural to interpret the deformation in the symplectic geometry side as follows: not only the symplectic structure but also the foliation structure is modified in the sense of \cite{kajiura}, and each object $\mathcal{L}_{(A,p,q)}^{\nabla}$ is also deformed to something associated to such a modification.

\section*{Acknowledgment}
I would like to thank Hiroshige Kajiura for discussions related to noncommutative deformations. This work is supported by JSPS KAKENHI Grant Number 25K17251.


\begin{thebibliography}{99}
\bibitem{abouzaid}
M. Abouzaid, I. Smith, \textit{Homological mirror symmetry for the four-torus}, Duke Mathematical Journal, 152.3 (2010), 373-440.
\bibitem{A-P}
D. Arinkin, A. Polishchuk, \textit{Fukaya category and Fourier transform}, AMS IP STUDIES IN ADVANCED MATHEMATICS, 2001, 23: 261-274.
\bibitem{part1}
O. Ben-Bassat, \textit{Mirror symmetry and generalized complex manifolds\textup{:} Part I. The transform on vector bundles, spinors, and branes}, Journal of Geometry and Physics, 56(4), 533-558, 2006.
\bibitem{part2}
O. Ben-Bassat, \textit{Mirror symmetry and generalized complex manifolds\textup{:} Part I\hspace{-.1em}I. Integrability and the transform for torus bundles}, Journal of Geometry and Physics, 56(7), 1096-1115, 2006.
\bibitem{nc-fm}
O. Ben-Bassat, J. Block, T. Pantev, \textit{Noncommutative tori and Fourier-Mukai duality}, Compositio Mathematica 143.2 (2007): 423-475.
\bibitem{block-1}
J. Block, \textit{Duality and equivalence of module categories in noncommutative geometry I}, A celebration of the mathematical legacy of Raoul Bott, CRM Proc. Lecture Notes, vol. 50, Amer. Math. Soc., Providence, RI, 2010, pp.311-339. MR 2648899 1.
\bibitem{block-2}
J. Block, \textit{Duality and equivalence of module categories in noncommutative geometry I\hspace{-.1em}I\textup{:} Mukai duality for holomorphic noncommutative tori}, arXiv: math.QA/0604296.
\bibitem{bondal}
A. Bondal, M. Kapranov, \textit{Enhanced triangulated categories}, Math. USSR Sbornik 70:93-107, 1991.
\bibitem{steven}
S.B. Bradlow, \textit{Special metrics and stability for holomorphic bundles with global sections}, Journal of Differential Geometry, 33(1), 169-213.
\bibitem{cal}
A.H. C\u{a}ld\u{a}raru, \textit{Derived categories of twisted sheaves on Calabi-Yau manifolds}, Cornell University, 2000.
\bibitem{con-rie}
A. Connes, M. Rieffel, \textit{Yang-Mills for non-commutative two-tori}, Contemp. Math., 62:237-265, 1987.
\bibitem{chat}
D.S. Chatterjee, \textit{On gerbes}, A dissertation submitted towards the degree of Doctor of Philosophy at the. University of Oxford, Trinity College (1998).
\bibitem{Fukaya category}
K. Fukaya, \textit{Morse homotopy, $A^{\infty }$-category, and Floer homologies}, In: Proceedings of GARC Workshop on Geometry and Topology '93 (Seoul, 1993). Lecture Notes in Series, vol. 18, pp. 1-102. Seoul Nat. Univ., Seoul (1993).
\bibitem{Fuk}
K. Fukaya, \textit{Mirror symmetry of abelian varieties and multi theta functions}, J. Alg. Geom. 11, 393-512 (2002).
\bibitem{giraud}
J. Giraud, \textit{Cohomologie non-ab\'{e}lienne}, Grundlehren 179, Springer (Berlin 1971).
\bibitem{Gual}
M. Gualtieri, \textit{Generalized complex geometry}, PhD thesis, Oxford University, 2003, arXiv: math.DG/0401221.
\bibitem{h}
N. Hitchin, \textit{Lectures on special Lagrangian submanifolds}, Winter school on mirror symmetry, vector bundles and Lagrangian submanifolds (Cambridge, MA, 1999) (C. Vafa and S.-T. Yau, eds.), AMS/IP Stud. Adv. Math. vol. 23, Amer. Math. Soc., Providence, RI, 2001, pp. 151-182.
\bibitem{hitchin}
N. Hitchin, \textit{Generalized Calabi-Yau manifolds}, Quarterly Journal of Mathematics, 2003, 54.3: 281-308.
\bibitem{hms-nc-two-tori}
H. Kajiura, \textit{Homological mirror symmetry on noncommutative two-tori}, arXiv: hep-th/0406233.
\bibitem{star}
H. Kajiura, \textit{Star product formula of theta functions}, Letters in Mathematical Physics 75.3 (2006): 279-292.
\bibitem{nc}
H. Kajiura, \textit{Categories of holomorphic line bundles on higher dimensional noncommutative complex tori}, Journal of mathematical physics 48.5 (2007).
\bibitem{Kaj}
H. Kajiura, \textit{Noncommutative tori and mirror symmetry}, In New development of Operator Algebras, volume 1587 of RIMS K\^{o}Ky\^{u}roku, pages 27–72. RIMS, Kyoto Univ., 2008.
\bibitem{kajiura}
H. Kajiura, \textit{On some deformation of Fukaya categories}, in: Symplectic, Poisson, and Noncommutative Geometry, in: MSRI Publ., vol. 62, Cambridge Univ. Press, New York, 2014, pp. 93-130.
\bibitem{kap1}
A. Kapustin, D.O. Orlov, \textit{Vertex algebras, mirror symmetry, and D-branes\textup{:} the case of complex tori}, Communications in mathematical physics 233. 1 (2003): 79-136.
\bibitem{kap2}
A. Kapustin, D.O. Orlov, \textit{Lectures on mirror symmetry, derived categories, and D-branes}, Russian Mathematical Surveys 59(5) (2004): 907.
\bibitem{k-theory}
M. Karoubi, \textit{Twisted bundles and twisted K-theory}, Topics in noncommutative geometry. Vol. 16. Amer. Math. Soc. Providence, RI, 2012. 223-257.
\bibitem{kim}
E. Kim, H. Kim, \textit{Mirror duality and noncommutative tori}, Journal of Physics A: Mathematical and Theoretical 42.1 (2008): 015206.
\bibitem{kazushi}
K. Kobayashi, \textit{Remarks on the homological mirror symmetry for tori}, Journal of Geometry and Physics 164 (2021): 104190.
\bibitem{bijection}
K. Kobayashi, \textit{The bijectivity of mirror functors on tori}, Kyoto Journal of Mathematics 62(3) (2022) 655-682.
\bibitem{gerby}
K. Kobayashi, \textit{A gerby deformation of complex tori and the homological mirror symmetry}, Kodai Mathematical Journal 46(3) (2023) 291-323.
\bibitem{b-field}
K. Kobayashi, \textit{On a B-field transform of generalized complex structures over complex tori}, Journal of Geometry and Physics 206 (2024): 105336.
\bibitem{beta}
K. Kobayashi, \textit{Deformations of Poisson structures on complexified symplectic tori and the homological mirror symmetry}, In preparation.
\bibitem{dg-vect}
S. Kobayashi, \textit{Differential geometry of complex vector bundles}, Princeton University Press, 1987.
\bibitem{Kon}
M. Kontsevich, \textit{Homological algebra of mirror symmetry}, In Proceedings of the International Congress of Mathematicians, Vol. 1, 2 (Z\"{u}rich, 1994), Birkh\"{a}user, 1995, pages 120-139.
\bibitem{dg}
M. Kontsevich, Y. Soibelman, \textit{Homological mirror symmetry and torus fibrations}, In Symplectic geometry and mirror symmetry (Seoul, 2000), pages 203-263. World Sci.Publishing, River Edge, NJ, 2001. 
\bibitem{leung}
N.C. Leung, S.T. Yau, E. Zaslow, \textit{From special Lagrangian to Hermitian-Yang-Mills via Fourier-Mukai transform}, Adv. Theor. Math. Phys. 4: 13191341, 2000.
\bibitem{lieb}
M. Lieblich, \textit{Moduli of twisted sheaves}, Duke Math. J. 138 (2007), no. 1, 23-118.
\bibitem{proj-flat}
Y. Matsushima, \textit{Heisenberg groups and holomorphic vector bundles over a complex torus}, Nagoya Math. J. Vol. 61 (1976), 161-195.
\bibitem{semi-hom}
S. Mukai, \textit{Semi-homogeneous vector bundles on an abelian variety}, J. Math. Kyoto Univ. 18 (1978), no. 2, 239-272.
\bibitem{fm}
S. Mukai, \textit{Duality between $D(X)$ and $D(\hat{X})$ with its application to Picard sheaves}, Nagoya Mathematical Journal 81 (1981): 153-175.
\bibitem{k-htwisted}
A. Perego, \textit{Kobayashi-Hitchin correspondence for twisted vector bundles}, Complex Manifolds, 8(1) (2021) 1-95.
\bibitem{okuda}
N. Okuda, \textit{Fourier-Mukai transforms for non-commutative complex tori}, arXiv: math.AG/2301.03745.
\bibitem{orlov}
D.O. Orlov, \textit{Remarks on generators and dimensions of triangulated categories}, Moscow Mathematical Journal 9.1 (2009): 143-149.
\bibitem{p-s-nc}
A. Polishchuk, A. Schwarz, \textit{Categories of holomorphic vector bundles on noncommutative two-tori}, Communications in mathematical physics, 236(1), 135-159.
\bibitem{elliptic}
A. Polishchuk, E. Zaslow, \textit{Categorical mirror symmetry\textup{:} the elliptic curve}, Adv. Theor. Math. Phys. 2, 443-470 (1998).
\bibitem{SYZ}
A. Strominger, S.T. Yau, E. Zaslow, \textit{Mirror Symmetry is T-duality}, Nucl. Phys. B, 479: 243-259, 1996.
\end{thebibliography}
\end{document}